\documentclass[aos,preprint]{imsart}
\usepackage[latin9]{inputenc}
\usepackage{prettyref}
\usepackage{float}
\usepackage{booktabs}
\usepackage{mathrsfs}
\usepackage{enumitem}
\usepackage{amsmath}
\usepackage{amsthm}
\usepackage{amssymb}
\usepackage{esint}
\usepackage[numbers]{natbib}

\makeatletter

\providecommand{\tabularnewline}{\\}

\theoremstyle{plain}
\newtheorem{assumption}{\protect\assumptionname}
\theoremstyle{remark}
\newtheorem{rem}{\protect\remarkname}[section]
\theoremstyle{plain}
\newtheorem{thm}{\protect\theoremname}[section]
\theoremstyle{definition}
\newtheorem{defn}{\protect\definitionname}[section]
\theoremstyle{plain}
\newtheorem{lem}{\protect\lemmaname}[section]
\theoremstyle{plain}
\newtheorem{prop}{\protect\propositionname}[section]

\usepackage{amsthm}
\usepackage{color}

\usepackage{alphalph}

\setlist[enumerate,1]{label=\upshape{(\roman*)}, ref=(\roman*)}
\setlist[enumerate,2]{label=\upshape{(\alph*)}, ref=(\alph*)}
\setlist[enumerate,3]{label=\upshape{\roman*.}, ref=\roman*}

\newtheoremstyle{named}{\medskipamount}{\medskipamount}{\itshape}{\parindent}{\scshape}{.}{1em}{#1 \thmnote{#3}}
\theoremstyle{named}
\newtheorem*{namedlemma}{Lemma}

\let\@old@begintheorem=\@begintheorem
\def\@begintheorem#1#2[#3]{%
  \gdef\@thm@name{#3}%
  \@old@begintheorem{#1}{#2}[#3]%
}
\def\namedthmlabel#1{\begingroup
   \edef\@currentlabel{\@thm@name}%
   \label{#1}\endgroup
}

\endlocaldefs

\usepackage{comment}
\usepackage{tabto}

\makeatletter
\newcommand \Dotfill {\leavevmode \cleaders \hb@xt@ .88em{\hss .\hss }\hfill \kern \z@}
\makeatother

\allowdisplaybreaks

\makeatother

\providecommand{\assumptionname}{Assumption}
\providecommand{\definitionname}{Definition}
\providecommand{\lemmaname}{Lemma}
\providecommand{\propositionname}{Proposition}
\providecommand{\remarkname}{Remark}
\providecommand{\theoremname}{Theorem}

\begin{document}
\begin{frontmatter}


\title{Estimation and Inference in the Presence of Fractional $\lowercase{d}=1/2$
and Weakly Nonstationary Processes}

\maketitle
\runtitle{Fractional and Weakly Nonstationary Processes}

\author{\fnms{James A.} \snm{Duffy}\corref{}\ead[label=e1]{james.duffy@economics.ox.ac.uk}}

\address{Department of Economics\\
University of Oxford\\
Oxford OX1 3UQ\\
United Kingdom\\  \printead{e1}} \and

\author{\fnms{Ioannis} \snm{Kasparis}\ead[label=e2]{kasparis@ucy.ac.cy}}

\address{Department of Economics\\University of Cyprus\\P.O. Box 20537\\CY-1878,  Nicosia CYPRUS\\ \printead{e2}}

\runauthor{Duffy and Kasparis}
\begin{abstract}
We provide new limit theory for functionals of a general class of
processes lying at the boundary between stationarity and nonstationarity
-- what we term weakly nonstationary processes (WNPs). This includes,
as leading examples, fractional processes with $d=1/2$, and arrays
of autoregressive processes with roots drifting slowly towards unity.
We first apply the theory to study inference in parametric and nonparametric
regression models involving WNPs as covariates. We then use these
results to develop a new specification test for parametric regression
models. By construction, our specification test statistic has a $\chi^{2}$
limiting distribution regardless of the form and extent of persistence
of the regressor, implying that a practitioner can validly perform
the test using a fixed critical value, while remaining agnostic about
the mechanism generating the regressor. Simulation exercises confirm
that the test controls size across a wide range of data generating
processes, and outperforms a comparable test due to Wang and Phillips
(2012, Ann.\ Stat.) against many alternatives.
\end{abstract}


\end{frontmatter}


\global\long\def\argmin{\operatorname*{argmin}}%

\section{Introduction}

Inference in regression models when data is temporally dependent is
a challenging problem, which has engendered a voluminous literature.
Previous work has investigated the asymptotics of parametric and nonparametric
regression estimators under a variety of assumptions on the form and
extent of that dependence, including, for example: regressors generated
by autoregressive fractionally integrated moving average (ARFIMA)
models, general linear and nonlinear processes, and by partial sums
and arrays formed from such processes. Both stationary and nonstationary
processes have been considered, and quite distinct arguments -- relying
on stationary laws of large numbers (LLNs) and central limit theorems
(CLTs) for the former, and the weak convergence of stochastic processes
for the latter -- have been utilised to handle these cases.

It might therefore seem as though little work remains to be done on
this problem. However, in its treatment of nonstationary processes,
previous work has typically employed assumptions that prevent these
from being wholly contiguous with their stationary counterparts, leaving
some significant gaps in the domain of the existing theory. Thus,
for example, when regressors are generated by arrays of ARIMA processes,
both the cases of fixed stationary autoregressive roots and of an
autoregressive root drifting towards unity at rate $n^{-1}$ (a `nearly
integrated' process, henceforth `NI'; see Chan and Wei, 1987, 1988;
Phillips, 1987a,b) have been closely studied, but the intermediate
case of roots drifting towards unity at a strictly slower rate than
$n^{-1}$ (a `mildly integrated' process, henceforth `MI'; see
Giraitis and Phillips, 2006; Magdalinos and Phillips, 2007) has received
less attention until very recently. Similarly, the asymptotics of
regression estimators when applied to fractionally integrated regressors
of order $d$ (henceforth, $I(d)$) are well understood both when
$d\in(-1/2,1/2)$ or $d\in(1/2,3/2)$, but not nearly so well when
$d=1/2$. In view of the importance of ARFIMA models to the statistical
modelling of time series, it is striking that the limit theory for
these processes still remains to be fully characterised.\footnote{Although recent work by Shimotsu and Phillips (2005) and Hualde and
Robinson (2011) allow for cases where $d=1/2$, the limit theory developed
by these authors is specific to functionals in the frequency domain
that arise in the context of memory estimation.}

Existing work on the asymptotics of regression estimators may thus
be divided into two literatures: that dealing with stationary processes,
and that with strongly dependent nonstationary processes, with a certain
space left in between them. Though these literatures are too vast
to be cited exhaustively here, some particularly notable contributions
include the following. For stationary long memory processes, the asymptotics
of nonparametric regression estimators were developed by Wu and Mielniczuk
(2002) and Wu, Huang and Huang (2010). For nonstationary processes,
parametric regression estimators have been studied by Chan and Wei
(1987, 1988), Phillips (1987a,b), Phillips (1995), Park and Phillips
(1999, 2001) and Chan and Wang (2015). Robinson and Hualde (2003),
Christensen and Nielsen (2006), Hualde and Robinson (2010) and Johansen
and Nielsen (2012a) consider fractional systems. Several papers consider
the problem of inference in regressions with a NI covariate, including
Mikusheva (2007), Phillips and Magdalinos (2009) and Kostakis, Magdalinos
and Stamatogiannis (2015). Wang and Phillips (2009a,b, 2011, 2012)
consider nonparametric methods for estimation and inference in regressions
with a NI or nonstationary fractional covariate; some closely related
work on nonparametric estimation in the setting of null recurrent
Markov chains is the subject of the papers by Karlsen and Tjøstheim
(2001) and Karlsen, Myklebust and Tjøstheim (2007).

The present work aims to fill the gap between these two literatures,
by developing the asymptotics of regression estimators for a class
of processes intermediate between the stationary and more strongly
dependent nonstationary processes previously considered. We term these
\emph{weakly nonstationary processes} (WNPs), with leading examples
being MI and $I(1/2)$ processes. To appreciate the significance of
our results, consider the regression model
\begin{equation}
y_{t}=m(x_{t-1})+u_{t}\label{eq:npreg}
\end{equation}
where $u_{t}$ is a martingale difference sequence; suppose for concreteness
that $x_{t}$ is $I(d)$ for an unknown $d\in(-1/2,3/2)$. The object
of interest is the regression function $m(\cdot)$; e.g.\ we would
like to test such a null hypotheses as $\mathcal{H}_{0}:m(x)=m_{0}(x)$
for a given $x$ and $m_{0}$. Since $\mathcal{H}_{0}$ places no
restriction on the process followed by $x_{t}$, establishing the
asymptotic validity of any test of $\mathcal{H}_{0}$ requires that
its asymptotics be developed under all possible values of the nuisance
parameter $d$ (and indeed, under appropriate drifting sequences $\{d_{n}\}$:
see e.g.\ Mikusheva, 2007; Andrews, Cheng and Guggenberger, 2020).\footnote{Formally, the asymptotic size of a test $\phi_{n}\in\{0,1\}$ would
be defined as $\limsup_{n\rightarrow\infty}\sup_{d\in(-1/2,3/2),m\in\mathscr{M}}\mathbb{P}_{d,m}\{\phi_{n}=1\}$,
for $\mathscr{M}$ a class of functions respecting $\mathcal{H}_{0}$.} This requires results for the case where $d=1/2$, no less than for
$d\in(-1/2,1/2)$ and $d\in(1/2,3/2)$. Empirically, values of $d$
in the vicinity of $1/2$ have been systematically found to provide
a good description of the dynamics of inflation and realised volatility
series (see e.g.\ Hassler \& Wolter, 1995; Baillie, Chung and Tieslau,
1996; and Andersen, Bollerslev, Diebold and Labys, 2001). Indeed,
recent work by Hassler and Pohle (2019) finds that when forecasting
such series, an ARFIMA($p$,$d$,$0$) model with $d$ fixed at $1/2$
gives superior forecasts to those produced by the same model with
an estimated value of $d$. As such, our results should be particularly
relevant for inference in regressions involving such series as r.h.s.\ variables.

In the context of (\ref{eq:npreg}), we show that nonparametric kernel
estimators of $m$ are asymptotically (mixed) Gaussian when $x_{t}$
is a WNP; consequently, the $t$ statistic for testing $\mathcal{H}_{0}$
is asymptotically standard normal. This accords with previous results
for both stationary and strongly dependent nonstationary processes
(e.g.\ NI processes or $I(d)$ process with $d>1/2$), which establish
the asymptotic normality of the $t$ statistic in these cases\textbf{
}(see Wu and Mielniczuk, 2002; and Wang and Phillips, 2009a,b, 2011).
It follows that conventional tests of $\mathcal{H}_{0}$, involving
the comparison of the $t$ statistic to normal critical values, are
asymptotically valid even when the regressor has an unknown, but possibly
high, degree of persistence. This is of particular importance for
practitioners, since it implies that tests of $\mathcal{H}_{0}$ can
be conducted without having to in any way adjust for the persistence
of $x_{t}$.

We also consider the case where $m$ is parametrised as $m(x)=\mu+\gamma g(x)$
for a known function $g$. When $x_{t}$ is a WNP, least squares estimators
of $(\mu,\gamma)$ are shown to exhibit the elevated rates of convergence
familiar from when regressors are more strongly dependent, but with
limiting distributions that are (mixed) Gaussian, similarly to the
case of stationary regressors. This result is less directly useful
to practitioners, since it breaks down when $x_{t}$ is more strongly
dependent, in which case the limiting distributions of these estimators
are well known to be nonstandard (see e.g.\ Phillips, 1995; Marinucci
and Robinson, 1998; Robinson and Hualde, 2003).

We build on these results to develop a test for parametric specifications
of $m$ the form $m(x)=\mu+\gamma g(x)$. The test is based on a comparison
of the fit provided by parametric and nonparametric estimates of $m$,
and is designed so as to inherit the asymptotic (mixed) normality
of the nonparametric estimator. It therefore has the attractive property
that the limiting distribution of the test statistic is invariant
to the persistence of $x_{t}$, implying that a practitioner can perform
the test using a fixed set of critical values, while remaining agnostic
about the dependence properties of the regressor. Simulations confirm
that the size of the test is successfully controlled in finite samples,
as the process generating the regressor varies between weakly dependent
and stationary, weakly nonstationary, and strongly dependent nonstationary.
Relative to the specification test proposed by Wang and Phillips (2012),
our test appears to have greater power against a broad range of alternatives,
with these power improvements being especially pronounced for integrable
and asymptotically vanishing alternatives.

Underpinning our limit theory for regression estimators and specification
tests are a collection of new technical results concerning the asymptotics
of additive functionals of WNPs, of the form
\begin{equation}
\frac{1}{n}\sum_{t=1}^{n}f(\beta_{n}^{-1}x_{t})\qquad\text{and}\qquad\frac{\beta_{n}}{nh_{n}}\sum_{t=1}^{n}K\left(\frac{x_{t}-x}{h_{n}}\right)\label{functionals}
\end{equation}
where $\beta_{n}^{2}=Var(x_{n})$, $f$ is locally integrable, $K$
is integrable, and $h_{n}$ denotes a bandwidth sequence. These new
results are needed because WNPs are characterised by being: (a) sufficiently
nonstationary to resist the application of existing LLNs; and (b)
so weakly dependent that the finite-dimensional distributions of the
process $r\mapsto\beta_{n}^{-1}x_{\lfloor nr\rfloor}$ converge to
those of a nonseparable Gaussian process.\footnote{By the \emph{nonseparability} of a process $G$, we mean there does
\emph{not} exist a countable $T\subset[0,1]$ such that for every
open interval $I\subset[0,1]$, $\inf_{r\in I\cap T}G(r)=\inf_{r\in I}G(r)$
and $\sup_{r\in I\cap T}G(r)=\sup_{r\in I}G(r)$ (see Loève, 1978,
p.\ 171).} Since this convergence cannot be strengthened to weak convergence
with respect to the uniform or Skorokhod topologies, the asymptotics
of (\ref{functionals}) are not amenable to an application of the
continuous mapping theorem, or even to more general results on the
convergence of integral functionals (see Gikhman and Skorokhod, 1969,
p.\ 485, Thm 1). Despite this, we shall prove that for WNPs
\begin{equation}
\frac{1}{n}\sum_{t=1}^{n}f(\beta_{n}^{-1}x_{t})\overset{d}{\rightarrow}\int_{\mathbb{R}}f(x+X^{-})\varphi_{\sigma_{+}^{2}}(x)dx\label{eq:WNPlim1}
\end{equation}
where $X^{-}\sim N[0,\sigma_{-}^{2}]$ (with possibly $\sigma_{-}^{2}=0$)
and $\varphi_{\sigma_{+}^{2}}$ denotes the $N[0,\sigma_{+}^{2}]$
density, and
\begin{equation}
\frac{\beta_{n}}{nh_{n}}\sum_{t=1}^{n}K\left(\frac{x_{t}-x}{h_{n}}\right)\overset{d}{\rightarrow}\varphi_{\sigma_{+}^{2}}(-X^{-})\int_{\mathbb{R}}K(u)du.\label{eq:WNPlim2}
\end{equation}

The remainder of this paper is organised as follows. To help further
motivate this work, two leading examples of WNPs, $I(1/2)$ and MI
processes, are discussed in detail in Section~\ref{process-intro}.
Our general results on the asymptotics of the functionals in (\ref{functionals})
are presented in Section~\ref{sec:functionals}. These provide the
basis for the asymptotics of both parametric and nonparametric estimators
in regression models involving WNPs -- or non-linear transformations
thereof -- developed in Section~\ref{applications}. These results
are in turn used, in Section~\ref{Test}, to propose a specification
test whose limiting distribution is invariant to the persistence of
the regressor. Its finite-sample performance is evaluated through
the simulation exercises presented in Section~\ref{sec:Simulations}.
Proofs of all results are given in the Appendices (in the Supplementary
Material).

\paragraph{Notation}

$\overset{a.s.}{\rightarrow}$, $\overset{p}{\rightarrow}$ and $\overset{d}{\rightarrow}$
respectively denote convergence almost surely, in probability, and
in distribution. For deterministic sequences $\{a_{n}\}$ and $\{b_{n}\}$,
$a_{n}\sim b_{n}$ denotes $\lim_{n\rightarrow\infty}a_{n}/b_{n}=1$
and $a_{n}\asymp b_{n}$ denotes $\lim_{n\rightarrow\infty}|a_{n}/b_{n}|\in(0,\infty)$.
For a random variable $X$, $X\sim F$ denotes that $X$ has distribution
$F$. For a positive real number $x$, $\lfloor x\rfloor$ denotes
its integer part. $\mathbf{1}\{A\}$ denotes the indicator function
for the set $A$. $\overline{\mathbb{R}}$, $\mathbb{R}_{+}$ and
$\mathbb{R}_{+}^{\ast}$ are the extended, the nonnegative, and (strictly)
positive real numbers respectively. $f^{(j)}(x)$ denotes the $j$th
derivative of the function $f(x)$. All limits are taken as $n\rightarrow\infty$
unless otherwise indicated.

\section{Leading examples of WNPs: $I(1/2)$ and MI processes}

\label{process-intro}

In the next section, we provide general limit theorems for the additive
functionals in (\ref{functionals}) under high-level conditions, which
may be regarded as defining the class of weakly nonstationary processes
(WNPs); these results will then be specialised to $I(1/2)$ and MI
processes. Before doing so, we give a precise definition of these
two processes, which helps to motivate our high-level conditions.
To bring these processes into a common framework, consider a linear
process array of the form 
\begin{equation}
x_{t}(n)=\sum_{j=0}^{t-1}\phi_{j}(n)v_{t-j},\qquad\text{where}\qquad v_{t}=\sum_{i=0}^{\infty}c_{i}\xi_{t-i},\label{model01}
\end{equation}
$t=1,...,n\in\mathbb{N}$, and the coefficients $\phi_{j}(n)$ and
$c_{i}$ will be specified below. (Where there is no possibility of
ambiguity, we shall generally denote $x_{t}(n)$ as simply $x_{t}$,
for ease of notation.) $\{\xi_{t}\}_{t\in\mathbb{Z}}$ satisfies

\setcounter{assumption}{6461}
\begin{assumption}
~
\begin{enumerate}
\item $\xi_{t}$ is i.i.d.\ with $\mathbf{E}\xi_{1}=0$ and $Var(\xi_{1})=\sigma_{\xi}^{2}<\infty$.
\item $\xi_{1}$ has an absolutely continuous distribution, and a characteristic
function $\psi_{\xi}(\lambda)$ that satisfies $\int_{\mathbb{R}}|\psi_{\xi}(\lambda)|^{\theta}d\lambda<\infty$,
for some $\theta\in\mathbb{N}$.
\end{enumerate}
\end{assumption}

\subsection{$I(1/2)$ processes}

The definition of a `fractionally integrated process' used this paper
closely follows that of Marinucci and Robinson (1999). These authors
classify a non-stationary fractional process $x_{t}$ as type I or
type II according to the `type' of the fractional Brownian motion
(fBM) to which the finite dimensional distributions of $\beta_{n}^{-1}x_{\lfloor nr\rfloor}$
converge, where $\beta_{n}^{2}:=Var(x_{n})$. Although these authors
consider processes with a long memory component that is specified
`parametrically' -- via the expansion of an autoregressive lag polynomial
$(1-L)^{d}$ (see also Remark~\ref{rem:FR}(c)\textbf{ }below) --
their classification extends straightforwardly to the case where this
is instead formulated `semi-parametrically', in terms of the decay
rate of the coefficients $\{\phi_{j}\}$ in (\ref{model01}). Thus
we shall say that for $d\in(1/2,1)$, $x_{t}$ is an $I(d)$ process
of
\begin{itemize}
\item type I: if $\phi_{j}=1$, $c_{s}\sim\ell(s)s^{d-2}$, $\sum_{s=0}^{\infty}c_{s}=0$;
and
\item type II: if $\phi_{j}\sim\ell(j)j^{d-1}$, $\sum_{s=0}^{\infty}\lvert c_{s}\rvert<\infty$,
and $\sum_{s=0}^{\infty}c_{s}\neq0$;
\end{itemize}
\noindent where $\ell:\mathbb{R}_{+}\rightarrow\mathbb{R}_{+}$ is
locally integrable on $[1,\infty)$, positive-valued and slowly varying
at infinity (henceforth, `SV') in the sense of Bingham, Goldie and
Teugels (1987, p.\ 6). Fractional processes of this kind have been
widely studied when $d>1/2$: see e.g.\ Taqqu (1975), Kasahara and
Maejima (1988) and Jeganathan (2004, 2008) for the type I case, and
Robinson and Hualde (2003), Phillips and Shimotsu (2004), Shimotsu
and Phillips (2005) and Hualde and Robinson (2011) for the type II
case.

The preceding extend naturally to $d=1/2$, and give the definitions
of $I(1/2)$ processes (of each type) used throughout this paper,
even though $\beta_{n}^{-1}x_{\lfloor nr\rfloor}$ will not converge
weakly to an fBM of either type in the case. We shall accordingly
develop our limit theory for these processes under

\setcounter{assumption}{173}
\begin{assumption}
$x_{t}(n)$ is generated by (\ref{model01}). $\ell$ is SV such that
$L(n):=\int_{1}^{n}[\ell^{2}(x)/x]dx\rightarrow\infty$ as $n\rightarrow\infty$,
and either:
\begin{enumerate}[label=\upshape{\textbf{FR\arabic*}}]
\item $\phi_{j}=1\ \forall j\geq0$, $c_{s}\sim\ell(s)s^{-3/2}$, and $\sum_{s=0}^{\infty}c_{s}=0$;
or
\item $\phi_{j}\sim\ell(j)j^{-1/2}$, $\phi_{0}\neq0$, $\sum_{s=0}^{\infty}|c_{s}|<\infty$,
$\sum_{s=0}^{\infty}c_{s}\neq0$.
\end{enumerate}
\end{assumption}
\begin{rem}
\label{rem:FR}\textbf{(a)} \textbf{FR1} and \textbf{FR2} respectively
imply that $x_{t}(n)$ is an $I(1/2)$ process of types I and II.

\textbf{(b)} $Var(x_{n})\asymp L(n)$ under \textbf{FR}, where $L(n)$
is itself SV. The divergence or convergence of $L(n)$ effectively
demarcates the boundary between WNPs and stationary long memory processes.
Either is possible, depending on $\ell$: e.g.\ $\ell(n)=1$ gives
$L(n)\sim\ln n$ and $\ell(n)=(\ln n)^{-1/2}$ gives $L(n)\asymp\ln\ln n$,
whereas $\ell(n)=(\ln n)^{-1}$ gives a bounded $L(n)$. When $L(n)$
diverges, $L(n)^{-1/2}x_{n}$ will obey a CLT. But when $L(n)$ is
bounded, no CLT applies: and indeed in this case we have under \textbf{FR2}
that $\sum_{j=0}^{\infty}\phi_{j}^{2}<\infty$, so that $x_{t}$ is
stationary. Such processes fall within the purview of existing results,
and so have been excluded by our assumption that $L(n)\rightarrow\infty$.

\textbf{(c)} \textbf{FR }encompasses both types of parametric ARFIMA
models with $d=1/2$. For example, consider the ARFIMA($1/2$) type
II model 
\begin{equation}
(1-L)^{1/2}x_{t}=v_{t}1\left\{ t>0\right\} ,\qquad\text{where }a(L)v_{t}=b(L)\xi_{t},\label{Parametric1}
\end{equation}
where $a$ and $b$ denote finite-order polynomials in the lag operator
$L$. In this case, it is possible to write $x_{t}=\sum_{j=0}^{t-1}\phi_{j}v_{t-j}$,
where $\{\phi_{j}\}_{j\geq0}$ are the coefficients in the power series
expansion of $(1-L)^{-1/2}$; and so $\phi_{0}=1$ and $\phi_{j}\asymp j^{-1/2}$
(see e.g.\ p.\ 673 in Johansen and Nielsen, 2012b). Further, if
all the roots of $a$ lie outside the unit circle, $v_{t}=a(L)^{-1}b(L)\xi_{t}$
is a linear process with geometrically decaying coefficients. Thus
\textbf{FR2} is satisfied. Since $\ell(x)=1$, we have trivially that
$L(n)\rightarrow\infty$. Similar arguments show that \textbf{FR1}
is consistent with an ARFIMA($1/2$) type I model, under which $x_{t}$
is the partial sum of a ARFIMA($-1/2$) type II process.
\end{rem}

\subsection{MI processes}

Mildly integrated (MI) processes are closely related to the nearly
integrated (NI) processes studied by Chan and Wei (1987) and Phillips
(1987), and more recently extended by Buchmann and Chan (2007). Both
MI and NI processes may be defined in terms of an array as 
\begin{equation}
x_{t}(n)=(1-\kappa_{n}^{-1})x_{t-1}(n)+v_{t},\label{MI}
\end{equation}
where $x_{0}(n)=0$ and $v_{t}$ is a stationary process and $\kappa_{n}>1$
with $\kappa_{n}\rightarrow\infty$, so that the autoregressive coefficient
approaches unity as $n$ grows. They can thus be encompassed within
the framework of (\ref{model01}) if we allow $\phi_{j}$ to depend
on $n$ as per 
\begin{equation}
\phi_{j}=\phi_{j}(n)=(1-\kappa_{n}^{-1})^{j}.\label{eq:micoef}
\end{equation}
Both NI and MI processes thus describe highly persistent autoregressive
processes, which have a root in the vicinity of unity. They have accordingly
been used to investigate the behaviour of various inferential procedures
under local departures from unit roots (e.g. Mikusheva, 2007, and
Duffy, 2020), and in the construction of robust inferential procedures
(e.g. Magdalinos and Phillips, 2011; Kostakis, Magdalinos and Stamatogiannis,
2015; Demetrescu et al 2019; Yang, Long, Peng and Cai, 2019). The
crucial difference between NI and MI processes concerns the assumed
growth rate of the sequence $\kappa_{n}$. NI processes are defined
by $\kappa_{n}/n\rightarrow c\neq0$, with the consequence that $n^{-1/2}x_{\lfloor nr\rfloor}$
converges weakly to an Ornstein-Uhlenbeck process. MI processes have
$\kappa_{n}/n\rightarrow0$, which tilts $x_{t}(n)$ closer to stationarity:
and as a consequence, FCLTs cannot be used to derive the asymptotics
of functionals of these processes. Their variance also grows at a
slower rate, since $Var(x_{n}(n))\asymp\kappa_{n}=o(n)$.

Formally, we define an MI process as follows, specifying regularity
conditions that are helpful in unifying notation and simplifying some
derivations.\footnote{Previous work on these processes has assumed that $v_{t}$ is short
memory in the sense that $\sum_{s=0}^{\infty}\lvert c_{s}\rvert<\infty$.
In a previous working paper version of the present work (available
as arXiv:1812.07944v1), we allowed $v_{t}$ to have long memory in
the sense that $c_{s}\sim s^{-m}$ for $m\in(1/2,1)$, thereby extending
this previous work much in the manner of Buchmann and Chan's (2007)
extension of earlier work on NI processes. To keep the length of the
paper manageable, this generalisation is not reported here.}

\setcounter{assumption}{346}
\begin{assumption}
$x_{t}\left(n\right)$ and $\phi_{j}(n)$ are as in (\ref{model01})
and (\ref{eq:micoef}). $\{c_{s}\}_{s\in\mathbb{Z}}$ is such that
$\sum_{s=0}^{\infty}\left\vert c_{s}\right\vert <\infty$ and $\sum_{s=0}^{\infty}c_{s}\neq0$.
$\{\kappa_{n}\}_{n\in\mathbb{N}}$ has $\kappa_{n}>1$, $\kappa_{n}=n^{\alpha_{\kappa}}\ell_{\kappa}(n)$
for $\ell_{\kappa}$ SV and $\alpha_{\kappa}\in[0,1)$, $\kappa_{n}\rightarrow\infty$
and $\sup_{n\geq1}\sup_{1\leq t\leq n}\kappa_{n}^{-1}\kappa_{t}<\infty$.
\end{assumption}

\section{Limit theory for functionals of WNPs}

\label{sec:functionals}

\subsection{Additive functionals of standardised processes}

\label{additive-funcs}

Consider
\begin{equation}
\frac{1}{n}\sum_{t=1}^{n}f(\beta_{n}^{-1}x_{t})\label{target-fnal}
\end{equation}
where $f$ is locally integrable, and $\beta_{n}^{2}:=Var(x_{n}(n))$.
Here we provide high-level conditions (\textbf{Assumption HL}) under
which the asymptotics of (\ref{target-fnal}) may be derived. These
conditions, particularly \textbf{HL0--2} and \textbf{HL4} below,
may be taken as providing an abstract definition of a WNP -- with
\textbf{HL3}, \textbf{HL5} and \textbf{HL6} being merely regularity
conditions that may be dispensed with, if $f$ satisfies certain assumptions.
These high-level conditions are stated in terms of a general random
array denoted $\{X_{t}(n)\}$, to distinguish it from the linear process
array $\{x_{t}(n)\}$ introduced in (\ref{model01}) above.

\setcounter{assumption}{219}
\begin{assumption}[high-level conditions]
~
\begin{enumerate}[label=\upshape{\textbf{HL\arabic*}},start=0]
\item  Let $\left\{ X_{t}\left(n\right)\right\} _{t=1}^{n}$, $n\in\mathbb{N}$\ be
a random array and $\left\{ \mathcal{F}_{t}\right\} _{t=-\infty}^{\infty}$\ a
filtration such that $X_{t}\left(n\right)$\ is $\mathcal{F}_{t}$-measurable
for all $t$\ and $n$. Let $\{\beta_{n}\}$ denote a positive sequence
with $\beta_{n}\rightarrow\infty$.
\item $X_{t}(n)=X_{t}^{+}(n)+X_{t}^{-}(n)+R_{t}(n)$, where $X_{t}(n)^{-}$\ is
$\mathcal{F}_{0}$-measurable, and $\sup_{1\leq t\leq n}\mathbf{P}\{\beta_{n}^{-1}\lvert R_{t}(n)\rvert>\epsilon\}\rightarrow0$\ for
every $\epsilon>0$.
\item There are random variables $X^{+}$ and $X^{-}$, where $X^{+}$ has
bounded continuous density $\Phi_{X^{+}}$\ such that: for every
$\delta\in(0,1)$ and $\left\{ t_{n}\right\} $\ with $\lfloor n\delta\rfloor\leq t_{n}\leq n$
\begin{enumerate}
\item $\beta_{n}^{-1}X_{t_{n}}^{+}(n)\overset{d}{\rightarrow}X^{+}$, conditionally
on $\mathcal{F}_{0}$\ in the sense that for all bounded and continuous
$h:\mathbb{R}\rightarrow\mathbb{R}$
\[
\mathbf{E}\left[h\left(\beta_{n}^{-1}X_{t_{n}}^{+}(n)\right)\mid\mathcal{F}_{0}\right]\overset{p}{\rightarrow}\int_{\mathbb{R}}h(x)\Phi_{X^{+}}(x)dx\text{\textit{;} \textit{and}}
\]
\item $\beta_{n}^{-1}X_{t_{n}}^{-}(n)\overset{d}{\rightarrow}X^{-}$,\ and
$\beta_{n}^{-1}[X_{n}^{-}(n)-X_{t_{n}}^{-}(n)]\overset{p}{\rightarrow}0$.
\end{enumerate}
\item $\beta_{t}^{-1}X_{t}(n)$ has density $\mathcal{D}_{n,t}(x)$ such
that for some $n_{0}\geq t_{0}\geq1$, 
\[
\sup_{n\geq n_{0},t_{0}\leq t\leq n}\sup_{x}\mathcal{D}_{n,t}(x)<\infty
\]
\item For every bounded and Lipschitz continuous $g:\mathbb{R}\rightarrow\mathbb{R}$
\[
\frac{1}{n}\sum_{t=1}^{n}g\left(\beta_{n}^{-1}X_{t}(n)\right)=\frac{1}{n}\sum_{t=1}^{n}\mathbf{E}\left[g\left(\beta_{n}^{-1}X_{t}(n)\right)\mid\mathcal{F}_{0}\right]+o_{p}(1).
\]
\item For some $\lambda\in\left(0,\infty\right)$ and $n_{0}\geq1$
\begin{enumerate}
\item $\sup_{n\geq n_{0},1\leq t\leq n}\mathbf{E}\left\vert \beta_{n}^{-1}X_{t}\left(n\right)\right\vert ^{\lambda}<\infty$;
or
\item $\sup_{n\geq n_{0},1\leq t\leq n}\mathbf{E}\exp\left(\lambda\left\vert \beta_{n}^{-1}X_{t}\left(n\right)\right\vert \right)<\infty$.
\end{enumerate}
\item $\sup_{n\in\mathbb{N}}\frac{\beta_{n}}{n}\sum_{t=t_{0}}^{n}\beta_{t}^{-1}<\infty$,
for $t_{0}$ as in \textbf{HL3}.
\end{enumerate}
\end{assumption}
\begin{rem}
\label{rem:HL}\textbf{(a)} When $X_{t}(n)$ is a linear process array
formed from an underlying i.i.d.\ sequence $\{\xi_{t}\}$ as in (\ref{model01}),
\textbf{HL1} is trivially satisfied by splitting it into terms depending
on $\{\xi_{s}\}_{s\leq0}$ and $\{\xi_{s}\}_{s=1}^{t}$.

\textbf{(b)} \textbf{HL2} expresses one of the key properties of a
WNP: that its finite dimensional distributions should converge (upon
standardisation), albeit not to those of a separable process. Its
requirements may be illustrated by an $I(1/2)$ type I process with
$\ell(x)=1$, denoted $\{x_{t}\}$. Per the previous remark, write
$x_{t}=x_{t}^{+}+x_{t}^{-}$, where $x_{t}^{+}$ and $x_{t}^{-}$
are respectively weighted sums of $\{\xi_{s}\}_{s=1}^{t}$ and $\{\xi_{s}\}_{s\leq0}$.
Then a CLT for weighted sums of linear processes (see Abadir, Distaso,
Giraitis, and Koul, 2014) yields that for every $r,s\in(0,1]$
\begin{equation}
\beta_{n}^{-1}(x_{\lfloor nr\rfloor}^{+},x_{\lfloor ns\rfloor}^{+})\overset{d}{\rightarrow}(\eta_{r},\eta_{s})\label{t1-curr-cvg}
\end{equation}
where $\beta_{n}^{2}\asymp\ln n$, and $\eta_{r}$ and $\eta_{s}$
are independent $N[0,1/2]$ random variables. We thus have the marginal
convergence of each coordinate of $\beta_{n}^{-1}x_{\lfloor nr\rfloor}^{+}$
to identical distributional limits, as per \textbf{HL2(a)} -- and
if $\xi_{t}$ is i.i.d.,\ the required conditional convergence holds
trivially. In that case, (\ref{t1-curr-cvg}) holds jointly with (and
independently of) 
\begin{equation}
\beta_{n}^{-1}(x_{\lfloor nr\rfloor}^{-},x_{\lfloor ns\rfloor}^{-})\overset{d}{\rightarrow}(\eta^{-},\eta^{-})\label{t1-past-cvg}
\end{equation}
where $\eta^{-}\sim N[0,1/2]$. Note the degeneracy in the joint distribution
of the limit in (\ref{t1-past-cvg}), consistent with the second part
of \textbf{HL2(b)}.

\textbf{(c)} \textbf{HL3} is useful for establishing $L_{1}$-approximations
to functionals of WNPs, which permit convergence results proved under
the requirement that $f$ in (\ref{target-fnal}) be bounded and continuous
to be extended to a much broader class of integrable functions. High-level
conditions similar to \textbf{HL3} have been employed for similar
purposes in many previous works, e.g.\ Jeganathan (2004, 2008), Pötscher
(2004), Gao, King, Lu and Tjøstheim (2009), Wang and Phillips (2009a,b;
2012) among others.

\textbf{(d)} \textbf{HL4} expresses the requirement that a WNP should
not be too strongly dependent. It would fail both for $I(d)$ processes
with $d>1/2$, and for NI processes -- and indeed for any process
for which $\beta_{n}^{-1}X_{\lfloor nr\rfloor}(n)$ converges weakly
to a process with continuous sample paths.
\end{rem}
\begin{thm}
\label{MainThm} Suppose $f:\mathbb{R}\rightarrow\mathbb{R}$\ is\ locally
integrable, and \textbf{HL0}--\textbf{4} and \textbf{HL6} hold. Further
suppose $f$ is bounded or that:
\begin{enumerate}
\item There is a $\mathcal{Y\subseteq}\mathbb{R}$ such that $\int_{\mathbb{R}}\left\vert f\left(x+y\right)\right\vert \Phi_{X^{+}}(x)dx<\infty$
for all $y\in\mathcal{Y}$, and $\mathbf{P}\left(X^{-}\in\mathcal{Y}\right)=1$;
\item $n^{-1}\sum_{t=1}^{t_{0}-1}f(\beta_{n}^{-1}X_{t}(n))=o_{p}\left(1\right)$,
for $t_{0}$\ as in \textbf{HL3};
\item For some $\lambda^{\prime}\in\left(0,\lambda\right)$, where $\lambda$
is as in \textbf{HL5}, either:
\begin{enumerate}
\item $\left\vert f\left(x\right)\right\vert =O(\left\vert x\right\vert ^{\lambda^{\prime}})$,
as $\left\vert x\right\vert \rightarrow\infty$\ and \textbf{HL5(a)}\ holds;
or
\item $\left\vert f\left(x\right)\right\vert =O\left(\exp\left(\lambda^{\prime}\left\vert x\right\vert \right)\right)$,
as $\left\vert x\right\vert \rightarrow\infty$ and \textbf{HL5(b)}\ holds. 
\end{enumerate}
\end{enumerate}
\noindent Then as $n\rightarrow\infty$,
\begin{equation}
\frac{1}{n}\sum_{t=1}^{n}f\left(\beta_{n}^{-1}X_{t}\left(n\right)\right)\overset{d}{\rightarrow}\int_{\mathbb{R}}f\left(x+X^{-}\right)\Phi_{X^{+}}(x)dx.\label{Thm2.1}
\end{equation}
\end{thm}
\begin{rem}
\textbf{(a)} If $f$ is bounded, conditions (i)--(iii) hold trivially,
and \textbf{HL5} is unnecessary for (\ref{Thm2.1}). If $f$ is additionally
Lipschitz, then \textbf{HL3} and \textbf{HL6} may also be dispensed
with.

\textbf{(b)} Condition (ii) of Theorem \ref{MainThm} is a technical
requirement that has also been employed in other studies that develop
limit theory for functionals of nonstationary processes, e.g.\ Jeganathan
(2004) and Pötscher (2004). This condition is redundant if \textbf{HL3}
holds with $t_{0}=1$.

\textbf{(c) }Since $X_{t}(n)$ is continuously distributed under \textbf{Assumption
INN}, (\ref{Thm2.1}) continues to hold if $f$ is modified on a set
of Lebesgue measure zero. Thus e.g.\ if $f$ has an integrable pole
at some $x_{0}\in\mathbb{R}$, and otherwise satisfies the requirements
of Theorem~\ref{MainThm}, then (\ref{Thm2.1}) holds regardless
of how $f$ is defined at $x_{0}$.
\end{rem}

For $I(1/2)$ and MI processes, it may be shown that \textbf{Assumption
INN} and each of \textbf{FR} and \textbf{MI} are sufficient for \textbf{Assumption
HL}. Theorem~\ref{MainThm} therefore specialises as follows. Recall
that $\varphi_{\sigma^{2}}(x)$\ denotes the $N[0,\sigma^{2}]$ density,
and take $\beta_{n}^{2}=Var(x_{n}(n))$. For $X^{-}\sim N[0,1/2]$,
let
\begin{equation}
\varrho(x):=\begin{cases}
\varphi_{1/2}(x-X^{-}) & \text{under \textbf{FR1}}\\
\varphi_{1}(x) & \text{under \textbf{FR2}, \textbf{MI},}
\end{cases}\label{eq:integrals}
\end{equation}
noting that this is a density, albeit a random one under \textbf{FR1}.
\begin{thm}
\label{MainLL} Suppose $f:\mathbb{R}\rightarrow\mathbb{R}$\ is\ locally\ Lebesgue
integrable, and \textbf{Assumption INN} and either \textbf{FR }or
\textbf{MI} hold. Further suppose:
\begin{enumerate}[label={\upshape\arabic*.}]
\item Either: (a) for each $t^{\prime}\in\mathbb{N}\ $ fixed,$\ n^{-1}\sum_{t=1}^{t^{\prime}}f\left(\beta_{n}^{-1}x_{t}\left(n\right)\right)=o_{p}\left(1\right)$;
or (b) \textbf{Assumption INN} holds with $\theta=1$.
\item For some $\lambda^{\prime}\in\left(0,\infty\right)$, as $\left\vert x\right\vert \rightarrow\infty$\ either
\begin{enumerate}
\item $\lvert f(x)\rvert=O(\lvert x\rvert^{\lambda^{\prime}})$, and $\mathbf{E}\left\vert \xi_{1}\right\vert ^{2\vee\lambda}<\infty$
for some $\lambda>\lambda^{\prime}$; or
\item $\lvert f(x)\rvert=O\left(\exp\left(\lambda^{\prime}\left\vert x\right\vert \right)\right)$,
and $\xi_{1}$ has a finite m.g.f.\ in a neighbourhood of zero.
\end{enumerate}
\end{enumerate}
\noindent Then
\begin{equation}
\frac{1}{n}\sum_{t=1}^{n}f(\beta_{n}^{-1}x_{t}(n))\overset{d}{\rightarrow}\int_{\mathbb{R}}f(x)\varrho(x)dx.\label{Corollary2.4}
\end{equation}
\end{thm}
\begin{rem}
\textbf{(a)} Theorem \ref{MainLL} bridges existing asymptotic results
for $I(d)$ processes of order $d\in(-1/2,1/2)$ and $d\in(1/2,3/2)$.
There are similarities and differences between (\ref{Corollary2.4})
and the limit theory that applies in these two cases. Firstly, there
is some analogy with the LLN results that hold when $d\in(-1/2,1/2)$.
As in that case, the limit in (\ref{Corollary2.4}) is determined
by an expectation, but with the difference that the expectation in
(\ref{Corollary2.4}) is with respect to a limiting distribution,
rather than the invariant distribution of a strictly stationary process.
Secondly, the limiting density ($\varphi_{1}$ or $\varphi_{1/2}$)
is obtained by an application of a CLT, and in this respect the limit
theory is analogous to that for $d\in(1/2,3/2)$, which involves FCLTs.
In that case, the weak limits of additive functionals are stochastic,
being a functional of a limiting fBM. This nondegeneracy of the limit
carries over to the $I(1/2)$ type I process (\textbf{FR1}), as evinced
by the dependence of $\varrho$ on $X^{-}\sim N[0,1/2]$ in this case.
(The type II process can be regarded as a truncated type I process,
and the additional variability of the latter appears to make its behaviour
closer to that of an $I(d)$ processes with $d>1/2$).

\textbf{(b) }Theorem\ \ref{MainLL} generalises the limit theory
of Giraitis and Phillips (2006) and Phillips and Magdalinos (2007)
for MI processes to general nonlinear functionals. Those two papers
consider quadratic functions (i.e. $f(x)=x^{2}$) of MI processes
driven by short memory linear processes errors. Using a direct approach
they show that $(n\beta_{n}^{2})^{-1}\sum_{t=1}^{n}x_{t}\left(n\right)^{2}$
$\overset{p}{\rightarrow}1$. This result can be understood as a special
case of Theorem \ref{MainLL}; indeed in this case the r.h.s.\ of
(\ref{Corollary2.4}) is 
\[
\int_{\mathbb{R}}f(x)\varphi_{1}(x)dx=\int_{\mathbb{R}}x^{2}\varphi_{1}(x)dx=1.
\]
Theorem\ \ref{MainLL} also generalises a result due to Tanaka (1999,
p.\ 555), who shows that $\frac{1}{n\ln n}\sum_{t=1}^{n}x_{t}^{2}=O_{p}(1)$
for an ARFIMA($1/2$) processes, in the course of deriving the asymptotics
of the maximum likelihood estimator for $d$ over a parameter space
that includes $d=1/2$.

\textbf{(c)} \textbf{Assumption INN} entails that $\xi_{t}$ has finite
second moment, and so $\beta_{n}^{-1}x_{\lfloor nr\rfloor}(n)$ satisfies
a CLT; the limiting distributions that appear in (\ref{Corollary2.4})
are therefore Gaussian.\footnote{Some preliminary work of the authors' shows that Theorem \ref{MainLL}
can be extended to the case where $\xi_{t}$ is in the domain of attraction
of an $\alpha$-stable law with parameter $\alpha\in(0,2)$, in which
case other stable distributions will appear in the limit. We leave
extensions of this kind for future work.} Note that the process $r\mapsto\beta_{n}^{-1}x_{\lfloor nr\rfloor}(n)$
does not converge weakly (with respect to the uniform or Skorokhod
topologies): as the example of an $I(1/2)$ type I process given in
Remark~\ref{rem:HL}(b) illustrates, this is impossible because the
finite-dimensional distributions of $\beta_{n}^{-1}x_{\lfloor nr\rfloor}(n)$
converge to those of a nonseparable process. (Similar calculations
show that if $x_{t}$ is $I(1/2)$ type II or MI, then the fidis converge
to those of a Gaussian `white noise' process $G$ for which $G(r)\sim N[0,1]$
is independent of $G(s)$ for all $r,s\in[0,1]$.) This accords with
the results of Johansen and Nielsen (2012b), who show that for an
$I(d)$ process with $d\in(1/2,3/2)$ of either type, $\{\xi_{t}\}$
must have moments of order greater than $(d-1/2)^{-1}$ if $\beta_{n}^{-1}x_{\lfloor nr\rfloor}(n)$
is to converge weakly to a fBM, a requirement that becomes progressively
more demanding as $d$ approaches $1/2$.
\end{rem}

\subsection{Kernel functionals}

\label{seckernels}

We next consider kernel functionals of the form 
\begin{equation}
\frac{\beta_{n}}{h_{n}n}\sum_{t=1}^{n}K\left(\frac{x_{t}(n)-x}{h_{n}}\right),\label{eq:kernelfunc}
\end{equation}
where $x\in\mathbb{R}$, $h_{n}$ is a bandwidth sequence, and $K$
is an integrable kernel function satisfying

\setcounter{assumption}{10}
\begin{assumption}[kernel]
 $K:\mathbb{R}\rightarrow\overline{\mathbb{R}}$ is such that $K$
and $K^{2}$ are Lebesgue integrable.
\end{assumption}
Whereas in (\ref{target-fnal}) the nonlinear transformation $f$
is applied to the standardised process $\beta_{n}^{-1}x_{t}(n)$,
in (\ref{eq:kernelfunc}) $K$ is applied to the unstandardised process
$x_{t}(n)$. This leads to a different limit theory, which is partly
reflected in the different normalisations of the sums in (\ref{target-fnal})
and (\ref{eq:kernelfunc}). A notable difference between (\ref{eq:kernelfunc})
and the usual expression for a kernel density estimator is the appearance
of $\beta_{n}$. The more variable that $x_{t}$ is, the less frequent
are its visits to the support of $K[(\cdot-x)/h_{n}]$ -- and since
the variance $\beta_{n}$ of a WNP grows with $n$, these visits accumulate
only at rate $nh_{n}/\beta_{n}$, a fact reflected exactly in the
normalisation in (\ref{eq:kernelfunc}).

Nonetheless, the results of the preceding section turn out to be highly
relevant for the asymptotics of (\ref{eq:kernelfunc}). To explain
why this is the case, we return to the setting of \textbf{Assumption
HL} above, which we now augment by the following additional smoothness
conditions on the density of the increments of $X_{t}(n)$. To state
these, let 
\[
\Omega_{n}\left(\eta\right):=\left\{ \left\{ s,t\right\} \in\mathbb{N}:\lfloor\eta n\rfloor\leq s\leq\lfloor(1-\eta)n\rfloor,\text{ }\lfloor\eta n\rfloor+s\leq t\leq n\right\} ,
\]
for $\eta\in(0,1)$.

\setcounter{assumption}{219}
\begin{assumption}[continued]
~
\begin{enumerate}[label=\upshape{\textbf{HL\arabic*}},start=7]
\item Let $X_{0}(n):=0$ and $t>s\geq0$. Conditionally on $\mathcal{F}_{s}$,
$\beta_{t-s}^{-1}(X_{t}(n)-X_{s}(n))$\ has density $\mathcal{D}$$_{t,s,n}(x)$\ such
that for some $n_{0},t_{0}\geq1$ 
\[
\sup_{n\geq n_{0},0\leq s<t\leq n,t-s\geq t_{0}}\sup_{x}\mathcal{D}_{t,s,n}(x)<\infty.
\]
\item For all $q_{0},q_{1}>0$
\[
\lim_{\eta\downarrow0}\limsup_{n\rightarrow\infty}\sup_{(s,t)\in\Omega_{n}(\eta)}\sup_{\left\vert x\right\vert \leq q_{0}\eta^{q_{1}}}\left\vert \mathcal{D}_{t,s,n}(x)-\mathcal{D}_{t,s,n}(0)\right\vert =0
\]
\item For $t_{0}$\ as in \textbf{HL7}:
\begin{enumerate}
\item $\lim_{\eta\rightarrow0}\lim_{n\rightarrow\infty}\frac{\beta_{n}}{n}\sum_{t=\lfloor(1-\eta)n\rfloor}^{n}\beta_{t}^{-1}=0$;
\item $\lim_{\eta\rightarrow0}\lim_{n\rightarrow\infty}\frac{\beta_{n}}{n}\sum_{t=t_{0}}^{\lfloor\eta n\rfloor}\beta_{t}^{-1}=0$;
\item $\lim\sup_{n\rightarrow\infty}\frac{\beta_{n}}{n}\sum_{t=t_{0}}^{n}\beta_{t}^{-1}<\infty$;
\item there exist $l_{0},l_{1}>0$ such that $\liminf_{n\rightarrow\infty}\beta_{n}^{-1}\inf_{(s,t)\in\Omega_{n}(\eta)}\beta_{t-s}\geq\eta^{l_{1}}/l_{0}$
for all $\eta\in(0,1)$;
\item $\sup_{n\geq1}\sup_{1\leq t\leq n}\beta_{n}^{-1}\beta_{t}<\infty$.
\end{enumerate}
\end{enumerate}
\end{assumption}
\textbf{HL7} and \textbf{HL9} may be regarded as strengthened versions
of \textbf{HL3} and \textbf{HL6}, and are closely related to Assumption~2.3
in Wang and Phillips (2009a). Under these conditions, an $L_{1}$-approximation
argument developed by those authors and Jeganathan (2004) yields that,
for $t_{0}$ as in \textbf{HL7}, 
\[
\frac{\beta_{n}}{h_{n}n}\sum_{t=t_{0}}^{n}K\left(\frac{X_{t}(n)-x}{h_{n}}\right)=\frac{1}{n}\sum_{t=t_{0}}^{n}\varphi_{\varepsilon^{2}}(\beta_{n}^{-1}X_{t}(n))\int_{\mathbb{R}}K(u)du+o_{p}(1),
\]
as $n\rightarrow\infty$ and then $\varepsilon\rightarrow0$. The
leading order term on the r.h.s.\ clearly has the same form as the
l.h.s.\ of (\ref{Thm2.1}) and is thus is amenable to a direct application
of Theorem \ref{MainThm}, which entails
\begin{align}
\frac{1}{n}\sum_{t=1}^{n}\varphi_{\varepsilon^{2}}(\beta_{n}^{-1}X_{t}(n)) & \overset{d}{\rightarrow}\int\varphi_{\varepsilon^{2}}(x+X^{-})\Phi_{X^{+}}(x)dx,\quad\text{as }n\rightarrow\infty\nonumber \\
 & \overset{a.s.}{\rightarrow}\Phi_{X^{+}}(-X^{-}),\quad\text{as }\varepsilon\rightarrow0.\label{kernapprox}
\end{align}
We thus have the following counterpart of Theorem \ref{MainThm} for
kernel functionals.
\begin{thm}
\label{MainKernel} Suppose that, in addition to \textbf{Assumptions
K}, \textbf{HL0--2}, \textbf{HL4} and \textbf{HL6--9}, the following
hold:
\begin{enumerate}
\item $\{h_{n}\}$ is a positive sequence with $\beta_{n}^{-1}h_{n}+\beta_{n}(nh_{n})^{-1}\rightarrow0$;
and
\item for each $x\in\mathbb{R}$ and $t_{0}$ as in \textbf{HL7}, $\frac{\beta_{n}}{nh_{n}}\sum_{t=1}^{t_{0}-1}K\left(\frac{X_{t}(n)-x}{h_{n}}\right)=o_{p}\left(1\right)$.
\end{enumerate}
Then 
\[
\frac{\beta_{n}}{h_{n}n}\sum_{t=1}^{n}K\left(\frac{X_{t}(n)-x}{h_{n}}\right)\overset{d}{\rightarrow}\Phi_{X^{+}}(-X^{-})\int_{\mathbb{R}}K(u)du.
\]
\end{thm}
For $I(1/2)$ and MI processes, the preceding specialises as follows.
\begin{thm}
\label{KernelLL} Suppose that, in addition to \textbf{Assumption
K}:
\begin{enumerate}
\item $x_{t}(n)$ satisfies \textbf{Assumption INN},\textbf{ }and \textbf{FR}
or \textbf{MI};
\item $\{h_{n}\}$ satisfies condition (i) of Theorem \ref{MainKernel};
and
\item for each $x\in\mathbb{R}$ and $t^{\prime}\in\mathbb{N}$, $\frac{\beta_{n}}{nh_{n}}\sum_{t=1}^{t^{\prime}}K\left(\frac{x_{t}(n)-x}{h_{n}}\right)=o_{p}\left(1\right)$.
\end{enumerate}
\noindent Then 
\[
\frac{\beta_{n}}{h_{n}n}\sum_{t=1}^{n}K\left(\frac{x_{t}(n)-x}{h_{n}}\right)\overset{d}{\rightarrow}\varrho(0)\int_{\mathbb{R}}K(u)du.
\]
\end{thm}
\begin{rem}
\noindent \textbf{(a)} Theorem \ref{KernelLL} fills a gap in existing
asymptotic theory for kernel functionals of linear processes. A general
theory for stationary linear processes, including $I(d)$ processes
with $\left|d\right|<1/2$, is given in Wu and Mielniczuk (2002).
Supposing that $\int_{\mathbb{R}}K=1$, under their conditions kernel
functionals converge in probability to the invariant density of the
stationary process. Jeganathan (2004, 2008) provides limit theorems
for kernel functionals of $I(d)$ processes with $1/2<d<3/2$. In
that case, kernel functionals converge to the local time of a fractional
Brownian motion (or fractional stable motion if innovations are in
the domain of attraction of a stable law) -- so their limit is an
occupation density rather than the invariant density of some stationary
process. The limiting behaviour of kernel functionals of $I(1/2)$
processes is intermediate between these two cases. These converge
to the density of a random variable, rather than to an occupation
density, but the density corresponds to a limiting random variate,
rather than the invariant density of a stationary process.

\textbf{(b)} Theorem \ref{KernelLL} nests a similar result provided
by Duffy (2020) for bounded kernel functionals of MI processes, which
unlike \textbf{Assumption K} requires $K$ to be bounded and Lipschitz
continuous.
\end{rem}

\section{Estimation and inference in regressions with WNPs}

\label{applications}

The preceding results are fundamental to the asymptotics of parametric
and nonparametric least squares estimators, in models involving WNPs
as regressors. In this section, we show that these estimators have
either Gaussian or mixed Gaussian limit distributions, and in consequence
their associated $t$ statistics are asymptotically standard Gaussian.
These results are in turn used, in Section~\ref{Test}, to derive
the asymptotic distribution of a proposed regression specification
test statistic, and in particular to show that it is asymptotically
pivotal, being unaffected by the persistence of the regressor process.

\subsection{Parametric regression}

\label{subLS}

Consider the ordinary least squares (OLS) estimator of $(\mu,\gamma)$
in the model
\begin{equation}
y_{t}=\mu+\gamma g(x_{t-1})+u_{t}\label{regres1}
\end{equation}
given by $(\hat{\mu},\hat{\gamma}):=\text{argmin}_{(a,b)}\sum_{t=1}^{n}[y_{t}-a-bg(x_{t-1})]^{2}$,
where $g$ is a known nonlinear transformation. Since the regressor
is predetermined (i.e.\ $\mathcal{F}_{t-1}$-measurable) relative
to the error $u_{t}$, (\ref{regres1}) is an instance of a so-called
`predictive' or `reduced form' regression model. If $x_{t}$ is
stationary, the OLS estimator will be asymptotically normal; whereas
if $x_{t}$ is strongly dependent, the OLS estimator has a non-standard
limiting distribution, unless either $g$ is itself integrable, or
$x_{t}$ and $u_{t}$ satisfy a very restrictive `long-run orthogonality'
condition (see e.g.\ Park and Phillips, 1999, 2001).

When $x_{t}$ is a WNP, the OLS estimator is either asymptotically
normal or mixed normal, depending on the type of process. In either
case, the $t$ statistic is asymptotically $N[0,1]$, due to self-normalisation.
In this respect, the asymptotics are similar to those when $x_{t}$
is stationary; but since the variance of a WNP grows without bound,
the analysis requires arguments more appropriate to nonstationary
processes. In particular, the following property of $g$, first introduced
by Park and Phillips (1999, 2001), plays a key role.

\setcounter{defn}{889}
\begin{defn}[asymptotically homogeneous function]
Let $\{x_{t}(n)\}$ denote a random array and $\beta_{n}^{2}=Var(x_{n}(n))$.
$g:\mathbb{R}\rightarrow\overline{\mathbb{R}}$ is \textit{asymptotically
homogeneous} for $\{x_{t}(n)\}$, if for each $\lambda>0$ it admits
the decomposition 
\[
g(x)=\kappa_{g}(\lambda)H_{g}(x/\lambda)+R_{g}(x,\lambda),
\]
where $\kappa_{g}:\mathbb{R}_{+}^{\ast}\rightarrow\mathbb{R}_{+}^{\ast}$,
$H_{g}:\mathbb{R}\rightarrow\overline{\mathbb{R}}$, $R_{g}:\mathbb{R}\times\mathbb{R}_{+}^{\ast}\rightarrow\overline{\mathbb{R}}$
and for $j=\left\{ 1,2\right\} $, 
\begin{equation}
R_{g,n}^{j}:=\frac{1}{\kappa_{g}^{j}(\beta_{n})n}\sum_{t=1}^{n}\mathbf{E}\left\vert R_{g}(x_{t}(n),\beta_{n})\right\vert ^{j}=o(1).\label{AHF}
\end{equation}
\end{defn}

AHFs encompass a wide range of commonly used regression functions,
such as polynomial functions, cumulative distribution functions (with
$H_{g}(u)=\mathbf{1}\{u>0\}$), and logarithmic functions (with $H_{g}(u)=1$);
see Park and Phillips (1999, 2001) for some further examples. Such
a condition as 
\[
\lim_{\lambda\rightarrow\infty}\kappa_{g}(\lambda)^{-1}\sup_{x}\lvert R_{g}(x,\lambda)\rvert=0
\]
is sufficient, but not necessary, for (\ref{AHF}) to hold. The relevance
of AHFs for the OLS estimator can be seen most easily in the case
where $\mu=0$ is known and imposed, so that the OLS estimator for
$\gamma$ satisfies 
\[
\hat{\gamma}_{n}-\gamma=\frac{\sum_{t=2}^{n}g(x_{t-1})u_{t}}{\sum_{t=2}^{n}g^{2}(x_{t-1})}=(1+o_{p}(1))\frac{\sum_{t=2}^{n}H_{g}(\beta_{n}^{-1}x_{t-1})u_{t}}{\sum_{t=2}^{n}H_{g}^{2}(\beta_{n}^{-1}x_{t-1})}.
\]
Upon standardisation, the denominator on the r.h.s.\ is directly
amenable to an application of Theorem~\ref{MainLL}. Since the numerator
is a sum of martingale differences, it can be handled via an appropriate
martingale CLT (either Hall and Heyde, 1980; or Wang, 2014); here
Theorem~\ref{MainLL} is used to verify the stability condition pertaining
to its conditional variance.

Reasoning along these lines yields our main result on parametric OLS
esitmators. To state it, let $MN[0,\varsigma^{2}]$ denote a mixed
normal distribution with mixing variate $\varsigma^{2}$ (i.e.\ which
has characteristic function $u\mapsto\mathbf{E}e^{-\varsigma^{2}u^{2}/2}$),
and $\sigma(\{\xi_{s},u_{s}\}_{s\leq t})$ denote the $\sigma$-field
generated by $\{\xi_{s},u_{s}\}_{s\leq t}$.
\begin{thm}
\label{LS} Let $\{y_{t}\}_{t=1}^{n}$ be generated by (\ref{regres1}),
$\mathcal{F}_{t}:=\sigma(\{\xi_{s},u_{s}\}_{s\leq t})$, and suppose
that:
\begin{enumerate}
\item $x_{t}(n)$ satisfies (\ref{model01}), \textbf{Assumption INN} and
either \textbf{FR} or \textbf{MI};
\item $\{u_{t},\mathcal{F}_{t}\}_{t\geq1}$\ is a martingale difference
sequence such that $\mathbf{E}[u_{t}^{2}\mid\mathcal{F}_{t-1}]=\sigma_{u}^{2}$
a.s. for some constant $\sigma_{u}^{2}<\infty$;
\item $\sup_{1\leq t\leq n}\mathbf{E}[u_{t}^{2}\mathbf{1}\{\lvert u_{t}\rvert\geq A_{n}\}\mid\mathcal{F}_{t-1}]=o_{p}(1)$
for non-random $A_{n}\rightarrow\infty$;
\item $g(x)$ is AHF for $\{x_{t}(n)\}$, with limit homogeneous component
$H_{g}$ that is not a.e.\ constant, and is such that $H_{g}^{2}$
satisfies the conditions of Theorem~\ref{MainLL}.
\end{enumerate}
\noindent Then
\begin{equation}
n^{1/2}\begin{bmatrix}\hat{\mu}-\mu\\
\kappa_{g}(\beta_{n})(\hat{\gamma}-\gamma)
\end{bmatrix}\overset{d}{\rightarrow}MN\left[0,\ \sigma_{u}^{2}\left(\int\begin{bmatrix}1 & H_{g}\\
H_{g} & H_{g}^{2}
\end{bmatrix}\varrho\right)^{-1}\right].\label{eq:LSlimit}
\end{equation}
\end{thm}
\begin{rem}
\textbf{\label{rem:LS}(a)}\ For $I(1/2)$ type II and \textit{\emph{MI}}
processes $\hat{\gamma}_{n}$ is asymptotically normal; for $I(1/2)$
type I processes it is mixed normal, because then $\varrho(x)=\varphi(x-X^{-})$
is random. In either case, the $t$ statistics for testing hypotheses
about $\mu$ or $\gamma$ will be asymptotically standard normal,
so that inferences may be drawn in the usual manner. This contrasts
with the case where regressors are $I(d)$ for $d>1/2$, e.g.\ see
Phillips (1995), Park and Phillips (1999, 2001), Robinson and Hualde
(2003).

\textbf{(b)} In a linear regression model, i.e.\ $g(x)=x$, we have
$H_{g}(u)=u$ and $\kappa_{g}(\lambda)=\lambda$, and it follows that
the OLS estimator for $\gamma$ has convergence rate $\beta_{n}n^{1/2}$,
which is faster than the $n^{1/2}$-convergence rate that obtains
when the regressor is stationary. For $I(1/2)$ processes the gain
in convergence rate is given by the slowly varying factor $L^{1/2}(n)$
(see Remark \ref{rem:FR}(b) above).

\textbf{(c) }Suppose instead that $y_{t}=\mu+\gamma g(x_{t})+u_{t}$,
so that the regressor is no longer predetermined. In this case, the
asymptotics of the OLS estimator are different from (\ref{eq:LSlimit}):
for example, if it is known that $\mu=0$ and $g(x)=x$, (so that
$\kappa_{g}(\beta_{n})=\beta_{n}$) we have
\[
\beta_{n}^{2}\left(\hat{\gamma}-\gamma\right)\overset{d}{\rightarrow}\left[\int x^{2}\varrho(x)dx\right]^{-1}\lim_{n\rightarrow\infty}\mathbf{E}(x_{t}(n)-x_{t-1}(n))u_{t}.
\]
In this case there is a severe reduction in the convergence rate,
by a factor of $n^{1/2}/\beta_{n}$, due to the endogeneity of $x_{t}$.
This result is comparable to Theorem 5.2 of Marinucci and Robinson
(1998)  which gives the asymptotics of the OLS estimator when $x_{t}$
is $I(d)$ with $d\in(1/2,1)$. We expect that in this setting such
methods as narrowband LS (see e.g. Marinucci and Robinson, 1998; Robinson
and Hualde 2003; Christensen and Nielsen 2006) or those that use lagged
regressors as instruments will be more efficient. We leave the exploration
of alternative estimation procedures for future work.
\end{rem}

\subsection{Nonparametric regression}

\label{subNW}

We next consider the nonparametric estimation of $m$ in the predictive
regression
\begin{equation}
y_{t}=m(x_{t-1})+u_{t}.\label{regres2}
\end{equation}
In particular, we consider the kernel regression (Nadaraya--Watson;
NW) estimator 
\[
\hat{m}(x):=\sum_{t=2}^{n}K_{th}(x)y_{t}/\sum_{t=2}^{n}K_{th}(x),
\]
and the local linear (LL) estimator
\[
\begin{bmatrix}\tilde{m}(x)\\
\tilde{m}^{(1)}(x)
\end{bmatrix}:=\argmin_{(a,b)\in\mathbb{R}^{2}}\sum_{t=2}^{n}\left[y_{t}-a-b(x_{t-1}-x)\right]^{2}K_{th}(x),
\]
where $K_{th}(x):=K[(x_{t-1}-x)/h_{n}]$. The following theorem is
a direct consequence of Theorem~\ref{KernelLL} and certain martingale
central limit theorems, and is complementary to the recent work of
Wang and Phillips (2009a,b; 2012) who develop estimation and testing
procedures in the context of nonparametric regression with NI and
$I(d)$ processes with $d\in(1/2,3/2)$. Let
\[
Q:=\left\{ \int\left[\begin{array}{cc}
1 & x\\
x & x^{2}
\end{array}\right]K\right\} ^{-1}\left\{ \int\left[\begin{array}{cc}
1 & x\\
x & x^{2}
\end{array}\right]K^{2}\right\} \left\{ \int\left[\begin{array}{cc}
1 & x\\
x & x^{2}
\end{array}\right]K\right\} ^{-1},
\]
and $\nu_{K,i}:=\int x^{i}K(x)dx$ for $i\in\mathbb{N}$.
\begin{thm}
\label{thm:npboth}Let $\{y_{t}\}_{t=1}^{n}$ be generated by (\ref{regres2})
and suppose that:
\begin{enumerate}
\item conditions (i)--(iii) of Theorem~\ref{LS} hold, $\nu_{K,0}=1$
and $\nu_{K,2}\neq\nu_{K,1}^{2}$.
\item $x^{j}[K(x)+K^{2}(x)]$ are bounded and integrable for $j\in[0,3]$;
\item $h_{n}+\beta_{n}/nh_{n}\rightarrow0$;
\end{enumerate}
If $m$ has a bounded first derivative and $nh_{n}^{3}/\beta_{n}\rightarrow0$,
then
\begin{equation}
\left(\frac{nh_{n}}{\beta_{n}}\right)^{1/2}\left(\hat{m}(x)-m(x)\right)\overset{d}{\rightarrow}MN\left[0,\sigma_{u}^{2}\varrho(0)^{-1}\int K^{2}\right].\label{eq: NWLimit}
\end{equation}
Alternatively, if $m$ has a bounded second derivative and $nh_{n}^{5}/\beta_{n}\rightarrow0$,
then\textbf{ }
\begin{equation}
\left(\frac{nh_{n}}{\beta_{n}}\right)^{1/2}\left[\begin{array}{c}
\tilde{m}(x)-m(x)\\
h_{n}\left(\tilde{m}^{(1)}(x)-m^{(1)}(x)\right)
\end{array}\right]\stackrel{d}{\rightarrow}MN\left[0,\sigma_{u}^{2}\varrho(0)^{-1}Q\right].\label{eq:loclin}
\end{equation}
\end{thm}
\begin{rem}
\label{rem:kernel}\textbf{(a)} Since $\beta_{n}\rightarrow\infty$,
the convergence rate of both $\hat{m}$ and $\tilde{m}$ when $x_{t}$
is a WNP is slower than when $x_{t}$ is stationary. For $I(1/2)$
processes this convergence rate is reduced by the slowly varying factor
$L(n)^{1/2}$.

\textbf{(b)} Let $\tilde{\sigma}_{u}^{2}$ denote a consistent estimator
of $\sigma_{u}^{2}$, and consider the nonparametric $t$ statistic
for the hypothesis $\mathcal{H}_{0}:m(x)=m_{0}(x)$ based on the local
linear estimator, as given by
\[
\tilde{t}(x;m_{0}):=\left(\frac{\sum_{t=2}^{n}K_{th}(x)}{\tilde{\sigma}_{u}^{2}Q_{11}}\right)^{1/2}[\tilde{m}(x)-m_{0}(x)].
\]
It follows directly from Theorem~\ref{thm:npboth} that $\tilde{t}(x;m_{0})\overset{d}{\rightarrow}N[0,1]$,
and similarly when $(\tilde{m},Q_{11})$ is replaced by $(\hat{m},\int K^{2})$.

Thus in conjunction with the existing literature, Theorem~\ref{thm:npboth}
implies that kernel nonparametric $t$ statistics are asymptotically
standard Gaussian across a wide range of regressor processes, including:
stationary fractional (with $-1/2<d<1/2$; Wu and Mielniczuk, 2002),
weakly nonstationary (fractional with $d=1/2$ or mildly integrated),
nonstationary fractional ($1/2<d<3/2$) and (nearly) integrated processes
(Wang and Phillips, 2009a,b, 2011, 2012). This is in marked contrast
to parametric $t$ statistics, which when regressors are nonstationary
have limiting distributions that are typically nonstandard and dependent
on nuisance parameters relating to the persistence of the regressor,
which cannot be consistently estimated -- a fact that greatly complicates
parametric inference in these models (for an overview of this problem
and the relevant literature, see Phillips and Lee, 2013, pp.\ 251--254).

\textbf{(c)} Suppose that $x_{t-1}$ on the r.h.s.\ of (\ref{regres2})
is replaced by $x_{t}$, so that the regressor is no longer predetermined.
If $x_{t}$ is stationary, then the correlation between it and $u_{t}$
prevents $m$ from being consistently estimated. However, for the
case where $x_{t}$ is NI, Wang and Phillips (2009b) show that the
nonparametric regression estimator is consistent for $m$ and asymptotically
mixed Gaussian, even when $x_{t}$ is correlated with $u_{t}$, and
$u_{t}$ is serially dependent. In other words, for NI covariates
the asymptotics of the nonparametric regression estimator are unaffected
by whether $x_{t}$ or $x_{t-1}$ appears in (\ref{regres2}). We
conjecture that a similar results also holds for WNPs, but leave an
examination of this for future work.

\textbf{(d) }Our smoothness assumptions on $m$ could be relaxed along
the lines of Wang and Phillips (2009a,b) and Wang and Phillips (2011),
for the NW and LL estimators respectively; we have refrained from
doing so here to permit Theorem~\ref{thm:npboth} to be more concisely
stated.
\end{rem}

\section{Specification testing when a regressor has an unknown degree of persistence}

\label{Test}

In this section, we exploit the asymptotic normality of the nonparametric
$t$ statistic to develop a specification test statistic for parametric
regression models that has the same asymptotic distribution regardless
of the extent of the persistence of the regressor, and indeed regardless
of whether that persistence is modelled in terms of long memory (i.e.\ as
$I(d)$ for some $d\in(-1/2,3/2)$) or in terms of an autoregressive
root localised to unity (as in an MI or NI process). The proposed
test can thus be validly conducted, in a straightforward manner, without
requiring practitioners either to make an assumption on the persistence
of the regressor, or to somehow estimate this and take account of
it when carrying out the test.

The hypothesis to be tested is that the true regression function $m$
in (\ref{regres2}) belongs to a certain parametric family, as e.g.\ postulated
in (\ref{regres1}). Formally, the null is
\begin{equation}
\mathcal{H}_{0}:m(x)=\mu+\gamma g(x),\textrm{\textrm{for} some }(\mu,\gamma)\mathbb{\in\mathbb{R}}^{2}\text{ and all }x\in\mathbb{R};\label{eq:nullhyp}
\end{equation}
where $g$ is a known function;\textbf{ }the alternative is that no
such $\mu$ and $\gamma$ exist. Tests of $\mathcal{H}_{0}$, in a
setting with (possibly) nonstationary regressors, have also been considered
by Gao, King, Lu and Tjøstheim (2009), Wang and Phillips (2012; hereafter
`WP'), and Dong, Gao, Tjøstheim and Yin (2017). WP test a parametric
fit in the presence of a NI regressor, while Dong et al (2017) test
for a parametric fit in regressions with a $d=0$ and a $d=1$ covariate.
The test statistic of WP closely resembles that of Gao et al (2009),
who propose a studentised U-statistic formed of kernel-weighted OLS
regression residuals.\footnote{Gao et al (2009) apply their statistic to the problem of testing the
null of a random walk (of the form $x_{t}=x_{t-1}+\xi_{t}$), against
a (possibly nonlinear) stationary alternative. The underlying idea
is to test for a neglected nonlinear component in an autoregression,
whose presence would make the process stationary. The specification
test proposed in this paper could also be potentially used for this
purpose, but we leave explorations in this direction for future work.}

We propose to test $\mathcal{H}_{0}$ by comparing parametric OLS
and kernel nonparametric estimates of $m$. The model specified in
(\ref{eq:nullhyp}) can be estimated parametrically by OLS regression,
and also nonparametrically at each $x$ as
\[
\tilde{m}_{g}(x):=\argmin_{a\in\mathbb{R}}\min_{b\in\mathbb{R}}\sum_{t=1}^{n}\{y_{t}-a-b[g(x_{t-1})-g(x)]\}^{2}K_{th}(x)
\]
which under $\mathcal{H}_{0}$ has no asymptotic bias, even if $h$
remains fixed as $n\rightarrow\infty$.\textbf{ }If $g(x)=x$, so
that the null of linearity is being tested, $\tilde{m}_{g}(x)$ specialises
to the local linear regression estimator; but in general $\tilde{m}_{g}$
should be chosen consistent with the model under test, so that it
has no bias under the null. Provided that $g^{(1)}(x)\neq0$, a slight
modification of the proof of Theorem~\ref{thm:npboth} shows that
$\tilde{m}_{g}(x)$ has the same limiting distribution as displayed
in (\ref{eq:loclin}), under $\mathcal{H}_{0}$. Letting $(\hat{\mu},\hat{\gamma})$
denote the OLS estimates of $(\mu,\gamma)$, we can therefore compare
the fit provided by the parametric and local nonparametric estimates
of the model via an ensemble of $t$ statistics of the form
\[
\tilde{t}(x;\hat{\mu},\hat{\gamma}):=\left[\frac{\sum_{t=2}^{n}K_{th}(x)}{\tilde{\sigma}_{u}^{2}(x)Q_{11}}\right]^{1/2}\left[\tilde{m}_{g}(x)-\hat{\mu}-\hat{\gamma}g(x)\right],
\]
where $\tilde{\sigma}_{u}^{2}(x):=\left[\sum_{t=2}^{n}K_{th}(x)\right]^{-1}\sum_{t=2}^{n}\left[y_{t}-\hat{\mu}-\hat{\gamma}g(x_{t-1})\right]K_{th}(x)$.
Under $\mathcal{H}_{0}$, both the parametric and nonparametric estimators
converge to identical limits and so for each $x\in\mathbb{R}$,
\[
\tilde{t}(x;\hat{\mu},\hat{\gamma})=\tilde{t}(x;\mu,\gamma)+o_{p}(1)\stackrel{d}{\rightarrow}N[0,1],
\]
where the equality is due to the relatively faster convergence rate
of the parametric estimator, and the distributional limit follows
as per Remark~\ref{rem:kernel}(b). Under the alternative, only the
nonparametric estimator is consistent for $m$, and thus $|\tilde{t}|\stackrel{p}{\rightarrow}\infty$.

Our proposed specification test statistic is based on these $t$ statistics
evaluated at a set of $p$ points $\mathcal{X}\subset\mathbb{R}$,
constructed as per
\begin{equation}
\tilde{F}:=\sum_{x\in\mathcal{X}}\tilde{t}(x;\hat{\mu},\hat{\gamma})^{2}.\label{eq:spec_test}
\end{equation}
$\tilde{F}$ is related to the `non-predictability sum test' developed
by Kasparis, Andreou and Phillips (2015), who were concerned with
testing the null that $x_{t-1}$ cannot predict $y_{t}$, which in
the present framework can be expressed as $m(x)=\mu$ for all $x\in\mathbb{R}$.
Relative to other specification tests available in the literature,
the principal advantage of a test based on $\tilde{F}$ is that the
limiting distribution of this statistic is invariant to the extent
of persistence in the regressor, making valid inference in the presence
of data with an unknown degree of persistence straightforward.\footnote{The use of kernel methods in specification testing does not in and
of itself lead to conventional inference: for example, a test recently
proposed by Dong et al (2017) is also based on nonparametric methods,
but the limiting distribution of their test statistic, and therefore
the critical values for their test, depends on precise assumptions
as to the form and extent of the persistence of the regressor.}

\subsection{Asymptotics for WNPs}

Our final result gives the limiting distribution of $\tilde{F}$ under
$\mathcal{H}_{0}$, and under a sequence of local alternatives of
the form
\begin{equation}
\mathcal{H}_{1}:m(x)=\mu+\gamma g(x)+r_{n}g_{1}(x)\label{eq:alternative}
\end{equation}
for some $g_{1}$, where $r_{n}\rightarrow0$. This formulation of
the alternative is similar to that of Horowitz and Spokoiny (2001)
and WP. We only provide explicit results for the boundary case where
$\{x_{t}\}$ is a WNP, which is the main focus of the present work.
However, as discussed in Section~\ref{subsec:asy-other} below, analogous
results may be derived for the stationary fractional ($-1/2<d<1/2$),
nonstationary fractional ($1/2<d<3/2$) and nearly integrated cases,
on the basis of the limit theory presented in Wu and Mielniczuk (2002)
and Wang and Phillips (2009a,b).

For the purposes of the next result, assume $g_{1}(x)$ is either
integrable or an asymptotically homogeneous function (AHF) of asymptotic
order $\kappa_{g_{1}}$; this helps to characterise the limiting behaviour
of the test statistic under $\mathcal{H}_{1}$. Define $\mu_{\ast}:=-[\intop H_{g^{2}}\varrho\intop H_{g_{1}}\varrho-\intop H_{g}\varrho\intop H_{g}H_{_{g_{1}}}\varrho]/[\intop H_{g^{2}}\varrho-(\intop H_{g}\varrho)^{2}]$,
which is the distributional limit of $-[\kappa_{g_{1}}(\beta_{n})r_{n}]^{-1}(\hat{\mu}-\mu)$
when $g_{1}$ is AHF.
\begin{thm}
\label{test} Suppose that $\mathcal{X}$ has $p$ elements and
\begin{enumerate}
\item conditions (i)-(iii) of Theorem~\ref{thm:npboth} hold, $x^{j}[K(x)+K^{2}(x)]$
are bounded and integrable for $j\in[0,4]$, $\sup_{t}\mathbf{E}[u_{t}^{4}\mid\mathcal{F}_{t-1}]<\infty$
a.s., and $nh_{n}^{5}/\beta_{n}\rightarrow0$; and
\item $g$ is AHF for the array $\{x_{t}(n)\}$, with limit homogeneous
component $H_{g}$ such that $\kappa_{g}(\beta_{n})\rightarrow\infty$
and $H_{g}^{2}$ satisfies the requirements of Theorem~\ref{MainLL}.
Further, $g$ has bounded second derivative, and $g^{(1)}(x)\neq0$
for each $x\in\mathcal{X}$.
\end{enumerate}
Then under $\mathcal{H}_{0}$,
\begin{equation}
\tilde{F}\stackrel{d}{\rightarrow}\chi_{p}^{2}\label{eq:null}
\end{equation}

Suppose that in addition $p\geq2$ and:
\begin{enumerate}[resume]
\item $r_{n}\rightarrow0$ and either
\begin{enumerate}
\item $g_{1},$ $g\cdot g_{1}$ are bounded and integrable, $r_{n}^{-1}n^{-1/2}\beta_{n}\rightarrow0$;
or
\item $g_{1}$ is AHF for the array $\{x_{t}(n)\}$, with limit homogeneous
function $H_{g_{1}}$ such that $H_{g_{1}}H_{g}$ satisfies the requirements
of Theorem~\ref{MainLL}, and $r_{n}^{-1}n^{-1/2}\kappa_{g_{1}}^{-1}(\beta_{n})\rightarrow0$.
Further, each of the following limits exist (allowing `convergence'
to $\infty$):
\begin{align*}
\kappa_{\ast} & :=\lim_{n\rightarrow\infty}\kappa_{g_{1}}(\beta_{n}) & \kappa_{*\ast}:= & \lim_{n\rightarrow\infty}\kappa_{g_{1}}(\beta_{n})r_{n}.
\end{align*}
\end{enumerate}
\item $l_{n}nh_{n}/\beta_{n}\rightarrow\infty$ where
\[
l_{n}=\begin{cases}
r_{n}^{2} & \text{under (iii.a) or (iii.b) with }\kappa_{\ast}\in[0,\infty)\\
r_{n}^{2}\kappa_{g_{1}}^{2}\left(\beta_{n}\right) & \text{under (iii.b), }\kappa_{\ast\ast}=0\text{ and }\kappa_{\ast}=\infty\\
1 & \text{under (iii.b) and }\kappa_{\ast\ast}\in(0,\infty].
\end{cases}
\]
\item $g_{1}$ has bounded second derivative.
\item Either: (iii.a) holds and $g_{1}(x)\neq0$ for some $x\in\mathcal{X}$;
(iii.b) holds with $\kappa_{\ast}\in(0,\infty]$ and $g_{1}(x)\neq g_{1}(x^{\prime})$
for some $x,x^{\prime}\in\mathcal{X}$; or (iii.b) holds with $\kappa_{\ast}=\infty$
and $\mu_{\ast}\neq0$ $a.s$.
\end{enumerate}
Then under $\mathcal{H}_{1}$,
\begin{equation}
\tilde{F}\stackrel{p}{\rightarrow}\infty.\label{eq:H1}
\end{equation}
\end{thm}

\begin{rem}
\textbf{(a) }The requirement $\kappa_{g}(\beta_{n})\rightarrow\infty$
is a technical condition that is satisfied in most specifications
employed in empirical work, e.g.\ linear models. It can be relaxed
at the cost of a more involved exposition.

\textbf{(b)} If (iii.b) holds with $\kappa_{\ast}=\infty$, then $\mu_{\ast}\neq0$
a.s.\ is sufficient for the test to be consistent, in the sense that
(\ref{eq:H1}) holds. Consistency may still obtain when $\mu_{\ast}=0$,
but that requires a more detailed analysis than we are able to provide
here.

\textbf{(c)} The sequence $l_{n}nh_{n}/\beta_{n}$ gives the divergence
rate of the test statistic under $\mathcal{H}_{1}$, which is closely
related the power of the test. The maximal divergence rate (i.e.\ $nh_{n}/\beta_{n}$)
is attained when $g_{1}$ is AHF of diverging asymptotic order ($\kappa_{g_{1}}(\beta_{n})\rightarrow\infty$)
and $\kappa_{g_{1}}(\beta_{n})r_{n}\rightarrow\infty$. In such cases,
the divergence rate is otherwise unaffected by $r_{n}$. When $\kappa_{g_{1}}(\beta_{n})\rightarrow\infty$
but $\kappa_{g_{1}}(\beta_{n})r_{n}\rightarrow0$, the divergence
rate reduces to $\kappa_{g_{1}}^{2}(\beta_{n})r_{n}^{2}nh_{n}/\beta_{n}$.
The divergence rate is smallest (i.e.\ $r_{n}^{2}nh_{n}/\beta_{n}$)
in cases where $g_{1}$ is integrable or AHF of vanishing asymptotic
order (i.e. $\kappa_{g_{1}}(\beta_{n})\rightarrow0$).

\textbf{(d)} We have assumed that the set $\mathcal{X}$ comprises
a fixed number ($p$) of points. We can get some idea of the large-sample
distribution of $\tilde{F}$ if $p$ is allowed to grow with the sample
size from the fact that $(2p)^{-1/2}(\tilde{F}-p)\stackrel{d}{\rightarrow}(2p)^{-1/2}(\chi_{p}^{2}-p)\stackrel{d}{\rightarrow}N(0,1)$
as $n\rightarrow\infty$ and then $p\rightarrow\infty$.
\end{rem}

\subsection{Asymptotics in other cases}

\label{subsec:asy-other}

$\tilde{F}$ will have the same limiting distribution as given by
(\ref{eq:null}) of Theorem~\ref{test}, even when $x_{t}$ is a
stationary, nonstationary fractional, or near-integrated process.
This can be established using results available in the existing literature,
by arguments outlined in the remainder of this section, and is confirmed
by the simulation exercises presented in Section~\ref{sec:Simulations}
below. (These also permit our results on the consistency of the test
against local alternatives to be extended beyond weakly nonstationary
processes.) 

From the proof of Theorem~\ref{test}, it is clear that (\ref{eq:null})
holds if
\begin{enumerate}
\item the OLS estimator for $(\mu,\gamma)$ converges faster than the nonparametric
estimator for $m$ (e.g. NW/LL);
\item the nonparametric estimator is asymptotically (mixed) Gaussian; and
\item for every $\mathit{x},x'\in\mathcal{X}$, such that $x\neq x'$,
\begin{equation}
\frac{\beta_{n}}{nh_{n}}\sum_{t=1}^{n}K_{j}[(x_{t}-x)/h_{n}]K_{j'}[(x_{t}-x')/h_{n}]=o_{p}(1),\label{eq:product}
\end{equation}
for all $j,j^{\prime}\in\{0,1,2\}$, where $K_{j}(x):=x^{j}K(x)$.
\end{enumerate}
This last requirement ensures that the component $t$ statistics in
(\ref{eq:spec_test}) are asymptotically independent of each other.

When $x_{t}$ is nonstationary fractional ($1/2<d<3/2$) or nearly
integrated, the arguments of Park and Phillips (2001) (see also Christopeit,
2009 and the references therein) show that the convergence rates of
$(\hat{\mu},\hat{\gamma})$ are as in (\ref{eq:LSlimit}); that the
convergence rates of the kernel nonparametric regression estimators
are slower follows from Wang and Phillips (2009a,b). Conditions~(ii)
and (iii) also follow from Wang and Phillips (2009a, and particularly
pp.~1910--11 of 2009b for (\ref{eq:product})).

When $x_{t}$ is stationary, it is well known that the OLS estimator
is $n^{-1/2}$-consistent while the $(nh_{n})^{-1/2}$-consistency
and asymptotic normality of the nonparametric estimators follows e.g.\ from
the results of Wu and Mielniczuk (2002). Thus conditions (i)--(ii)
hold. Finally, if $x_{t}$ has a bounded density $\mathcal{D}_{x}(u)$
(see Wu and Mielniczuk, 2002, Lemma~1) and under standard regularity
conditions on $K$, we have
\begin{multline*}
\frac{1}{nh_{n}}\sum_{t=1}^{n}\mathbf{E}|K_{j}[(x_{t}-x)/h_{n}]K_{j^{\prime}}[(x_{t}-x')/h_{n}]|\\
\leq\sup_{u}\mathcal{\mathcal{D}}_{x}(u)\int_{\mathbb{R}}\left|K_{j}[z]K_{j'}[z+(x-x')/h_{n}]\right|dz\rightarrow0,
\end{multline*}
so that condition (iii) holds, since $\beta_{n}\asymp1$ in this case.

\section{Simulations}

\label{sec:Simulations}

\begin{table}[t]
\caption{Size: maximum rejection frequency over $\rho\in\{-0.5,0,0.5\}$; $\alpha=0.1$}
\label{tbl:size}%
\begin{tabular}{cccccccccccccc}
\toprule 
\addlinespace
$\boldsymbol{d}$ & $\boldsymbol{n}$ &  & \multicolumn{3}{c}{\textbf{WP (2012)}} &  & \multicolumn{3}{c}{$\boldsymbol{p=17}$} &  & \multicolumn{3}{c}{$\boldsymbol{p=25}$}\tabularnewline\addlinespace
\multicolumn{2}{c}{$h=n^{b}$, $b=$} &  & -0.2 & -0.1 & -0.05 &  & -0.2 & -0.1 & -0.05 &  & -0.2 & -0.1 & -0.05\tabularnewline\addlinespace
\midrule
\addlinespace
\textbf{$0.25$} & 100 &  & 0.04 & 0.02 & 0.01 &  & 0.07 & 0.05 & 0.04 &  & 0.09 & 0.06 & 0.04\tabularnewline
 & 200 &  & 0.05 & 0.02 & 0.01 &  & 0.07 & 0.06 & 0.04 &  & 0.09 & 0.07 & 0.05\tabularnewline
 & 500 &  & 0.06 & 0.02 & 0.01 &  & 0.07 & 0.06 & 0.04 &  & 0.09 & 0.07 & 0.05\tabularnewline\addlinespace
\textbf{$0.50$} & 100 &  & 0.05 & 0.02 & 0.01 &  & 0.07 & 0.06 & 0.04 &  & 0.09 & 0.07 & 0.05\tabularnewline
 & 200 &  & 0.06 & 0.03 & 0.02 &  & 0.08 & 0.07 & 0.05 &  & 0.11 & 0.09 & 0.07\tabularnewline
 & 500 &  & 0.07 & 0.04 & 0.03 &  & 0.07 & 0.07 & 0.05 &  & 0.10 & 0.09 & 0.07\tabularnewline\addlinespace
\textbf{$0.75$} & 100 &  & 0.06 & 0.04 & 0.03 &  & 0.08 & 0.07 & 0.06 &  & 0.11 & 0.09 & 0.07\tabularnewline
 & 200 &  & 0.08 & 0.06 & 0.04 &  & 0.08 & 0.08 & 0.07 &  & 0.10 & 0.10 & 0.09\tabularnewline
 & 500 &  & 0.08 & 0.07 & 0.05 &  & 0.07 & 0.08 & 0.07 &  & 0.09 & 0.09 & 0.09\tabularnewline\addlinespace
$1.00$ & 100 &  & 0.08 & 0.06 & 0.05 &  & 0.09 & 0.09 & 0.08 &  & 0.13 & 0.11 & 0.11\tabularnewline
 & 200 &  & 0.09 & 0.07 & 0.06 &  & 0.08 & 0.10 & 0.10 &  & 0.11 & 0.12 & 0.12\tabularnewline
 & 500 &  & 0.09 & 0.08 & 0.08 &  & 0.09 & 0.09 & 0.09 &  & 0.10 & 0.11 & 0.11\tabularnewline\addlinespace
\bottomrule
\end{tabular}
\end{table}

We conducted simulations to evaluate the finite-sample performance
of the proposed specification test, in terms of size and power against
a range of alternatives. For this exercise, a natural comparison is
with the specification test of WP, which is known to have a standard
Gaussian limiting distribution when $x_{t}$ is NI. (In our simulation
exercises, we assume that this also holds when $x_{t}$ is fractionally
integrated and/or stationary, and so compare their statistic to normal
critical values in these cases.)

\begin{table}[H]
\caption{Size-adjusted power when $d\in\{0.5,1.0\}$; $\alpha=0.1$}
\label{tbl:power}%
\begin{tabular}{lccccccccccccc}
\toprule 
\addlinespace
 & $\boldsymbol{n}$ &  & \multicolumn{3}{c}{\textbf{WP (2012)}} &  & \multicolumn{3}{c}{$\boldsymbol{p=17}$} &  & \multicolumn{3}{c}{$\boldsymbol{p=25}$}\tabularnewline\addlinespace
\multicolumn{2}{l}{$h=n^{b}$, $b=$} &  & -0.2 & -0.1 & -0.05 &  & -0.2 & -0.1 & -0.05 &  & -0.2 & -0.1 & -0.05\tabularnewline\addlinespace
\midrule
\addlinespace
\multicolumn{2}{l}{$\boldsymbol{d=0.5}$} &  & \multicolumn{3}{c}{\textbf{Size adj.\ power}} &  & \multicolumn{7}{c}{\textbf{Size adjusted, relative to WP}}\tabularnewline\addlinespace
$\varphi_{1}(x)$ & 100 &  & 0.08 & 0.06 & 0.04 &  & 0.08 & 0.09 & 0.09 &  & 0.11 & 0.11 & 0.11\tabularnewline
 & 200 &  & 0.14 & 0.13 & 0.11 &  & 0.11 & 0.14 & 0.14 &  & 0.15 & 0.17 & 0.17\tabularnewline
 & 500 &  & 0.37 & 0.42 & 0.41 &  & 0.15 & 0.18 & 0.19 &  & 0.20 & 0.22 & 0.22\tabularnewline\addlinespace
$\varphi_{1}(2x)$ & 100 &  & 0.07 & 0.04 & 0.03 &  & 0.07 & 0.08 & 0.06 &  & 0.10 & 0.09 & 0.08\tabularnewline
 & 200 &  & 0.12 & 0.09 & 0.07 &  & 0.09 & 0.10 & 0.09 &  & 0.13 & 0.13 & 0.12\tabularnewline
 & 500 &  & 0.29 & 0.27 & 0.23 &  & 0.13 & 0.15 & 0.14 &  & 0.19 & 0.19 & 0.18\tabularnewline\addlinespace
$|x|^{-2}\wedge1$ & 100 &  & 0.06 & 0.03 & 0.02 &  & 0.11 & 0.12 & 0.12 &  & 0.15 & 0.15 & 0.14\tabularnewline
$(\times0.5)$ & 200 &  & 0.10 & 0.06 & 0.04 &  & 0.16 & 0.19 & 0.18 &  & 0.21 & 0.23 & 0.22\tabularnewline
 & 500 &  & 0.19 & 0.16 & 0.12 &  & 0.12 & 0.14 & 0.14 &  & 0.16 & 0.16 & 0.16\tabularnewline\addlinespace
$|x|^{-1}\wedge1$ & 100 &  & 0.06 & 0.03 & 0.02 &  & 0.07 & 0.09 & 0.08 &  & 0.10 & 0.11 & 0.10\tabularnewline
$(\times0.5)$ & 200 &  & 0.07 & 0.05 & 0.04 &  & 0.10 & 0.14 & 0.14 &  & 0.15 & 0.17 & 0.17\tabularnewline
 & 500 &  & 0.14 & 0.13 & 0.11 &  & 0.14 & 0.18 & 0.19 &  & 0.19 & 0.21 & 0.22\tabularnewline\addlinespace
$|x|^{1.5}$ & 100 &  & 0.11 & 0.09 & 0.07 &  & 0.03 & 0.04 & 0.04 &  & 0.05 & 0.06 & 0.05\tabularnewline
$(\times0.02)$ & 200 &  & 0.22 & 0.23 & 0.22 &  & 0.04 & 0.05 & 0.06 &  & 0.06 & 0.07 & 0.07\tabularnewline
 & 500 &  & 0.61 & 0.66 & 0.66 &  & 0.04 & 0.09 & 0.11 &  & 0.08 & 0.12 & 0.13\tabularnewline\addlinespace
$x^{2}$ & 100 &  & 0.07 & 0.05 & 0.04 &  & 0.06 & 0.07 & 0.07 &  & 0.09 & 0.09 & 0.09\tabularnewline
$(\times0.02)$ & 200 &  & 0.13 & 0.12 & 0.11 &  & 0.08 & 0.12 & 0.13 &  & 0.12 & 0.15 & 0.16\tabularnewline
 & 500 &  & 0.37 & 0.42 & 0.43 &  & 0.07 & 0.12 & 0.13 &  & 0.12 & 0.15 & 0.15\tabularnewline\addlinespace
\midrule
\addlinespace
\multicolumn{2}{l}{$\boldsymbol{d=1.0}$} &  & \multicolumn{3}{c}{\textbf{Size adj.\ power}} &  & \multicolumn{7}{c}{\textbf{Size adjusted, relative to WP}}\tabularnewline\addlinespace
$\varphi_{1}(x)$ & 100 &  & 0.08 & 0.07 & 0.06 &  & 0.04 & 0.06 & 0.06 &  & 0.08 & 0.09 & 0.09\tabularnewline
 & 200 &  & 0.10 & 0.09 & 0.08 &  & 0.02 & 0.05 & 0.07 &  & 0.06 & 0.09 & 0.10\tabularnewline
 & 500 &  & 0.12 & 0.12 & 0.12 &  & 0.02 & 0.05 & 0.05 &  & 0.05 & 0.07 & 0.09\tabularnewline\addlinespace
$\varphi_{1}(2x)$ & 100 &  & 0.08 & 0.06 & 0.05 &  & 0.02 & 0.04 & 0.05 &  & 0.07 & 0.08 & 0.07\tabularnewline
 & 200 &  & 0.09 & 0.08 & 0.07 &  & 0.01 & 0.03 & 0.04 &  & 0.05 & 0.06 & 0.07\tabularnewline
 & 500 &  & 0.10 & 0.09 & 0.09 &  & 0.02 & 0.02 & 0.03 &  & 0.03 & 0.04 & 0.05\tabularnewline\addlinespace
$|x|^{-2}\wedge1$ & 100 &  & 0.08 & 0.06 & 0.05 &  & 0.05 & 0.08 & 0.08 &  & 0.10 & 0.12 & 0.12\tabularnewline
$(\times0.5)$ & 200 &  & 0.09 & 0.07 & 0.07 &  & 0.04 & 0.08 & 0.10 &  & 0.09 & 0.12 & 0.14\tabularnewline
 & 500 &  & 0.09 & 0.09 & 0.08 &  & 0.02 & 0.06 & 0.08 &  & 0.05 & 0.09 & 0.11\tabularnewline\addlinespace
$|x|^{-1}\wedge1$ & 100 &  & 0.08 & 0.06 & 0.05 &  & 0.04 & 0.07 & 0.07 &  & 0.09 & 0.10 & 0.10\tabularnewline
$(\times0.5)$ & 200 &  & 0.08 & 0.07 & 0.07 &  & 0.03 & 0.07 & 0.08 &  & 0.08 & 0.11 & 0.12\tabularnewline
 & 500 &  & 0.11 & 0.11 & 0.11 &  & 0.02 & 0.06 & 0.08 &  & 0.06 & 0.09 & 0.12\tabularnewline\addlinespace
$|x|^{1.5}$ & 100 &  & 0.10 & 0.09 & 0.08 &  & 0.04 & 0.07 & 0.08 &  & 0.09 & 0.10 & 0.11\tabularnewline
$(\times0.02)$ & 200 &  & 0.13 & 0.13 & 0.13 &  & -0.02 & 0.04 & 0.07 &  & 0.05 & 0.08 & 0.11\tabularnewline
 & 500 &  & 0.19 & 0.21 & 0.22 &  & -0.10 & -0.01 & 0.00 &  & -0.04 & 0.01 & 0.02\tabularnewline\addlinespace
$x^{2}$ & 100 &  & 0.09 & 0.07 & 0.06 &  & 0.04 & 0.06 & 0.08 &  & 0.07 & 0.09 & 0.09\tabularnewline
$(\times0.02)$ & 200 &  & 0.11 & 0.11 & 0.11 &  & 0.00 & 0.01 & 0.01 &  & 0.01 & 0.01 & 0.01\tabularnewline
 & 500 &  & 0.16 & 0.19 & 0.19 &  & 0.00 & 0.00 & 0.00 &  & 0.00 & 0.00 & 0.00\tabularnewline\addlinespace
\bottomrule
\end{tabular}
\end{table}

For all simulation exercises, the null hypothesis is $\mathcal{H}_{0}:m(x)=\mu+\beta x$,
so that the proposed specification test can be implemented by comparing
the fit of a local linear regression with an OLS regression. The data
generating process is
\begin{align*}
y_{t} & =x_{t-1}+g_{1}(x_{t-1})+u_{t} & (1-L)^{d}x_{t} & =\xi_{t}+0.5\xi_{t-1}
\end{align*}
with $x_{0}=0$, where $(u_{t},\xi_{t})$ are i.i.d.\ bivariate Gaussian
with unit variances and correlation $\rho$. For each value of $d\in\{0.25,0.50,0.75,1.00\}$,
we evaluate the size of the test by computing the maximum rejection
frequency under the null (i.e.\ when $g_{1}(x)=0$) for $\rho\in\{-0.5,0.0,+0.5\}$
(with 5000 replications). We consider sample sizes $n\in\{100,200,500\}$,
bandwidths of the form $h=n^{b}$ for $b\in\{-0.2,-0.1,-0.05\}$,
and use the Gaussian kernel in all cases. For our test, which requires
the choice of points at which to compare the nonparametric and parametric
estimates of the regression function, we consider two choices: $p=17$
or $25$ points, evaluated at the quantiles of $\{x_{t}\}_{t=1}^{n}$
equally spaced between the $0.1$ and $0.9$ quantiles.

The results are displayed in Table~\ref{tbl:size}, for a test having
10 per cent nominal significance level. They clearly illustrate that
our test has good size control across the range of bandwidths considered,
both when the data is in the stationary and nonstationary regions,
and on the boundary between these, suggesting that the asymptotics
developed in Sections~\ref{sec:functionals}--\ref{Test} provide
a good approximation to the finite-sample distribution of the statistic.

We also computed the size-adjusted power of our test, and of WP's,
against alternatives that are either integrable\textbf{ }($g_{1}(x)=\varphi_{1}(x)$,
$\varphi_{1}(2x)$, or $x^{-2}\wedge1$), non-integrable but vanishing
at infinity ($g_{1}(x)=x^{-1}\wedge1$), or polynomials ($g_{1}(x)=|x|^{v}$
with $v\in\{1.5,2.0\}$).\footnote{By size-adjusted power, we mean that if the test is found to reject
at rate $\hat{\alpha}>0.1$ under the null (as reported in the top
panel), then the power of the test is adjusted downwards by subtracting
$\hat{\alpha}-0.1$ from the rejection rate under each alternative
(so if $\hat{\alpha}\leq0.1$, no adjustment is made).} The simulation designs are the same as for the size calculations,
except that we here only report results for $\rho=0$ and $d\in\{0.5,1.0\}$
(results for $d\in\{0.25,0.75\}$ and MI processes are broadly similar,
and are provided in Appendix~F in the Supplementary Material). The
alternatives are scaled by the factors indicated in Table~\ref{tbl:power}
so as to ensure non-trivial power for these designs. To facilitate
the comparison between the power of our test and that of WP, we report
the size-adjusted power of their procedure in the first three columns
of Table~\ref{tbl:power}, and the \emph{relative} size-adjusted
power of our test alongside, i.e.\ the difference between the power
of our test and of theirs.

It is noticeable that our test generally outperforms WP's: indeed,
the relevant entries of the table are almost uniformly positive, with
the exception of a few cases where $g_{1}(x)=|x|^{1.5}$ and $d=1$.
The most pronounced power improvements are for those cases where $d$
is smaller, and the alternatives are either integrable or asymptotically
vanishing (i.e. tranformations that exhibit weaker signal and are
therefore harder to detect); whereas the performance of the two tests
is less easily distinguishable for polynomially growing alternatives,
when $d=1$.

\newpage{}
\title{Supplement to: Estimation and Inference in the Presence of Fractional
$\lowercase{d}=1/2$ and Weakly Nonstationary Processes}

\bigskip{}

\noindent \begin{center}
CONTENTS
\par\end{center}

\noindent A\tabto{0.75cm}Proofs under high-level conditions\Dotfill\pageref{sec:proofs}

\noindent \tabto{0.75cm}A.1\tabto{1.75cm}Technical lemmas\Dotfill\pageref{subsec:Technical-lemmas}

\noindent \tabto{0.75cm}A.2\tabto{1.75cm}Proofs of Theorems 3.1
and 3.3\Dotfill\pageref{subsec:Proofs-of-Theorems}

\noindent \tabto{0.75cm}A.3\tabto{1.75cm}Proofs of Lemmas A.1--A.3\Dotfill\pageref{subsec:Proofs-of-Lemmas}

\noindent B\tabto{0.75cm}Proofs under low-level conditions\Dotfill\pageref{sec:LLapp}

\noindent \tabto{0.75cm}B.1\tabto{1.75cm}Sufficient conditions for
Assumption HL\Dotfill\pageref{sec:suff-lp}

\noindent \tabto{0.75cm}B.2\tabto{1.75cm}Lemmas for $I(1/2)$ and
MI processes\Dotfill\pageref{subsec:procaux}

\noindent \tabto{0.75cm}B.3\tabto{1.75cm}Proofs of Theorems~3.2
and 3.4\Dotfill\pageref{subsec:lowleveltheorems}

\noindent C\tabto{0.75cm}Proofs auxiliary to Appendix B\Dotfill\pageref{sec:auiliaryapp}

\noindent \tabto{0.75cm}C.1\tabto{1.75cm}Proofs of Lemmas COEF,
LVAR and CLT\Dotfill\pageref{subsec:FRMIlemmas}

\noindent \tabto{0.75cm}C.2\tabto{1.75cm}Proofs of Lemmas B.1 and
B.2\Dotfill\pageref{subsec:technicalproofs}

\noindent D\tabto{0.75cm}Proofs of Theorems 4.1 and 4.2\Dotfill\pageref{sec:Proofs-of-Theorems4.1}

\noindent E\tabto{0.75cm}Proof of Theorem 5.1\Dotfill\pageref{sec:Proof-of-Theorem5.1}

\noindent F\tabto{0.75cm}Additional simulations for Section 6\Dotfill\pageref{sec:Additional-simulations-for}

\bigskip{}

Throughout, $C,C_{1},C_{2},\ldots\in(0,\infty)$ denote generic constants
which may take different values at each appearance, even within the
same proof. $\overset{a.s.}{\rightarrow}$, $\overset{p}{\rightarrow}$
and $\overset{d}{\rightarrow}$ respectively denote convergence almost
surely, in probability, and in distribution. For deterministic sequences
$\{a_{n}\}$ and $\{b_{n}\}$, $a_{n}\sim b_{n}$ denotes $\lim_{n\rightarrow\infty}a_{n}/b_{n}=1$
and $a_{n}\asymp b_{n}$ denotes $\lim_{n\rightarrow\infty}|a_{n}/b_{n}|\in(0,\infty)$.
For random variables $X$ and $Y$, $X\sim F$ denotes that $X$ has
distribution $F$, and $X\overset{d}{=}Y$ that $X$ has the same
distribution as $Y$. For a positive real number $x$, $\lfloor x\rfloor$
denotes its integer part. $\mathbf{1}\{A\}$ denotes the indicator
function for the set $A$. $\overline{\mathbb{R}}$, $\mathbb{R}_{+}$
and $\mathbb{R}_{+}^{\ast}$ are the extended, the nonnegative, and
(strictly) positive real numbers respectively. $f^{(j)}(x)$ denotes
the $j$th derivative of the function $f(x)$. All limits are taken
as $n\rightarrow\infty$ unless otherwise indicated.

References to Bingham, Goldie and Teugels (1987) are henceforth abbreviated
to `BGT'.

\renewcommand{\theequation}{S.\arabic{equation}}
\setcounter{equation}{0}

\pagebreak

\appendix

\section{Proofs under high-level conditions}

\label{sec:proofs}

This appendix provides proofs of our main results under high level
conditions (Assumption\textbf{ HL}), i.e.\ Theorems \ref{MainThm}
and \ref{MainKernel}. We begin by stating some auxiliary technical
lemmas, whose proofs appear in Appendix~A.3. 

\subsection{Technical lemmas}

\label{subsec:Technical-lemmas}
\begin{lem}
\label{1A} Suppose $f:\mathbb{R}\rightarrow\mathbb{R}$\ is locally
integrable, and let $\varepsilon,\eta>0$. Then there is a Lipschitz
continuous $f_{\varepsilon,\eta}$ such that $\int_{\left\vert x\right\vert \leq\eta}\left\vert f(x)-f_{\varepsilon,\eta}(x)\right\vert dx<\varepsilon$\ and
$f_{\varepsilon,\eta}(x)=0$\ for $\left\vert x\right\vert >\eta$.
\end{lem}
\begin{lem}
\label{Kallenberg}Let $\{X_{n}\}$ and $\{Y_{n}\}$ be real valued
random sequences on some probability space $(\Omega,\mathcal{A},\mathbf{P})$\ and
$\mathcal{F\subset A}$\ a $\sigma$-field, for which
\begin{enumerate}
\item $X_{n}\overset{d}{\rightarrow}X\sim F_{X}$, conditionally on $\mathcal{F}$,
in the sense that $\mathbf{E}\left(h(X_{n})\mid\mathcal{F}\right)\overset{p}{\rightarrow}\int_{\mathbb{R}}h(x)dF_{X}(x)$
for all $h:\mathbb{R}\rightarrow\mathbb{R}$\ bounded and continuous;
and
\item $Y_{n}\overset{d}{\rightarrow}Y$, where $Y_{n}$\ is $\mathcal{F}$-measurable
for each $n$.
\end{enumerate}
\noindent Then for all $g:\mathbb{R}\times\mathbb{R}\rightarrow\mathbb{R}$\ bounded
and Lipschitz continuous, 
\[
\mathbf{E}\left(g(X_{n},Y_{n})\mid\mathcal{F}\right)\overset{d}{\rightarrow}\int_{\mathbb{R}}g(x,Y)dF_{X}(x).
\]
\end{lem}
\begin{lem}
\label{Jegs} Suppose that:
\begin{enumerate}
\item \textbf{HL0} and\textbf{ HL7--9} hold;
\item $K$ satisfies Assumption\textbf{ K}; and
\item $\left\{ h_{n}\right\} $ is such that $\beta_{n}^{-1}h_{n}+\beta_{n}(nh_{n})^{-1}\rightarrow0$.
\end{enumerate}
\noindent Then for all $x\in\mathbb{R}$, and $t_{0}$ as in \textbf{HL7},
\[
\lim_{\varepsilon\downarrow0}\limsup_{n\rightarrow\infty}\mathbf{E}\left|\frac{\beta_{n}}{h_{n}n}\sum_{t=t_{0}}^{n}K\left(\frac{X_{t}(n)-x}{h_{n}}\right)-\frac{1}{n}\sum_{t=t_{0}}^{n}\varphi_{\varepsilon^{2}}\left(\beta_{n}^{-1}X_{t}(n)\right)\right|=0.
\]

\end{lem}

\subsection{Proofs of Theorems 3.1 and 3.3}

\label{subsec:Proofs-of-Theorems}
\begin{proof}[Proof of Theorem \ref{MainThm}]
 By Lemma\ \ref{1A}, for each $\varepsilon>0$ there is a Lipschitz
continuous $f_{\varepsilon}\left(x\right)$ such that $\int_{\left\vert x\right\vert \leq\varepsilon^{-1}}\left\vert f\left(x\right)-f_{\varepsilon}\left(x\right)\right\vert dx<\varepsilon$
and $f_{\varepsilon}\left(x\right)=0$ for $\left\vert x\right\vert >\varepsilon^{-1}$.
We shall prove that:
\begin{equation}
\frac{1}{n}\sum_{t=1}^{n}f\left(\beta_{n}^{-1}X_{t}\left(n\right)\right)=\frac{1}{n}\sum_{t=1}^{n}f_{\varepsilon}\left(\beta_{n}^{-1}X_{t}\left(n\right)\right)+o_{p}(1),\label{Thm1.1.1}
\end{equation}
as $n\rightarrow\infty$ then $\varepsilon\rightarrow0$; that for
each $\varepsilon>0$
\begin{equation}
\frac{1}{n}\sum_{t=1}^{n}f_{\varepsilon}\left(\beta_{n}^{-1}X_{t}\left(n\right)\right)\overset{d}{\rightarrow}\int_{\mathbb{R}}f_{\varepsilon}\left(x+X^{-}\right)\Phi_{X^{+}}(x)dx,\label{Thm1.1.2}
\end{equation}
as $n\rightarrow\infty$; and that 
\begin{equation}
\int_{\mathbb{R}}f_{\varepsilon}\left(x+X^{-}\right)\Phi_{X^{+}}(x)dx\overset{a.s.}{\rightarrow}\int_{\mathbb{R}}f\left(x+X^{-}\right)\Phi_{X^{+}}(x)dx.\label{Thm1.1.3}
\end{equation}
as $\varepsilon\rightarrow0$, where the integral on the r.h.s.\ exists
a.s.\ by condition (i) of Theorem~\ref{MainThm}. In view of (\ref{Thm1.1.1})-(\ref{Thm1.1.3}),
(\ref{Thm2.1}) then follows from Theorem 4.2 in Billingsley (1968).

\emph{Proof of (\ref{Thm1.1.1}).} We prove (\ref{Thm1.1.1})\textbf{\ }under
condition (iii.a) of Theorem~\ref{MainThm}. The proof under (iii.b)
is identical and therefore omitted. By condition (iii.a) we may choose
$\varepsilon>0$\ sufficiently small that $\left\vert f\left(x\right)\right\vert \leq C\left\vert x\right\vert ^{\lambda^{\prime}}$
for all $\left\vert x\right\vert \geq\varepsilon^{-1}$. Decompose
\[
f\left(x\right)=f\left(x\right)1\left\{ \left\vert x\right\vert \leq\varepsilon^{-1}\right\} +f\left(x\right)1\left\{ \left\vert x\right\vert >\varepsilon^{-1}\right\} =:f_{1,\varepsilon}\left(x\right)+f_{2,\varepsilon}\left(x\right),
\]
where $\int_{\left\vert x\right\vert \leq\varepsilon^{-1}}\left\vert f_{1,\varepsilon}\left(x\right)-f_{\varepsilon}\left(x\right)\right\vert dx<\varepsilon$,
and $\left\vert f_{2,\varepsilon}\left(x\right)\right\vert \leq C\left\vert x\right\vert ^{\lambda^{\prime}}1\{\lvert x\rvert>\mathbb{\varepsilon}^{-1}\}$.
Letting $\tilde{X}_{t}\left(n\right):=\beta_{n}^{-1}X_{t}\left(n\right)$,
in view of condition (ii) of Theorem~\ref{MainThm}, (\ref{Thm1.1.1})
will follow once we have shown that
\begin{equation}
\frac{1}{n}\sum_{t=t_{0}}^{n}\mathbf{E}\lvert f_{1,\varepsilon}(\tilde{X}_{t}(n))-f_{\varepsilon}(\tilde{X}_{t}(n))\rvert+\frac{1}{n}\sum_{t=1}^{n}\mathbf{E}\lvert f_{2,\varepsilon}(\tilde{X}_{t}(n))\rvert\rightarrow0,\label{Thm1.1.4}
\end{equation}
as $n\rightarrow\infty$ and then $\varepsilon\rightarrow0$, for
$t_{0}$ as in condition (ii).

To that end, note that by \textbf{HL5(a)} and condition (iii.a) of
Theorem~\ref{MainThm}, there is an $n_{0}\geq1$ and a $\lambda>\lambda^{\prime}$
such that $\sup_{n\geq n_{0},1\leq t\leq n}\mathbf{E}\lvert\tilde{X}_{t}(n)\rvert^{\lambda}<\infty$.
Hence $\lvert\tilde{X}_{t}(n)\rvert^{\lambda^{\prime}}$\ is uniformly
integrable, and so for $n\geq n_{0}$
\[
\frac{1}{n}\sum_{t=1}^{n}\mathbf{E}\lvert f_{2,\varepsilon}(\tilde{X}_{t}(n))\rvert\leq\sup_{n\geq n_{0},1\leq t\leq n}\mathbf{E}\lvert\tilde{X}_{t}(n)\rvert^{\lambda^{\prime}}1\{\lvert\tilde{X}_{t}(n)\rvert>\varepsilon^{-1}\}\rightarrow0,
\]
as $\varepsilon\rightarrow0$. This gives the required negligibility
of the second l.h.s.\ term in (\ref{Thm1.1.4}). For the first l.h.s.\ term,
we note \textbf{HL3} implies that for $n$ sufficiently large
\begin{align*}
 & \frac{1}{n}\sum_{t=t_{0}}^{n}\mathbf{E}\lvert f_{1,\varepsilon}(\tilde{X}_{t}(n))-f_{\varepsilon}(\tilde{X}_{t}(n))\rvert\\
 & \qquad\qquad=\frac{1}{n}\sum_{t=t_{0}}^{n}\int_{\mathbb{R}}\left\vert f_{1,\varepsilon}\left(\frac{\beta_{t}}{\beta_{n}}x\right)-f_{\varepsilon}\left(\frac{\beta_{t}}{\beta_{n}}x\right)\right\vert \mathcal{D}_{n,t}\left(x\right)dx\\
 & \qquad\qquad\leq\sup_{n\geq n_{0},t_{0}\leq t\leq n}\sup_{u}\mathcal{D}_{n,t}\left(u\right)\int_{\mathbb{R}}\lvert f_{1,\varepsilon}(x)-f_{\varepsilon}(x)\rvert dx\cdot\frac{\beta_{n}}{n}\sum_{t=t_{0}}^{n}\beta_{t}^{-1}\\
 & \qquad\qquad\leq C\int_{\left\vert x\right\vert \leq\varepsilon^{-1}}\left\vert f_{1,\varepsilon}\left(x\right)-f_{\varepsilon}\left(x\right)\right\vert dx\\
 & \qquad\qquad\leq C\varepsilon\rightarrow0
\end{align*}
as $\mathbb{\varepsilon}\rightarrow0$, where the second inequality
holds by \textbf{HL6}.

\emph{Proof of (\ref{Thm1.1.2}).} Let $\mathbf{E}_{0}(\cdot):=\mathbf{E}\left(\cdot\mid\mathcal{F}_{0}\right)$,
and $\delta\in(0,1)$. By \textbf{HL4} and the boundedness of $f_{\varepsilon}$,
\begin{equation}
\frac{1}{n}\sum_{t=1}^{n}f_{\varepsilon}(\tilde{X}_{t}(n))=\frac{1}{n}\sum_{t=\lfloor n\delta\rfloor+1}^{n}\mathbf{E}_{0}[f_{\varepsilon}(\tilde{X}_{t}(n))]+o_{p}(1)\label{Thm1.1.5}
\end{equation}
as $n\rightarrow\infty$ and then $\delta\rightarrow0$. Now let $\tilde{X}_{t}^{+}(n):=\beta_{n}^{-1}X_{t}^{+}(n)$
and $\tilde{X}_{t}^{-}(n):=\beta_{n}^{-1}X_{t}^{-}(n)$. Since $f_{\mathbb{\varepsilon}}$
is bounded and Lipschitz, it follows from \textbf{HL1} that 
\begin{multline}
\frac{1}{n}\sum_{t=\lfloor n\delta\rfloor+1}^{n}\mathbf{E}\left\vert \mathbf{E}_{0}[f_{\varepsilon}(\tilde{X}_{t}(n))]-\mathbf{E}_{0}[f_{\varepsilon}(\tilde{X}_{t}^{+}(n)+\tilde{X}_{t}^{-}(n))]\right\vert \\
\leq C\sup_{1\leq t\leq n}\mathbf{E}(\beta_{n}^{-1}\lvert R_{t}(n)\rvert\wedge1)\rightarrow0\label{Thm1.1.5b}
\end{multline}
as $n\rightarrow\infty$ for each $\delta\in(0,1)$; and from \textbf{HL2(b)}
that 
\begin{align}
 & \frac{1}{n}\sum_{t=\lfloor n\delta\rfloor+1}^{n}\mathbf{E}\left\vert \mathbf{E}_{0}[f_{\varepsilon}(\tilde{X}_{t}^{+}(n)+\tilde{X}_{t}^{-}(n))]-\mathbf{E}_{0}[f_{\varepsilon}(\tilde{X}_{t}^{+}(n)+\tilde{X}_{n}^{-}(n))]\right\vert \nonumber \\
 & \qquad\qquad\qquad\qquad\leq C\sup_{\lfloor n\delta\rfloor+1\leq t\leq n}\mathbf{E}(\lvert\tilde{X}_{t}^{-}(n)-\tilde{X}_{n}^{-}(n)\rvert\wedge1)\nonumber \\
 & \qquad\qquad\qquad\qquad=C\mathbf{E}(\lvert\tilde{X}_{l_{n}}^{-}(n)-\tilde{X}_{n}^{-}(n)\rvert\wedge1)\rightarrow0\label{Thm1.1.6}
\end{align}
as $n\rightarrow\infty$ for each $\delta\in(0,1)$, where $l_{n}\in\{\lfloor n\delta\rfloor+1,\ldots,n\}$
may always be chosen such that the final equality holds. Finally,
by \textbf{HL2(a)}, Theorem 2.1 in Billingsley (1968) and Lemma~\ref{Kallenberg}
we have as 
\begin{multline}
\frac{1}{n}\sum_{t=\lfloor n\delta\rfloor+1}^{n}\mathbf{E}_{0}[f_{\varepsilon}(\tilde{X}_{t}^{+}(n)+\tilde{X}_{n}^{-}(n))]\\
\overset{d}{\rightarrow}\left(1-\delta\right)\int_{\mathbb{R}}f_{\varepsilon}\left(x+X^{-}\right)\Phi_{X^{+}}(x)dx\label{Thm1.1.7}
\end{multline}
as $n\rightarrow\infty$, for each $\delta\in(0,1)$. Hence (\ref{Thm1.1.2})
follows from (\ref{Thm1.1.5})-(\ref{Thm1.1.7}) and Theorem 4.2 in
Billingsley (1968).

\emph{Proof of (\ref{Thm1.1.3}).} Let $y\in\mathcal{Y}$ for $\mathcal{Y}$
as in condition (i) of Theorem~\ref{MainThm}. Noting $f_{\mathbb{\varepsilon}}(x)=0$
for $\lvert x\rvert>\mathbb{\varepsilon}$, and that $\Phi_{X^{+}}$
is bounded under \textbf{HL2, }we have 
\begin{align*}
 & \int\lvert f_{\mathbb{\varepsilon}}(x+y)-f(x+y)\rvert\Phi_{X^{+}}(x)dx\\
 & \qquad\leq\sup_{u}\Phi_{X^{+}}(u)\int_{\lvert x\rvert\leq\mathbb{\varepsilon}^{-1}}\lvert f_{\mathbb{\varepsilon}}(x)-f(x)\rvert dx+\int_{\lvert x\rvert>\mathbb{\varepsilon}^{-1}}\lvert f(x)\rvert\Phi_{X^{+}}(x-y)dx\\
 & \qquad=C\mathbb{\varepsilon}+o(1)
\end{align*}
as $\mathbb{\varepsilon}\rightarrow0$, where the negligibility of
the second r.h.s.\ term follows by condition (i) of Theorem~\ref{MainThm}
and the dominated convergence theorem. Noting that $\mathbf{P}\{X^{-}\in\mathcal{Y}\}=1$
completes the proof.
\end{proof}
\begin{proof}[Proof of Theorem~\ref{MainKernel}]
 By Lemma~\ref{Jegs} and condition (ii) of the theorem, 
\[
\frac{\beta_{n}}{h_{n}n}\sum_{t=1}^{n}K\left(\frac{X_{t}(n)-x}{h_{n}}\right)=\frac{1}{n}\sum_{t=1}^{n}\varphi_{\varepsilon^{2}}(\beta_{n}^{-1}X_{t}(n))+o_{p}(1),
\]
as $n\rightarrow\infty$ and then $\varepsilon\rightarrow0$. By Theorem~\ref{MainThm}
we have
\[
\frac{1}{n}\sum_{t=1}^{n}\varphi_{\varepsilon^{2}}(\beta_{n}^{-1}X_{t}(n))\overset{d}{\rightarrow}\int_{\mathbb{R}}\varphi_{\varepsilon^{2}}(x+X^{-})\Phi_{X^{+}}(x)dx,
\]
as $n\rightarrow\infty$. Finally, as noted in (\ref{kernapprox}),
\[
\int_{\mathbb{R}}\varphi_{\varepsilon^{2}}(x+X^{-})\Phi_{X^{+}}(x)dx\rightarrow\Phi_{X^{+}}(-X^{-}),
\]
as $\varepsilon\rightarrow0$ by the continuity of $\Phi_{X^{+}}$
under \textbf{HL2}. The result then follows by Theorem 4.2 in Billingsley
(1968).
\end{proof}

\subsection{Proofs of Lemmas A.1--A.3}

\label{subsec:Proofs-of-Lemmas}
\begin{proof}[Proof of Lemma \ref{1A}]
 See Theorem 2.26 in Folland (1999).
\end{proof}
\begin{proof}[Proof of Lemma \ref{Kallenberg}]
 By Theorems 6.3 and 6.4 in Kallenberg (2001), there is a probability
kernel $\nu_{n}$ from $\left(\Omega,\mathcal{F}\right)$ to $\left(\mathbb{R},\mathcal{B}\left(\mathbb{R}\right)\right)$
such that for each $\mathcal{F}$-measurable random variable $\eta$,
\[
\mathbf{E}\left(g(X_{n},\eta)\mid\mathcal{F}\right)\overset{a.s.}{=}\int_{\mathbb{R}}g(x,\eta)d\nu_{n}(x).
\]
For each $y\in\mathbb{R}$, define $h_{n}(y):=\int_{\mathbb{R}}g(x,y)d\nu_{n}(x)$:
so by the preceding and condition~(i) of the lemma, we have
\[
h_{n}\left(y\right)\overset{a.s.}{=}\mathbf{E}\left(g(X_{n},y)\mid\mathcal{F}\right)\overset{p}{\rightarrow}\int_{\mathbb{R}}g(x,y)dF_{X}(x)=:h(y).
\]
Moreover, by the Lipschitz continuity of $g$
\[
\left\vert h_{n}\left(y\right)-h_{n}\left(y^{\prime}\right)\right\vert \leq\int_{\mathbb{R}}\left\vert g(x,y)-g(x,y^{\prime})\right\vert d\nu_{n}(x)\leq C\left\vert y-y^{\prime}\right\vert .
\]
Hence, $\left\{ h_{n}\left(y\right)\right\} $\ is stochastically
equicontinuous on $\mathbb{R}$, whence $h_{n}\left(y\right)\overset{p}{\rightarrow}h\left(y\right)$\ uniformly
on every compact subset of $\mathbb{R}$\ (see for example Theorem
1 and Lemma 1 in Andrews, 1992). Finally, fix $\varepsilon>0$ and
choose $M_{\varepsilon}$ such that $\lim\sup_{n\rightarrow\infty}\mathbf{P}\left(\left\vert Y_{n}\right\vert >M_{\varepsilon}\right)<\varepsilon$,
which is possible since $Y_{n}\overset{d}{\rightarrow}Y$. Then 
\begin{align*}
 & \mathbf{P}\left(\left\vert h_{n}\left(Y_{n}\right)-h\left(Y_{n}\right)\right\vert >\varepsilon\right)\\
 & \qquad\leq\mathbf{P}\left(\left\{ \left\vert h_{n}\left(Y_{n}\right)-h\left(Y_{n}\right)\right\vert >\varepsilon\right\} \cap\left\{ \left\vert Y_{n}\right\vert \leq M_{\varepsilon}\right\} \right)+\mathbf{P}\left(\left\vert Y_{n}\right\vert >M_{\varepsilon}\right)\\
 & \qquad\leq\mathbf{P}\left(\sup_{\left\vert y\right\vert \leq M_{\varepsilon}}\left\vert h_{n}\left(y\right)-h\left(y\right)\right\vert >\varepsilon\right)+\varepsilon\\
 & \qquad\rightarrow\varepsilon,
\end{align*}
as $n\rightarrow\infty$, by the uniform convergence in probability
of $h_{n}$ on compacta. In view of the preceding, 
\[
\mathbf{E}\left(g(X_{n},Y_{n})\mid\mathcal{F}\right)\overset{a.s.}{=}h_{n}\left(Y_{n}\right)=h\left(Y_{n}\right)+o_{p}(1)\overset{d}{\rightarrow}_{(1)}\int_{\mathbb{R}}g(x,Y)dF_{X}(x),
\]
where $\overset{d}{\rightarrow}_{(1)}$ is due to condition (ii) and
the continuity of $h(y)$. 
\end{proof}
\begin{proof}[Proof of Lemma \ref{Jegs}]
 The result follows from Lemma 7 in Jeganathan (2004) and arguments
similar to those used in Wang and Phillips (2009, pp.\ 725-728).
Note, in particular, that if we define $x_{k,n}:=\beta_{n}^{-1}X_{k}(n)$
and $d_{l,k,n}:=\beta_{n}^{-1}\beta_{l-k}$, then \textbf{HL7--9}
ensure that their Assumption~2.3 is satisfied, with the exception
that the density of $x_{k,n}$ is bounded only for $k\geq t_{0}$.
If $t_{0}=1$, the result follows immediately. Otherwise, it follows
via very minor modifications of the arguments leading to (5.2) in
Wang and Phillips (2009), making use in particular of the additional
condition imposed on $\{\beta_{n}\}$ by \textbf{HL9(e)}.
\end{proof}

\section{Proofs under low-level conditions}

\label{sec:LLapp}

This appendix provides proofs for the remaining results of Section~3,
i.e.\ Theorems~\ref{MainLL} and \ref{KernelLL}. Technical lemmas
stated in this appendix are proved in Appendix~\ref{sec:auiliaryapp}.

\subsection{Sufficient conditions for Assumption HL}

\label{sec:suff-lp}

Preliminary to the proof of Theorem~\ref{MainLL}, we first present
a set of `intermediate-level conditions' for general linear processes
(Assumption\textbf{ LP} below); their sufficiency for Assumptions\textbf{~HL0--HL6}
is established by Proposition~\ref{LPtoHL} below. These conditions
are of interest in their own right, insofar as they allow the conclusions
of Theorem~\ref{MainThm} to be extended to a broad class of linear
processes. They also allow the proof of Theorem~\ref{MainLL} to
be reduced essentially to verifying that $I(1/2)$ and MI processes
satisfy Assumption \textbf{LP} (see Appendix~\ref{subsec:lowleveltheorems}).

Our conditions on linear processes shall be stated in terms of arrays
of the form 
\begin{equation}
x_{t}(n)=\sum_{k=0}^{\infty}a_{k,t}(n)\xi_{t-k}.\label{linproc}
\end{equation}
where $\{\xi_{t}\}$ is the i.i.d.\ sequence appearing in Assumption
\textbf{INN}. We showed in Section~\ref{process-intro} that it was
possible to write both $I(1/2)$ and MI processes in the form
\begin{equation}
x_{t}(n)=\sum_{j=0}^{t-1}\phi_{j}(n)v_{t-j},\qquad\text{where}\qquad v_{t}=\sum_{i=0}^{\infty}c_{i}\xi_{t-i}\label{eq:frmilinproc}
\end{equation}
for appropriate coefficients $\{c_{i}\}$ and $\{\phi_{j}(n)\}$.
Suppressing the dependence of these quantities on $n$ for the sake
a readability, as we shall do freely below, (\ref{eq:frmilinproc})
implies that
\begin{equation}
a_{k,t}=\sum_{j=0}^{(t-1)\wedge k}\phi_{j}c_{k-j}=\begin{cases}
\sum_{j=0}^{k}\phi_{j}c_{k-j}=:a_{k} & \text{if }0\leq k\leq t-1,\\
\sum_{j=0}^{t-1}\phi_{j}c_{k-j}=:a_{k,t}^{-} & \text{if }k\geq t.
\end{cases}\label{lincoefs}
\end{equation}
Note, in particular, that $a_{k,t}$ does not depend on $t$ for $1\leq k\leq t-1$,
and we accordingly denote these coefficients by simply $a_{k}$. For
$k\geq t$, the notation $a_{k,t}^{-}$ reminds us that these coefficients
refer to innovations dated $t\leq0$.

The following conditions on linear processes do \emph{not} require
these to have been generated according to a specific time series model
(e.g.\textbf{\ }Assumption\textbf{ FR/MI}), or indeed as in (\ref{eq:frmilinproc}).
We shall, however, impose one restriction consistent with that model:
that the coefficients $a_{k,t}$ should not depend on $t$ for $1\leq k\leq t-1$.
Our conditions thus envisage an array of the form 
\begin{equation}
x_{t}(n)=\sum_{k=0}^{t-1}a_{k}(n)\xi_{t-k}+\sum_{k=t}^{\infty}a_{k,t}^{-}(n)\xi_{t-k}=:x_{t}^{+}(n)+x_{t}^{-}(n),\label{model01b}
\end{equation}
associated to which, define 
\[
\beta_{n,t}^{2}:=Var(x_{t}(n))=\sigma_{\xi}^{2}\sum_{k=0}^{t-1}a_{k}(n)^{2}+\sigma_{\xi}^{2}\sum_{k=t}^{\infty}a_{k,t}^{-}(n)^{2}=:(\beta_{n,t}^{+})^{2}+(\beta_{n,t}^{-})^{2}
\]
and set $\beta_{n}:=\beta_{n,n}$, for $n\in\mathbb{N}$ and $t\in\{1,\ldots,n\}$.
The following will always be applied in conjunction with Assumption\textbf{
INN}, and are stated in terms of the $\theta$ and the i.i.d.\ sequence
$\{\xi_{t}\}$ appearing in that assumption.

\setcounter{assumption}{327}
\begin{assumption}[linear process]
~
\begin{enumerate}[label=\upshape{\textbf{LP\arabic*}}]
\item $x_{t}(n)$ is as in (\ref{model01b}), with $\beta_{n,t}\in(0,\infty)$
for all $n,t\in\mathbb{N}$.
\item Either:
\begin{enumerate}[label=\upshape{\textbf{(\alph*)}}]
\item $\mathbf{E}\lvert\xi_{1}\rvert^{\lambda}<\infty$ for some $\lambda\in[2,\infty)$;
or 
\item $\xi_{1}$ has a finite moment generating function (m.g.f.) in a neighbourhood
of zero.
\end{enumerate}
\item There exists $t_{0}\in\mathbb{N}$ such that
\begin{enumerate}[label=\upshape{\textbf{(\alph*)}}]
\item  $\liminf_{n\rightarrow\infty}\inf_{t_{0}\leq t\leq n}\beta_{t}^{-1}\beta_{n,t}^{+}>0$;
\item $\limsup_{t\rightarrow\infty}\sup_{n\geq t}(\beta_{n,t}^{+})^{-1}\max_{0\leq k\leq t-1}\lvert a_{k}(n)\rvert=0$
\end{enumerate}
\item There exist $\delta>0$ and $n_{0}\in\mathbb{N}$ such that: for each
$1\leq t\leq t_{0}$ (as in \textbf{\emph{LP3}}) and $n\geq n_{0}$,
there exist $\{k_{1},\ldots k_{\theta}\}\subset\mathbb{N}$ with
\[
\min_{1\leq l\leq\theta}\lvert a_{k_{l},t}(n)\rvert>\delta.
\]
\item $\lim_{n\rightarrow\infty}n^{-1/2}\beta_{n}^{-1}\sum_{k=0}^{n-1}\lvert a_{k}(n)\rvert=0$.
\item For any $\delta\in(0,1)$\ and $\{t_{n}\}$\ such that $\lfloor n\delta\rfloor\leq t_{n}\leq n$:
\begin{enumerate}[label=\upshape{\textbf{(\alph*)}}]
\item $\beta_{n}^{-1}[\beta_{n,t_{n}}^{+},\beta_{n,t_{n}}^{-}]\rightarrow[\sigma_{+},\sigma_{-}]$
with $\sigma_{+}>0$\ and $\sigma_{-}\geq0$;
\item $\beta_{n}^{-1}(\max_{0\leq k\leq n-1}\lvert a_{k}(n)\rvert+\sup_{1\leq t\leq n,k\geq t}\lvert a_{k,t}^{-}(n)\rvert)\rightarrow0$;
\item $\beta_{n}^{-2}\sum_{l=-\infty}^{0}[a_{t_{n}-l,t_{n}}^{-}(n)-a_{n-l,n}^{-}(n)]^{2}\rightarrow0$.
\end{enumerate}
\item $\sup_{n\geq n_{0},1\leq t\leq n}\beta_{n}^{-1}\beta_{n,t}<\infty$
for some $n_{0}\in\mathbb{N}$.
\end{enumerate}
\end{assumption}
\begin{rem}
The principal relationships between the preceding conditions, and
the high-level conditions (Assumption\textbf{ HL}) as they would be
applied to $X_{t}(n)=x_{t}(n)$ may be summarised as follows; these
are formally established by Proposition~\ref{LPtoHL} below.

\textbf{(a) LP2} and \textbf{LP7} imply that $\beta_{n}^{-1}x_{t}(n)$
has uniformly bounded moments of a sufficient order (as per \textbf{HL5}).
For $\beta_{n}^{-1}x_{t}(n)$ to have finite $\lambda$-moments, it
is sufficient that $\xi_{t}$ also have finite $\lambda$-moments;
$\beta_{n}^{-1}x_{t}(n)$ will have finite exponential moments if
$\xi_{t}$ has a finite m.g.f.\ in a neighbourhood of zero.

\textbf{(b)} \textbf{LP3} and \textbf{LP4} ensure that $\beta_{t}^{-1}x_{t}(n)$
has a uniformly bounded density, as required by \textbf{HL3}. Under
\textbf{INN}, a weighted sum involving at least $\theta$ of the innovations
$\xi_{t}$ will have an integrable characteristic function (c.f.).
$\beta_{t}^{-1}x_{t}^{+}(n)=\sum_{k=0}^{t-1}\beta_{t}^{-1}a_{k}(n)\xi_{t-k}$
will thus have a density bounded uniformly over $n$ and $t$, provided
that the $L_{1}$ norm of its c.f.\ can be uniformly bounded. This
in turn requires that: (i) the variance of $\beta_{t}^{-1}x_{t}^{+}(n)$
can be bounded away from zero; and (ii) it is never dominated by less
than $\theta$ of the innovations that contribute to it. Both are
ensured by \textbf{LP3}, at least for $n$ and $t$ sufficiently large.
\textbf{LP4} entails that for \emph{every} $t$, a sufficient number
of coefficients $\{a_{k,t}(n)\}$ are bounded away from zero; together
with \textbf{LP3} it is sufficient for \textbf{HL3} to hold with $t_{0}=1$.

\textbf{(c)} \textbf{LP5} can be understood as a kind of weak dependence
condition, which is used solely to verify \textbf{HL4}.

\textbf{(d)} \textbf{LP6 }permits a central limit theorem for weighted
sums to be applied to each of $\beta_{n}^{-1}x_{t}^{+}(n)$ and $\beta_{n}^{-1}x_{t}^{-}(n)$,
as required by \textbf{HL2}. \textbf{LP6(a)} determines the limiting
variance of each of these two terms, while \textbf{LP6(b)} is a negligibility
requirement on the linear process coefficients, akin to a Lindeberg
condition. Finally, \textbf{LP6(c)} implies the second part of \textbf{HL2(b)}.
\end{rem}
\begin{prop}[LP $\Rightarrow$ HL]
\label{LPtoHL} Suppose Assumptions\textbf{~LP1} and \textbf{INN}
hold. Then \textbf{HL0--1} hold for $\mathcal{F}_{t}:=\sigma(\{\xi_{r}\}_{r\leq t})$,
$X_{t}^{+}(n):=x_{t}^{+}(n)$, $X_{t}^{-}(n):=x_{t}^{-}(n)$, and
$\beta_{n}^{2}=Var[x_{n}(n)]$. Moreover:
\begin{enumerate}
\item \textbf{LP6}\ $\Rightarrow$\textbf{\ HL2} with $(X^{+},X^{-})\sim N[0,\mathrm{diag}\{\sigma_{+}^{2},\sigma_{-}^{2}\}]$.
\item  
\begin{enumerate}
\item \textbf{LP3 }$\Rightarrow$\ \textbf{HL3};
\item \textbf{LP3 }and \textbf{LP4\ }$\Rightarrow$\ \textbf{HL3}\ with
$t_{0}=1$.
\end{enumerate}
\item \textbf{LP5}\ $\Rightarrow$\textbf{\ HL4}.
\item  
\begin{enumerate}
\item \textbf{LP2(a) }and \textbf{LP7}\ $\Rightarrow$\ \textbf{HL5(a)
}with $\lambda\geq2$;
\item \textbf{LP2(b)},\textbf{\ LP6(b)}\ and \textbf{LP7}\ $\Rightarrow$\ \textbf{HL5(b)}.
\end{enumerate}
\end{enumerate}
\end{prop}
The proof of Proposition~\ref{LPtoHL} (and subsequently, Theorem~\ref{KernelLL})
requires the following technical lemma, whose proof is deferred to
Appendix~\ref{subsec:technicalproofs}. Let $\mathbb{N}_{0}:=\mathbb{N}\cup\{0\}.$
\begin{lem}
\label{lem:UICF} Let $\{\xi_{t}\}$ be as in Assumption~\textbf{INN}
with $\sigma_{\xi}^{2}=1$, $\{\vartheta_{k}\}_{k\in\mathbb{N}_{0}}$
be a real sequence, and $\delta>0$.
\begin{enumerate}
\item Suppose $\sigma_{\vartheta}^{2}:=\sum_{k=0}^{\infty}\vartheta_{k}^{2}>0$
and $\max_{k\in\mathbb{N}_{0}}\vartheta_{k}^{2}\leq\sigma_{\vartheta}^{2}/2\theta$.
Then there exists a function $G(A;\sigma^{2},\psi_{\xi})$, not otherwise
depending on $\{\vartheta_{k}\}$, such that $\sigma^{2}\mapsto G(A;\sigma^{2},\psi_{\xi})$
is weakly decreasing in $\sigma^{2}$,
\[
\int_{\{\lvert\lambda\rvert\geq A\}}\left|\mathbf{E}\left(\mathrm{i}\lambda\sum_{k=0}^{\infty}\vartheta_{k}\xi_{k}\right)\right|d\lambda\leq G(A;\sigma_{\vartheta}^{2},\psi_{\xi})\leq C\sigma_{\vartheta}^{-1}
\]
where $C<\infty$ depends only on $\psi_{\xi}$, and $\lim_{A\rightarrow\infty}G(A;\sigma_{\vartheta}^{2},\psi_{\xi})=0$.
\item Suppose there exist $\{k_{1},\ldots,k_{\theta}\}\subset\mathbb{N}_{0}$
such that $\lvert\vartheta_{k_{i}}\rvert>\delta$ for all $i\in\{1,\ldots,\theta\}$.
Then
\[
\int_{\mathbb{R}}\left|\mathbf{E}\left(\mathrm{i}\lambda\sum_{k=0}^{\infty}\vartheta_{k}\xi_{k}\right)\right|d\lambda<C\delta^{-1}
\]
where $C<\infty$ depends only on $\psi_{\xi}$.
\end{enumerate}
\end{lem}
\begin{proof}[Proof of Proposition~\ref{LPtoHL}]
 That \textbf{HL0--1} are satisfied (with $R_{t}(n)=0$) is immediate
from (\ref{model01b}), and the fact that $\beta_{n}=\beta_{n,n}\in(0,\infty)$
by \textbf{LP1}.

\textbf{(i).} Since $\{\xi_{t}\}$ is i.i.d., in view of (\ref{model01b})
we have
\[
\beta_{n}^{-1}[x_{t_{n}}^{+}(n),x_{t_{n}}^{-}(n)]\overset{d}{=}\beta_{n}^{-1}\left[\sum_{k=0}^{t_{n}-1}a_{k}\left(n\right)\xi_{k},\sum_{k=t_{n}}^{\infty}a_{k,t_{n}}^{-}\left(n\right)\xi_{k}^{\ast}\right].
\]
where $\left\{ \xi_{k}^{\ast}\right\} \overset{d}{=}\left\{ \xi_{k}\right\} $
and $\left\{ \xi_{k}^{\ast}\right\} \perp\left\{ \xi_{k}\right\} $.
Using \textbf{LP6(a)--(b)}, the weak convergence of these quantities
to a $N[0,\mathrm{diag}\{\sigma_{+}^{2},\sigma_{-}^{2}\}]$ distribution,
as required by \textbf{HL2}, follows immediately from Lemma 2.1 of
Abadir, Distaso, Giraitis and Koul (2014). Finally, note that due
to \textbf{LP6(c)}
\begin{align}
\beta_{n}^{-2}\mathbf{E}\left[x_{t_{n}}^{-}(n)-x_{n}^{-}(n)\right]^{2} & =\beta_{n}^{-2}\mathbf{E}\left[\sum_{k=t_{n}}^{\infty}a_{k,t_{n}}^{-}(n)\xi_{t_{n}-k}-\sum_{k=n}^{\infty}a_{k,n}^{-}(n)\xi_{n-k}\right]^{2}\nonumber \\
 & =\beta_{n}^{-2}\mathbf{E}\left[\sum_{l=-\infty}^{0}a_{t_{n}-l,t_{n}}^{-}(n)\xi_{l}-\sum_{l=-\infty}^{0}a_{n-l,n}^{-}(n)\xi_{l}\right]^{2}\nonumber \\
 & =\beta_{n}^{-2}\sigma_{\xi}^{2}\sum_{l=-\infty}^{0}\left[a_{t_{n}-l,t_{n}}^{-}(n)-a_{n-l,n}^{-}(n)\right]^{2}=o(1),\label{eq:varneg}
\end{align}
which establishes the final part of \textbf{HL2(b)}.

\textbf{(ii). }By (\ref{model01b}) and the Fourier inversion theorem
(e.g.\ Feller, 1971, Theorem~XV.3), it suffices for \textbf{HL3}
to show that the c.f.\ of $\beta_{t}^{-1}x_{t}^{+}(n)$ has an $L^{1}$
norm that is bounded uniformly for all $n$ and $t\leq n$ sufficiently
large. Since $\beta_{t}^{-1}x_{t}^{+}(n)$ is a linear process with
coefficients $\vartheta_{k,n,t}:=\beta_{t}^{-1}a_{k}(n)\boldsymbol{1}\{0\leq k\leq t-1\}$
and variance $\sigma_{\vartheta,n,t}^{2}:=\beta_{t}^{-2}(\beta_{n,t}^{+})^{2}$,
we can do this with the aid of Lemma~\ref{lem:UICF}. \textbf{LP3(b)}
implies that there exists $n_{0},t_{0}\in\mathbb{N}$ such that
\[
\vartheta_{k,n,t}^{2}=\beta_{t}^{-2}a_{k}^{2}(n)\boldsymbol{1}\{0\leq k\leq t-1\}\leq\beta_{t}^{-2}(\beta_{n,t}^{+})^{2}/2\theta=\sigma_{\vartheta,n,t}^{2}/2\theta
\]
for all $k\in\mathbb{N}_{0}$, $t_{0}\leq t\leq n$ and $n\geq n_{0}$,
while \textbf{LP3(a)} implies that $n_{0}$ and $t_{0}$ may be additionally
chosen such that 
\[
\sigma_{\vartheta,n,t}=\beta_{t}^{-1}\beta_{n,t}^{+}>\epsilon>0
\]
for some $\epsilon>0$, for all $t_{0}\leq t\leq n$ and $n\geq n_{0}$.
By Lemma~\ref{lem:UICF}(i), the $L^{1}$ norm of $\mathbf{E}[\mathrm{i}\lambda\beta_{t}^{-1}x_{t}^{+}(n)]$
is bounded by $C\epsilon^{-1}<\infty$ for such $n$ and $t$, and
thus \textbf{HL3} holds for some $t_{0}\in\mathbb{N}$.

\textbf{HL3} will hold with $t_{0}=1$ if we can additionally bound
the $L^{1}$ norm of $\mathbf{E}[\mathrm{i}\lambda\beta_{t}^{-1}x_{t}(n)]$
(note the deliberate omission of the `$+$' superscript) for $1\leq t\leq t_{0}$
and all $n\geq n_{0}$. To this end, note that $\beta_{t}^{-1}x_{t}(n)$
is a linear process with coefficients $\vartheta_{k,n,t}:=\beta_{t}^{-1}a_{k,t}(n)\boldsymbol{1}\{k\geq0\}$,
and that for $\delta>0$ as in \textbf{LP4}, for each $1\leq t\leq t_{0}$
and $n\in\mathbb{N}$ there exist $\{k_{1},\ldots,k_{\theta}\}\in\mathbb{N}_{0}$
such that
\[
\lvert\vartheta_{k_{i},n,t}\rvert=\beta_{t}^{-1}\lvert a_{k_{i},t}(n)\rvert\geq\left(\min_{1\leq t\leq t_{0}}\beta_{t}^{-1}\right)\delta
\]
where the r.h.s.\ is nonzero and independent of $n$ and $t$. Applying
Lemma~\ref{lem:UICF}(ii) completes the proof.

\textbf{(iii).} The argument is similar to that of Wu and Mielniczuk
(2002, p.\ 1452). Let $\mathbf{E}_{t}[\cdot]:=\mathbf{E}[\cdot\mid\mathcal{F}_{t}]$,
and decompose
\[
g(x_{t})-\mathbf{E}_{0}g(x_{t})=\sum_{s=0}^{t-1}\{\mathbf{E}_{t-s}g(x_{t})-\mathbf{E}_{(t-1)-s}g(x_{t})\}
\]
so that
\begin{align}
\sum_{t=1}^{n}\left[g(x_{t})-\mathbf{E}_{0}g(x_{t})\right] & =\sum_{t=1}^{n}\sum_{s=0}^{t-1}[\mathbf{E}_{t-s}g(x_{t})-\mathbf{E}_{(t-1)-s}g(x_{t})]\nonumber \\
 & =\sum_{s=0}^{n-1}\sum_{t=s+1}^{n}[\mathbf{E}_{t-s}g(x_{t})-\mathbf{E}_{(t-1)-s}g(x_{t})]=:\sum_{s=0}^{n-1}M_{n,s},\label{eq:decomp}
\end{align}
where each $M_{n,s}$ is a sum of martingale differences. 

By (\ref{linproc}) we can write
\begin{alignat*}{4}
x_{t} & =\sum_{k=0}^{s-1}a_{k,t}\xi_{t-s} &  & +a_{s,t}\xi_{t-s} &  & +\sum_{k=s+1}^{\infty}a_{k,t}\xi_{t-k}\\
 & \overset{d}{=}\sum_{k=0}^{s-1}a_{k,t}\xi_{t-s} &  & +a_{s,t}\xi^{\ast} &  & +\sum_{k=s+1}^{\infty}a_{k,t}\xi_{t-k} &  & =:x_{t}^{\ast}
\end{alignat*}
where $\xi^{\ast}\overset{d}{=}\xi_{0}$ is independent of $\{\xi_{t}\}$.
Since $g$ is Lipschitz,
\begin{align*}
\lvert\mathbf{E}_{t-s}g(x_{t})-\mathbf{E}_{(t-1)-s}g(x_{t})\rvert & =\lvert\mathbf{E}_{t-s}[g(x_{t})-g(x_{t}^{\ast})\rvert\\
 & \leq C\lvert a_{s,t}\rvert\mathbf{E}_{t-s}\lvert\xi_{t-s}-\xi^{\ast}\rvert
\end{align*}
whence, by the orthogonality of martingale differences
\begin{align*}
\mathbf{E}M_{n,s}^{2} & =\sum_{t=s+1}^{n}\mathbf{E}[\mathbf{E}_{t-s}g(x_{t})-\mathbf{E}_{(t-1)-s}g(x_{t})]^{2}\leq C\sum_{t=s+1}^{n}a_{s,t}^{2}\leq Cna_{s}^{2},
\end{align*}
where the final inequality follows since $a_{s,t}=a_{s}$ for $0\leq s\leq t-1$.
Deduce from (\ref{eq:decomp}) and the preceding that
\begin{align*}
\mathbf{E}\left|\frac{1}{n}\sum_{t=1}^{n}\left[g(x_{t})-\mathbf{E}_{0}g(x_{t})\right]\right| & \leq\frac{1}{n}\sum_{s=0}^{n-1}(\mathbf{E}M_{n,s}^{2})^{1/2}\leq\frac{C}{n^{1/2}}\sum_{s=0}^{n-1}\lvert a_{s}\rvert.
\end{align*}
Finally, note that if each $x_{t}$ were divided by a positive constant,
the preceding would hold with the r.h.s.\ divided by the same constant.
Hence by \textbf{LP5}
\[
\mathbf{E}\left|\frac{1}{n}\sum_{t=1}^{n}\left[g(\beta_{n}^{-1}x_{t}(n))-\mathbf{E}_{0}g(\beta_{n}^{-1}x_{t}(n))\right]\right|\leq\frac{C}{n^{1/2}\beta_{n}}\sum_{s=0}^{n-1}\lvert a_{s}\rvert\rightarrow0,
\]
which yields \textbf{HL4}.

\textbf{(iv).} Suppose \textbf{LP2(a)} holds for some $\lambda\geq2$.
By Theorem 2 of Whittle (1960), there exists a $C$ (depending only
$\lambda$ and $\mathbf{E}\lvert\xi_{1}\rvert^{\lambda}$) such that
\[
\mathbf{E}\left\vert \frac{x_{t}(n)}{\beta_{n}}\right\vert ^{\lambda}\leq C\left(\beta_{n}^{-2}\sum_{k=0}^{\infty}a_{k,t}^{2}(n)\right)^{\lambda/2}=\frac{C}{(\mathbf{E}\xi_{1}^{2})^{\lambda/2}}\left(\beta_{n}^{-2}\beta_{n,t}^{2}\right)^{\lambda/2}
\]
and thus \textbf{HL5(a)} follows by \textbf{LP7}.

Next, suppose \textbf{LP2(b)} holds. For any $\lambda\in(0,\infty)$,
we have
\[
\mathbf{E}e^{\lambda\beta_{n}^{-1}\left\vert x_{t}(n)\right\vert }\leq\mathbf{E}e^{\lambda\beta_{n}^{-1}x_{t}\left(n\right)}+\mathbf{E}e^{-\lambda\beta_{n}^{-1}x_{t}\left(n\right)}.
\]
We shall show that the first r.h.s.\ term is finite for sufficiently
small $\lambda$; the same argument delivers the bound for the second
term and thence \textbf{HL5(b)}. Since $\mathbf{E}\exp(\mu\xi_{1})<\infty$
for sufficiently small $\mu$, there exist $\varsigma_{\xi},b_{\xi}\in(0,\infty)$
such that $\mathbf{E}\exp(\mu\xi_{t})\leq\exp(\mu^{2}\varsigma_{\xi}^{2})$
for all $\lvert\mu\rvert<b_{\xi}$ (see e.g.\ Theorem~2.13 in Wainwright,
2019). In view of \textbf{LP6(b)}, we may choose $\lambda\in(0,\infty)$
such that
\[
\sup_{1\leq t\leq n}\sup_{k\geq0}\lambda\lvert\beta_{n}^{-1}a_{k,t}(n)\rvert<b_{\xi}
\]
for all $n$ sufficiently large. Hence by Fatou's lemma, for such
$n$,
\begin{align*}
\mathbf{E}e^{\beta_{n}^{-1}\lambda x_{t}(n)} & =\mathbf{E}\exp\left(\beta_{n}^{-1}\lambda\sum_{k=0}^{\infty}a_{k,t}(n)\xi_{t-k}\right)\\
 & \leq\liminf_{M\rightarrow\infty}\prod\limits _{k=0}^{M}\mathbf{E}\exp\left(\lambda\beta_{n}^{-1}a_{k,t}(n)\xi_{t-k}\right)\\
 & \leq\liminf_{M\rightarrow\infty}\prod\limits _{k=0}^{M}\exp\left(\lambda^{2}\beta_{n}^{-2}a_{k,t}^{2}(n)\varsigma_{\xi}^{2}\right)\\
 & =\exp\left(\lambda^{2}\beta_{n}^{-2}\sum_{k=0}^{\infty}a_{k,t}^{2}(n)\varsigma_{\xi}^{2}\right)=\exp(\lambda^{2}\beta_{n}^{-2}\beta_{n,t}^{2}\varsigma_{\xi}^{2}),
\end{align*}
where the final term is bounded uniformly over $1\leq t\leq n$ and
$n\geq n_{0}$ by \textbf{LP7}.
\end{proof}

\subsection{Lemmas for $I(1/2)$ and MI processes}

\label{subsec:procaux}

The proof of Theorem~\ref{MainLL} is now mostly a matter of verifying
that each of \textbf{FR} and \textbf{MI} (in conjunction with \textbf{INN})
imply \textbf{LP}, whereupon the result will follow by appeals to
Proposition~\ref{LPtoHL} and Theorem~\ref{MainThm}. The following
lemmas establish some key properties of $I(1/2)$ and MI processes
under these assumptions, which will be used in the proofs of each
of Theorems \ref{MainLL} and \ref{KernelLL}. Their proofs appear
in Appendix~\ref{subsec:FRMIlemmas}.

\begin{namedlemma}[COEF]\namedthmlabel{COEF} Suppose that \textbf{FR}
or \textbf{MI }holds. Then there exist $D_{1},D_{2}\in(0,\infty)$
and $n_{0},k_{0}\in\mathbb{N}$ such that for all $n\geq n_{0}$
\begin{enumerate}
\item $\sup_{1\leq t\leq n,k\in\mathbb{N}_{0}}\lvert a_{k,t}(n)\rvert\leq D_{2}$;
\end{enumerate}
and under \textbf{MI}, 
\begin{enumerate}[resume]
\item $\min_{k_{0}\leq k\leq\kappa_{n}\wedge(n-1)}\lvert a_{k}(n)\vert\geq D_{1}.$
\end{enumerate}
\end{namedlemma}

For the purposes of the next two lemmas, define
\begin{equation}
[\gamma_{n}^{2},\mathcal{V}_{\infty}]:=\left\{ \begin{alignedat}{4} & [L(n), & \  & 8\sigma_{\xi}^{2} &  & ] & \qquad & \text{under \textbf{FR1},}\\
 & [L(n), &  & \sigma_{\xi}^{2}({\textstyle \sum_{s=0}^{\infty}c_{s})^{2}} &  & ] &  & \text{under \textbf{FR2},}\\
 & [\kappa_{n}, &  & \sigma_{\xi}^{2}({\textstyle \sum_{s=0}^{\infty}c_{s})^{2}}/2 &  & ] &  & \text{under \textbf{MI}.}
\end{alignedat}
\right.\label{eq:gamma_n1}
\end{equation}
By Proposition 1.5.9a in BGT, $L(n)=\int_{1}^{n}x^{-1}\ell^{2}(x)dx$
is a slowly varying function, a fact that shall be used freely throughout
the following. Recall also that $L(n)\rightarrow\infty$ under \textbf{FR}.
\textbf{MI} entails that $\kappa_{n}$ is regularly varying with index
$\alpha\in[0,1)$, and thus in all cases, $\gamma_{n}$ is regularly
varying with index $\alpha\in[0,1/2)$.

\begin{namedlemma}[LVAR]\namedthmlabel{LVAR} Suppose \textbf{INN}
and either \textbf{FR} or \textbf{MI} holds. Let $t_{n}\in[\lfloor rn\rfloor,n]\cap\mathbb{N}$,
where $0<r\leq1$. Then
\begin{enumerate}
\item Under \textbf{FR1}:
\[
\begin{bmatrix}Var[\gamma_{n}^{-1}x_{t_{n}}^{+}(n)] & Var[\gamma_{n}^{-1}x_{t_{n}}^{-}(n)]\end{bmatrix}\rightarrow\begin{bmatrix}\mathcal{V}_{\infty}/2 & \mathcal{V}_{\infty}/2\end{bmatrix}
\]
and for any $s\in(0,r]$,
\[
Var\{\gamma_{n}^{-1}[x_{\lfloor rn\rfloor}^{-}(n)-x_{\lfloor sn\rfloor}^{-}(n)]\}\rightarrow0.
\]
\item Under \textbf{FR2} or \textbf{MI}:
\[
Var[\gamma_{n}^{-1}x_{t_{n}}(n)]=Var[\gamma_{n}^{-1}x_{t_{n}}^{+}(n)]+o(1)\rightarrow\mathcal{V}_{\infty}
\]
\end{enumerate}
\end{namedlemma}

\begin{namedlemma}[CLT]\namedthmlabel{CLT} Under the assumptions
of Lemma~\ref{LVAR},
\[
\gamma_{n}^{-1}\left[x_{t_{n}}^{+}(n),x_{t_{n}}^{-}(n)\right]\overset{d}{\rightarrow}N(0,\Sigma),
\]
where $\Sigma=\mathrm{diag}\{\mathcal{V}_{\infty}/2,\mathcal{V}_{\infty}/2\}$
under \textbf{FR1}, and $\Sigma=\mathrm{diag}\{\mathcal{V}_{\infty},0\}$
under \textbf{FR2} or \textbf{MI}. \end{namedlemma}

\subsection{Proofs of Theorems~3.2 and 3.4}

\label{subsec:lowleveltheorems}

We first state a preliminary lemma on regularly varying sequences,
whose proof is given in Appendix~\ref{subsec:technicalproofs} below. 
\begin{lem}
\label{SV} Let $\{\varphi_{j}\}_{j\in\mathbb{N}}$ be a real sequence,
and $\varsigma$ a positive-valued, slowly varying function, that
is locally integrable on $[1,\infty)$.
\begin{enumerate}
\item Suppose $\varphi_{j}\sim j^{l}\varsigma(j)$ for $l>-1$. Then for
all $0\leq s<r<\infty$ as $n\rightarrow\infty$ 
\[
\frac{1}{n^{1+l}\varsigma(n)}\sum_{j=\lfloor ns\rfloor+1}^{\lfloor nr\rfloor}\varphi_{j}\rightarrow\int_{s}^{r}x^{l}dx.
\]
\item Suppose $\varphi_{j}\sim j^{l}\varsigma(j)$ for $l<-1$. Then as
$n\rightarrow\infty$ 
\[
\frac{1}{n^{1+l}\varsigma(n)}\sum_{j=n}^{\infty}\varphi_{j}\rightarrow\int_{1}^{\infty}x^{l}dx.
\]
\item Suppose $S(x):=\int_{1}^{x}u^{-1}\varsigma(u)du$ is such that $S(x)\rightarrow\infty$
as $x\rightarrow\infty$. Then for any $n_{0}\geq1$, as $n\rightarrow\infty$
\[
\sum_{j=n_{0}}^{n}j^{-1}\varsigma(j)\sim S(n).
\]
\end{enumerate}
\end{lem}
\begin{proof}[Proof of Theorem~\ref{MainLL}]
 With the aid of Proposition~\ref{LPtoHL}, we show that the conditions
of\textbf{ }Theorem~\ref{MainThm} are satisfied for $X_{t}(n)=x_{t}(n)$.
We first verify Assumptions~\textbf{HL0--6}, before turning to the
other conditions of the theorem.

\textbf{HL0--1.} Under either of \textbf{FR }or \textbf{MI}, $x_{t}(n)$
can be written in the form (\ref{model01}); that $\beta_{n,t}<\infty$
for each $n,t\in\mathbb{N}$ follows from $\mathbf{E}v_{1}^{2}=\sigma_{\xi}^{2}\sum_{k=0}^{\infty}c_{k}^{2}<\infty$.
Recall also from (\ref{linproc}) and (\ref{lincoefs}) that we may
write $x_{t}(n)=\sum_{k=0}^{\infty}a_{k,t}(n)\xi_{t-k}$, where $a_{k,t}(n)=\sum_{j=0}^{(t-1)\wedge k}\phi_{j}(n)c_{k-j}$.
Under \textbf{FR}, $\phi_{0}(n)=\phi_{0}\neq0$, while under \textbf{MI},
$\phi_{0}(n)=1-\kappa_{n}^{-1}\rightarrow1$ and has $\phi_{0}(n)>0$
for all $n$, since $\kappa_{n}>1$. Now let $k^{\ast}$ denote the
smallest $k\in\mathbb{N}$ such that $c_{k}\neq0$. Then
\[
\beta_{n,t}^{2}=\sum_{k=0}^{\infty}a_{k,t}^{2}(n)\geq a_{k^{\ast},t}^{2}(n)=\phi_{0}^{2}(n)c_{k^{\ast}}^{2}>0
\]
 for all $n,t\in\mathbb{R}$. Thus \textbf{LP1} holds, whence \textbf{HL0--1}
follows by Proposition~\ref{LPtoHL}.

\textbf{HL2.} Once we have verified \textbf{LP6}, by Proposition~\ref{LPtoHL}(i)
this will hold with $(X^{+},X^{-})\sim N[0,\mathrm{diag}\{\sigma_{+}^{2},\sigma_{-}^{2}\}]$.
Lemma~\ref{LVAR} implies that \textbf{LP6(a)} holds with $(\sigma_{+}^{2},\sigma_{-}^{2})=(1/2,1/2)$
under \textbf{FR1} and $(\sigma_{+}^{2},\sigma_{-}^{2})=(1,0)$ under
\textbf{FR2/MI}. It further implies that $\beta_{n}\asymp\gamma_{n}\rightarrow\infty$,
and thus \textbf{LP6(b)} follows by Lemma~\ref{COEF}. With respect
to \textbf{LP6(c)}, it follows from (\ref{eq:varneg}) above that
\[
\beta_{n}^{-2}\sum_{l=-\infty}^{0}\left[a_{t_{n}-l,t_{n}}^{-}(n)-a_{n-l,n}^{-}(n)\right]^{2}\sigma_{\xi}^{2}\asymp\gamma_{n}^{-2}Var\left[x_{t_{n}}^{-}(n)-x_{n}^{-}(n)\right]
\]
for $t_{n}\in\{\lfloor nr\rfloor,\ldots,n\}$ for some $r\in(0,1]$.
By Lemma~\ref{LVAR}, the r.h.s.\ is immediately $o(1)$ under \textbf{FR1}.
Under \textbf{FR2} or \textbf{MI}, that same result implies $\gamma_{n}^{-2}Var[x_{t_{n}}^{-}(n)]=o(1)$,
and thus
\[
\gamma_{n}^{-2}Var\left[x_{t_{n}}^{-}(n)-x_{n}^{-}(n)\right]\leq2\gamma_{n}^{-2}\left\{ Var[x_{t_{n}}^{-}(n)]+Var[x_{n}^{-}(n)]\right\} =o(1).
\]

\textbf{HL3 (for some $\boldsymbol{t_{0}\in\mathbb{N}}$).} By Proposition~\ref{LPtoHL}(ii),
it suffices to verify\textbf{ LP3}. Consider \textbf{FR} first: in
this case, $x_{t}(n)$ does not depend on $n$. Thus $\beta_{n,t}^{+}=\beta_{t}^{+}\asymp\gamma_{t}$
and $\beta_{t}\asymp\gamma_{t}$ as $t\rightarrow\infty$ by Lemma~\ref{LVAR},
whence \textbf{LP3(a)} holds. By\textbf{ }Lemma~\ref{COEF}, $\max_{0\leq k\leq t-1}\lvert a_{k}(n)\rvert\leq D_{2}$,
which together with $\gamma_{t}=L^{1/2}(t)\rightarrow\infty$ delivers
\textbf{LP3(b)}.

Next suppose \textbf{MI} holds. By Lemma~\ref{COEF}, we have
\[
(\beta_{n,t}^{+})^{2}=\sum_{k=0}^{t-1}a_{k}^{2}(n)\geq D_{1}[(t-1)\wedge\kappa_{n}-k_{0}].
\]
Since $C_{1}:=\sup_{n\geq1,1\leq t\leq n}\kappa_{t}/\kappa_{n}<\infty$
by Assumption~\textbf{MI}, where trivially $C_{1}\geq1$, we have
\[
(t-1)\wedge\kappa_{n}-k_{0}\geq(t-1)\wedge C_{1}^{-1}\kappa_{t}-k_{0}\asymp\kappa_{t},
\]
since $\kappa_{t}=o(t)$ as $t\rightarrow\infty$. Thus there exists
a $t_{0}\in\mathbb{N}$ and $C\in(0,\infty)$ such that $\beta_{n,t}^{+}\geq C\kappa_{t}^{1/2}$
for all $t_{0}\leq t\leq n$. Since $\beta_{t}\asymp\gamma_{t}=\kappa_{t}^{1/2}\rightarrow\infty$
by Lemma~\ref{LVAR}, and $\max_{0\leq k\leq t}\lvert a_{k}(n)\rvert\leq D_{2}$
by Lemma~\ref{COEF}, both parts of \textbf{LP3} follow immediately.

\textbf{HL4.} By Proposition~\ref{LPtoHL}(iii), we need to only
to verify \textbf{LP5}. First suppose \textbf{FR1} holds. Then recalling
(\ref{lincoefs}) above,
\[
\lvert a_{k}(n)\rvert=\left|\sum_{j=0}^{k}\phi_{j}c_{k-j}\right|\leq\left|\sum_{j=0}^{k}c_{j}\right|=_{(1)}\left|\sum_{j=k+1}^{\infty}c_{j}\right|\asymp_{(2)}k^{-1/2}\ell(k)
\]
as $k\rightarrow\infty$, where $=_{(1)}$ follows by $\sum_{k=0}^{\infty}c_{k}=0$,
and $\asymp_{(2)}$ by Lemma~\ref{SV}(ii), since $c_{j}\sim k^{-3/2}\ell(k)$.
Hence, by Lemma~\ref{SV}(i) 
\[
\sum_{k=0}^{n}\lvert a_{k}(n)\rvert\leq Cn^{1/2}\ell(n)
\]
under \textbf{FR1}. Alternatively, if \textbf{FR2} holds, then we
also have
\[
\sum_{k=0}^{n}\lvert a_{k}(n)\rvert\leq\sum_{k=0}^{n}\sum_{j=0}^{k}\left|\phi_{j}c_{k-j}\right|\leq\sum_{j=0}^{n}\lvert\phi_{j}\rvert\sum_{k=0}^{\infty}\lvert c_{k}\rvert\leq Cn^{1/2}\ell(n)
\]
by $\sum_{k=0}^{\infty}\lvert c_{k}\rvert<\infty$ and Lemma~\ref{SV}(i).
Thus under \textbf{FR}, Lemma~\ref{LVAR} implies that
\[
\frac{1}{n^{1/2}\beta_{n}}\sum_{k=0}^{n}\lvert a_{k}(n)\rvert\leq C\frac{n^{1/2}\ell(n)}{n^{1/2}L^{1/2}(n)}=C\left(\frac{\ell^{2}(n)}{\int_{1}^{n}x^{-1}\ell^{2}(x)\mathrm{d}x}\right)^{1/2}\rightarrow0
\]
where the final convergence follows by Proposition 1.5.9a in BGT.

Next suppose \textbf{MI} holds. In this case $a_{k}(n)=\sum_{j=0}^{k}\rho_{n}^{j}c_{k-j}$
for $\rho_{n}:=1-\kappa_{n}^{-1}$, and so for all $n$ sufficiently
large that $\rho_{n}>0$,
\begin{align*}
\sum_{k=0}^{n}\lvert a_{k}(n)\rvert & \leq\sum_{k=0}^{n}\sum_{j=0}^{k}\lvert\rho_{n}^{j}c_{k-j}\vert\leq\sum_{j=0}^{n}\lvert\rho_{n}^{j}\rvert\sum_{k=j}^{n}\lvert c_{k-j}\rvert\leq_{(1)}\frac{C}{1-\rho_{n}}=C\kappa_{n}
\end{align*}
where $\leq_{(1)}$ holds since $\{c_{i}\}$ is absolutely summable.
Hence by Lemma~\ref{LVAR}
\[
\frac{1}{n^{1/2}\beta_{n}}\sum_{k=0}^{n}\lvert a_{k}(n)\rvert\leq\frac{C\kappa_{n}}{n^{1/2}\kappa_{n}^{1/2}}=C\left(\frac{\kappa_{n}}{n}\right)^{1/2}\rightarrow0.
\]

\textbf{HL5.} In view of Proposition~\ref{LPtoHL}(iv), having already
verified \textbf{LP6(b)}, we need to verify \textbf{LP7}. Suppose
first that \textbf{FR} holds. Then since $x_{t}(n)$ does not depend
on $n$, we have by Lemma~\ref{LVAR} that, as $t\rightarrow\infty$
\[
\beta_{n,t}^{2}=\beta_{t}^{2}\asymp L(t)=\int_{1}^{t}x^{-1}\ell^{2}(x)\mathrm{d}x
\]
which is clearly monotone increasing. Thus there exist $C_{1},C_{2}\in(0,\infty)$
such that
\[
\beta_{n,t}=\beta_{t}\leq C_{1}L(t)\leq C_{1}L(n)\leq C_{2}\beta_{n}
\]
for all $t\leq n$, whence \textbf{LP7} holds trivially.

Suppose next that \textbf{MI} holds. In this case, $x_{t}(n)=\sum_{k=0}^{\infty}a_{k,t}(n)\xi_{t-k}$
with $a_{k,t}=\sum_{j=0}^{(t-1)\wedge k}\phi_{j}(n)c_{k-j}$, where
$\sum_{i=0}^{\infty}\lvert c_{i}\rvert<\infty$ and $\phi_{j}(n)=(1-\kappa_{n}^{-1})^{j}>0$
for all $n$ sufficiently large. Now define a new process $\{\overline{x}_{t}(n)\}$
by $\overline{x}_{t}(n)=\sum_{k=0}^{\infty}\overline{a}_{k,t}(n)\xi_{t-k}$,
where $\overline{a}_{k,t}=\sum_{j=0}^{(t-1)\wedge k}\phi_{j}(n)\lvert c_{k-j}\rvert$.
Then $\{\overline{x}_{t}(n)\}$ also satisfies \textbf{MI}, and it
is easily verified that
\[
\lvert a_{k,t}\rvert\leq\overline{a}_{k,t}=\sum_{k=0}^{(t-1)\wedge k}\phi_{j}(n)\lvert c_{k-j}\rvert\leq\sum_{k=0}^{(n-1)\wedge k}\phi_{j}(n)\lvert c_{k-j}\rvert=\overline{a}_{n,t}
\]
for all $1\leq t\leq n$, whence
\[
\beta_{n,t}^{2}=\sum_{k=0}^{\infty}a_{k,t}^{2}\leq\sum_{k=0}^{\infty}\overline{a}_{k,t}^{2}=:\overline{\beta}_{n,t}\leq\overline{\beta}_{n,n}=\overline{\beta}_{n}.
\]
It follows from\textbf{ }Lemma~\ref{LVAR} that $\overline{\beta}_{n}^{2}\asymp\kappa_{n}\asymp\beta_{n}^{2}$,
whence
\[
\beta_{n}^{-2}\beta_{n,t}^{2}\leq\beta_{n}^{-2}\overline{\beta}_{n}^{2}=O(1).
\]

Thus by Proposition~\ref{LPtoHL}(iv), either \textbf{HL5(a)} or
\textbf{(b)} holds depending on the assumptions made on $\{\xi_{t}\}$
-- i.e.\ depending on whether condition 2(a) or 2(b) of Theorem~\ref{MainLL}
is maintained.

\textbf{HL6.} By Lemma~\ref{LVAR}, $\beta_{n}\asymp\gamma_{n}$,
which is regularly varying with index $\alpha\in[0,1)$. Thus \textbf{HL6}
follows by Lemma~\ref{SV}(i).

We have thus verified that \textbf{HL0--6} hold under our assumptions,
with $(X^{+},X^{-})\sim N[0,\mathrm{diag}\{\sigma_{+}^{2},\sigma_{-}^{2}\}]$.
The conclusion of Theorem~\ref{MainThm} accordingly holds for bounded
$f$, with
\[
\int_{\mathbb{R}}f(x+X^{-})\Phi_{X^{+}}(x)dx=\int_{\mathbb{R}}f(x)\Phi_{X^{+}}(x-X^{-})dx=\int_{\mathbb{R}}f(x)\varrho(x)dx
\]
where $\varrho$ is as defined in (\ref{eq:integrals}).

For more general $f$, we need to verify conditions (i)--(iii) of
Theorem~\ref{MainThm}. With respect to condition~(i), since is
assumed that $f$ is locally integrable and satisfies $\lvert f(x)\lvert=O(e^{\lambda^{\prime}\lvert x\lvert})$
for some $\lambda^{\prime}<\infty$ as $\lvert x\lvert\rightarrow\infty$,
we thus have
\[
\int_{\mathbb{R}}\left\vert f(x+y)\right\vert \Phi_{X^{+}}(x)dx\leq C\int_{\mathbb{R}}\left\vert f(x)\right\vert e^{-(x+y)^{2}/2\sigma_{+}^{2}}dx<\infty
\]
for every $y\in\mathbb{R}$, as required.

Condition~(iii) of Theorem~\ref{MainThm} is implied by condition~2
of Theorem~\ref{MainLL}, and the verification of \textbf{HL5} given
above.

It thus remains to verify condition~(ii) of Theorem~\ref{MainThm}.
Under condition 1(a) of Theorem~\ref{MainLL}, there is nothing to
left prove. Suppose instead that condition 1(b) of of Theorem~\ref{MainLL}
is assumed, i.e.\ that \textbf{INN} holds with $\theta=1$. In that
case, we need to verify that \textbf{HL3} holds with $t_{0}=1$. By
Proposition~\ref{LPtoHL}(ii)(b), it suffices to verify \textbf{LP4},
since we have already verified \textbf{LP3}. Recall (\ref{linproc})
and (\ref{lincoefs}) above, and let $k^{\ast}$ denote the smallest
$k\in\mathbb{N}$ such that $c_{k^{\ast}}\neq0$. Then 
\[
\lvert a_{k^{\ast},t}(n)\rvert=\lvert\phi_{0}(n)c_{k^{\ast}}\rvert\geq\lvert c_{k^{\ast}}\rvert\inf_{n\in\mathbb{N}}\lvert\phi_{0}(n)\rvert>0
\]
where the final inequality follows from arguments given in the course
of verifying \textbf{HL0--1} above. Thus \textbf{LP4} holds.
\end{proof}
\begin{proof}[Proof of Theorem~\ref{KernelLL}]
 We verify the assumptions of\textbf{ }Theorem~\ref{MainKernel}
are satisfied for $X_{t}(n)=x_{t}(n)$.

\textbf{HL0--2, HL4, HL6.} These follow by the arguments given in
the proof of Theorem \ref{MainLL}.

\textbf{HL7.} We first note that the proof of Proposition~\ref{LPtoHL}(ii),
together with the arguments used to verify \textbf{HL3} in the proof
of Theorem~\ref{MainLL}, show that $\{\beta_{t}^{-1}x_{t}^{+}(n)\}$
has a density $\mathcal{D}_{t,n}^{+}$ that is uniformly bounded over
all $n$ and $t_{0}\leq t\leq n$ for some $t_{0}\in\mathbb{N}$.
Now let $0\leq s<t\leq n$ and note from (\ref{model01b}) that we
may write 
\[
x_{t}=\sum_{k=0}^{(t-s)-1}a_{k}\xi_{t-k}+\sum_{k=t-s}^{\infty}a_{k,t}\xi_{t-k}=:x_{s+1,t}+x_{s,t}^{\prime}
\]
where $x_{s,t}^{\prime}$ is $\mathcal{F}_{s}$-measurable, and $x_{s+1,t}$
is independent of $\mathcal{F}_{s}$. Since
\[
x_{s+1,t}=\sum_{k=0}^{(t-s)-1}a_{k}\xi_{t-k}\overset{d}{=}\sum_{k=0}^{(t-s)-1}a_{k}\xi_{(t-s)-k}=x_{t-s}^{+}
\]
we must have that the density $\mathcal{D}_{t,s,n}$ of $\{\beta_{t-s}^{-1}[x_{t}(n)-x_{s}(n)]\}$,
conditional on $\mathcal{F}_{s}$, satisfies
\begin{equation}
\mathcal{D}_{t,s,n}(x)=\mathcal{D}_{t-s,n}^{+}(x-\chi_{s,t}(n))\label{eq:densityrelation}
\end{equation}
where $\chi_{s,t}(n):=\beta_{t-s}^{-1}[x_{s+1,t}(n)-x_{s}(n)]$ is
$\mathcal{F}_{s}$-measurable (with the convention that $x_{0}(n):=0$).
It follows that
\begin{multline*}
\sup_{0\leq s<t\leq n,t-s\geq t_{0}}\sup_{x\in\mathbb{R}}\lvert\mathcal{D}_{t,s,n}(x)\rvert\leq\sup_{0\leq s<t\leq n,t-s\geq t_{0}}\sup_{x\in\mathbb{R}}\lvert\mathcal{D}_{t-s,n}^{+}(x)\rvert\\
\leq\sup_{t_{0}\leq r\leq n}\sup_{x\in\mathbb{R}}\lvert\mathcal{D}_{r,n}^{+}(x)\rvert<\infty,
\end{multline*}
by the uniform boundedness of $\mathcal{D}_{r,n}^{+}$ noted above.

\textbf{HL8.} Let $\mathcal{D}_{r,n}^{\ast}$ denote the density of
$\gamma_{n}^{-1}x_{r}^{+}(n)=(\gamma_{n}^{-1}\beta_{r})\cdot\beta_{r}^{-1}x_{r}^{+}(n)$,
where $\{\gamma_{n}\}$ is as defined in (\ref{eq:gamma_n1}), and
$\delta_{\eta}:=q_{0}\eta^{q_{1}}$ for $\eta,q_{0},q_{1}>0$. Let
$s,t\in\Omega_{n}(\eta)$, and $r:=t-s$. In view of (\ref{eq:densityrelation}),
\begin{align*}
\sup_{\lvert x\rvert\leq\delta_{\eta}}\lvert\mathcal{D}_{t,s,n}(x)-\mathcal{D}_{t,s,n}(0) & \rvert\leq\sup_{y\in\mathbb{R}}\sup_{\lvert x\rvert\leq\delta_{\eta}}\lvert\mathcal{D}_{r,n}^{+}(y+x)-\mathcal{D}_{r,n}^{+}(y)\rvert\\
 & \leq(\beta_{r}^{-1}\gamma_{n})\sup_{y\in\mathbb{R}}\sup_{\lvert x\rvert\leq\delta_{\eta}(\beta_{r}^{-1}\gamma_{n})}\lvert\mathcal{D}_{r,n}^{\ast}(y+x)-\mathcal{D}_{r,n}^{\ast}(y)\rvert.
\end{align*}
By the definition of $\Omega_{n}(\eta)$, \textbf{HL8} will hold if
we can show that when $r=r_{n}$ for any sequence $\{r_{n}\}$ with
$r_{n}\in[\lfloor n\eta\rfloor,n]$, the r.h.s.\ converges to zero
as $n\rightarrow\infty$ and then $\eta\rightarrow0$.

By Lemma~\ref{LVAR}, $\beta_{r_{n}}^{-1}\gamma_{n}\asymp\gamma_{r_{n}}^{-1}\gamma_{n}\asymp1$,
since $\{\gamma_{n}\}$ is regularly varying. In particular, $\{\beta_{r_{n}}^{-1}\gamma_{n}\}$
is bounded above, and so (by redefinition of $q_{0}$) the result
will follow if we can prove that 
\begin{equation}
\sup_{y\in\mathbb{R}}\sup_{\lvert x\rvert\leq\delta_{\eta}}\lvert\mathcal{D}_{r_{n},n}^{\ast}(y+x)-\mathcal{D}_{r_{n},n}^{\ast}(y)\rvert\rightarrow0\label{eq:densast}
\end{equation}
as $n\rightarrow\infty$ and then $\eta\rightarrow0$. By Lemma~\ref{CLT},
$\gamma_{n}^{-1}x_{r_{n}}^{+}(n)\overset{d}{\rightarrow}N[0,\sigma_{+}^{2}]$
for some $\sigma_{+}^{2}>0$. By the argument given in the proof of
Corollary~2.2 in Wang and Phillips (2009), it therefore suffices
for (\ref{eq:densast}) to show that the characteristic functions
of $\{\gamma_{n}^{-1}x_{r_{n}}^{+}(n)\}$ are uniformly integrable:
since in this case, $\mathcal{D}_{r_{n},n}^{\ast}$ converges uniformly
to the $N[0,\sigma_{+}^{2}]$ density. To that end, note that $\gamma_{n}^{-1}x_{r_{n}}^{+}(n)$
is a linear process with variance converging to $\sigma_{+}^{2}>0$
(as $n\rightarrow\infty$) by Lemma~\ref{LVAR}, and coefficients
$\{\gamma_{n}^{-1}a_{k}(n)\}_{k=0}^{r_{n}-1}$ with $\max_{0\leq k\leq r_{n}-1}\gamma_{n}^{-1}\lvert a_{k}(n)\vert\rightarrow0$
by Lemma~\ref{COEF}. Thus the required uniform integrability follows
by Lemma \ref{lem:UICF}(i).

\textbf{HL9.} By Lemma~\ref{LVAR}, $\beta_{n}\asymp\gamma_{n}$,
which is regularly varying with index $\alpha\in[0,1/2)$ and so can
be written as $\gamma_{n}=n^{\alpha}\varsigma(n)$, for $\varsigma$
a positive-valued, slowly varying function. Thus \textbf{HL9(a)--(c)}
follow immediately from Lemma~\ref{SV}(i). For \textbf{HL9(d)},
we note that
\begin{align*}
\beta_{n}^{-1}\inf_{(t,s)\in\Omega_{n}(\eta)}\beta_{t-s} & >C\gamma_{n}^{-1}\inf_{(t,s)\in\Omega_{n}(\eta)}\gamma_{t-s}\\
 & =Cn^{-\alpha}\varsigma^{-1}(n)\inf_{(t,s)\in\Omega_{n}(\eta)}(t-s)^{\alpha}\varsigma(t-s)
\end{align*}
for some $C>0$. Since $(t,s)\in\Omega_{n}(\eta)$ implies $t-s>\lfloor\eta n\rfloor$,
the r.h.s.\ can be bounded below (up to a positive constant) by
\[
(\eta/2)^{\alpha}\inf_{\lambda\in[\eta/2,1]}\frac{\varsigma(\lambda n)}{\varsigma(n)}=(\eta/2)^{\alpha}(1+o(1))
\]
where the equality follows by Theorem~1.2.1 in BGT. Thus \textbf{HL9(d)}
holds. Finally, since $\beta_{n}\asymp\gamma_{n}$, and both $\beta_{n}>0$
and $\gamma_{n}>0$ for all $n$, there exist $\underline{C},\overline{C}\in(0,\infty)$
such that $\beta_{n}\in[\underline{C}\gamma_{n},\overline{C}\gamma_{n}]$
for all $n$. Thus
\[
\sup_{n\geq1}\sup_{1\leq t\leq n}\beta_{n}^{-1}\beta_{t}\leq\underline{C}^{-1}\overline{C}\sup_{n\geq1}\sup_{1\leq t\leq n}\gamma_{n}^{-1}\gamma_{t}.
\]
From (\ref{eq:gamma_n1}), under \textbf{FR} $\gamma_{n}=L(n)$ is
increasing, so $\sup_{n\geq1}\sup_{1\leq t\leq n}\gamma_{n}^{-1}\gamma_{t}\leq1$
trivially. Under \textbf{MI}, $\gamma_{n}=\kappa_{t}$, for which
$\sup_{n\geq1}\sup_{1\leq t\leq n}\kappa_{n}^{-1}\kappa_{t}<\infty$
is a maintained assumption. Thus \textbf{HL9(e)} holds.
\end{proof}
\pagebreak{}

\section{Proofs auxiliary to Appendix B}

\label{sec:auiliaryapp}

This appendix provides proofs of the lemmas stated in Appendix~\ref{sec:LLapp}.

\subsection{Proofs of Lemmas COEF, LVAR and CLT}

\label{subsec:FRMIlemmas}
\begin{proof}[Proof of Lemma~\ref{COEF}]
 \textbf{(i).} In all cases, $\lvert\phi_{j}(n)\rvert<C$ uniformly
over $j,n\in\mathbb{N}$, and $\{c_{s}\}$ is absolutely summable.
Hence from (\ref{lincoefs}),
\[
\lvert a_{k,t}(n)\rvert=\left|\sum_{j=0}^{(t-1)\wedge k}\phi_{j}c_{k-j}\right|\leq C\sum_{j=0}^{\infty}\vert c_{j}\lvert=:D_{2}<\infty.
\]

\textbf{(ii).} Under \textbf{MI}, we have $\phi_{j}(n)=\rho_{n}^{j}$
for $\rho_{n}:=1-\kappa_{n}^{-1}$, and thus
\begin{equation}
a_{k}(n)=\sum_{j=0}^{k}\rho_{n}^{k-j}c_{j}=\rho_{n}^{k}\left[\sum_{j=0}^{k}c_{j}+\sum_{j=0}^{k}(\rho_{n}^{-j}-1)c_{j}\right].\label{eq:aknMI}
\end{equation}
Since $\rho_{n}\in(0,1)$, we have for all $0\leq k\leq\lfloor\kappa_{n}\rfloor$
that 
\begin{equation}
1\geq\rho_{n}^{k}\geq\rho_{n}^{\kappa_{n}}=(1-\kappa_{n}^{-1})^{\kappa_{n}}\rightarrow e^{-1}\label{eq:rhopower}
\end{equation}
from which follows that $\rho_{n}^{k}\in(e^{-1}/2,1]$ for all $k\in\{0,\ldots,\lfloor\kappa_{n}\rfloor\}$,
for all $n\in\mathbb{N}$.

Thus to bound $a_{k}(n)$ away from zero, we need only to bound the
bracketed term in (\ref{eq:aknMI}) away from zero. To that end, note
that for all $0\leq k\leq\lfloor\kappa_{n}\rfloor$
\[
\left|\sum_{j=0}^{k}(\rho_{n}^{-j}-1)c_{j}\right|\leq\sum_{j=0}^{\infty}\lvert\rho_{n}^{-j}-1\rvert\lvert c_{j}\rvert1\{j\leq\kappa_{n}\}\rightarrow0
\]
as $n\rightarrow\infty$ by the dominated convergence theorem, since
$\lvert\rho_{n}^{-j}-1\rvert\lvert c_{j}\rvert1\{j\leq\kappa_{n}\}\rightarrow0$
for each $j\in\mathbb{R}$, and by an analogous argument to \prettyref{eq:rhopower}
is (for sufficiently large $n$) bounded above by $\lvert c_{j}\rvert(1+2e)$,
which is summable. Since $\sum_{j=0}^{\infty}c_{j}\neq0$, deduce
that there exist $n_{0},k_{0}\in\mathbb{N}$ such that
\[
\left|\sum_{j=0}^{k}(\rho_{n}^{-j}-1)c_{j}\right|\leq\frac{1}{2}\sum_{j=0}^{k}c_{j}
\]
for all $n\geq n_{0}$ and $k\in\{k_{0},\ldots,\lfloor\kappa_{n}\rfloor\}$.
\end{proof}
\begin{proof}[Proof of Lemma \ref{LVAR}]
 For convenience set $\sigma_{\xi}^{2}=1$. Under \textbf{FR} it
is sufficient to prove the result with $t_{n}=n$; the result in the
general case follows straightforwardly from the fact that $L(t_{n})/L(n)\rightarrow1$
as $n\rightarrow\infty$, since $L(n)$ is slowly varying.

\textbf{FR1.} Suppose we show that
\begin{align}
\sum_{k=0}^{n-1}\left(\sum_{j=0}^{k}c_{j}\right)^{2} & \sim4L(n) & \sum_{k=0}^{n-1}\left(\sum_{j=k+n+1}^{\infty}c_{j}\right)^{2} & =o[L(n)]\label{eq:FR1bnd1}
\end{align}
and that for any sequences $\{\delta_{1,n}\}$, $\{\delta_{2,n}\}$
with $\delta_{1,n}\asymp n\asymp\delta_{2,n}$,
\begin{equation}
\sum_{k=\delta_{1,n}}^{\infty}\left(\sum_{j=k+1}^{k+\delta_{2,n}}c_{j}\right)^{2}=o[L(n)].\label{eq:FR1bnd2}
\end{equation}
Since $\gamma_{n}^{2}=L(n)$ and $a_{k,t}=\sum_{j=0}^{(t-1)\wedge k}c_{k-j}$,
it follows from (\ref{eq:FR1bnd1}) that
\[
\gamma_{n}^{-2}Var(x_{n}^{+})=\gamma_{n}^{-2}\sum_{k=0}^{n-1}\left(\sum_{j=0}^{k}c_{k-j}\right)^{2}=\gamma_{n}^{-2}\sum_{k=0}^{n-1}\left(\sum_{j=0}^{k}c_{j}\right)^{2}\rightarrow4.
\]
Further, note that we may write
\begin{align}
Var(x_{n}^{-}) & =\sum_{k=n}^{\infty}\left(\sum_{j=0}^{n-1}c_{k-j}\right)^{2}=\sum_{k=0}^{\infty}\left(\sum_{j=k+1}^{k+n}c_{j}\right)^{2}\nonumber \\
 & =\sum_{k=0}^{n-1}\left(\sum_{j=k+1}^{k+n}c_{j}\right)^{2}+\sum_{k=n}^{\infty}\left(\sum_{j=k+1}^{k+n}c_{j}\right)^{2}.\label{eq:FR1xdec}
\end{align}
The second r.h.s.\ term is of the same form as (\ref{eq:FR1bnd2})
with $\delta_{1,n}=\delta_{2,n}=n$, and so is $o[L(n)]$. The first
r.h.s.\ term expands as
\begin{align*}
\sum_{k=0}^{n-1}\left(\sum_{j=k+1}^{k+n}c_{j}\right)^{2} & =\sum_{k=0}^{n-1}\left(\sum_{j=k+1}^{\infty}c_{j}-\sum_{j=k+n+1}^{\infty}c_{j}\right)^{2}\\
 & =\sum_{k=0}^{n-1}\left(\sum_{j=0}^{k}c_{j}+\sum_{j=k+n+1}^{\infty}c_{j}\right)^{2},
\end{align*}
where we have used the fact that $\sum_{j=0}^{\infty}c_{j}=0$, and
so by (\ref{eq:FR1bnd1}) and the Cauchy-Schwarz inequality,
\begin{equation}
\sum_{k=0}^{n-1}\left(\sum_{j=k+1}^{k+n}c_{j}\right)^{2}=\sum_{k=0}^{n-1}\left(\sum_{j=0}^{k}c_{j}\right)^{2}+o[L(n)].\label{eq:FR1xdec2}
\end{equation}
Thus it follows from (\ref{eq:FR1xdec}) and (\ref{eq:FR1xdec2})
that $\gamma_{n}^{-2}Var(x_{n}^{-})\rightarrow4$, as required. Finally,
noting that
\[
x_{t}^{-}=\sum_{k=t}^{\infty}a_{k,t}^{-}\xi_{t-k}=\sum_{k=t}^{\infty}\sum_{j=0}^{t-1}c_{k-j}\xi_{t-k}=\sum_{i=0}^{\infty}\sum_{j=0}^{t-1}c_{i+t-j}\xi_{-i}=\sum_{i=0}^{\infty}\sum_{l=1}^{t}c_{i+l}\xi_{-i}
\]
we have for $0<s<r\leq1$ that
\begin{align*}
Var(x_{\lfloor nr\rfloor}^{-}-x_{\lfloor ns\rfloor}^{-}) & =\mathbf{E}\left(\sum_{i=0}^{\infty}\sum_{l=1}^{\lfloor nr\rfloor}c_{i+l}\xi_{-i}-\sum_{i=0}^{\infty}\sum_{l=1}^{\lfloor ns\rfloor}c_{i+l}\xi_{-i}\right)^{2}\\
 & =\mathbf{E}\left(\sum_{i=0}^{\infty}\sum_{l=\lfloor ns\rfloor+1}^{\lfloor nr\rfloor}c_{i+l}\xi_{-i}\right)^{2}\\
 & =\sum_{i=0}^{\infty}\left(\sum_{l=\lfloor ns\rfloor+1}^{\lfloor nr\rfloor}c_{i+l}\right)^{2}=\sum_{i=\lfloor ns\rfloor}^{\infty}\left(\sum_{l=i+1}^{i+\lfloor nr\rfloor-\lfloor ns\rfloor}c_{l}\right)^{2},
\end{align*}
which is of the form (\ref{eq:FR1bnd2}) with $\delta_{1,n}=\lfloor ns\rfloor$
and $\delta_{2,n}=\lfloor nr\rfloor-\lfloor ns\rfloor$, and so is
$o[L(n)]$ as required.

It remains to prove (\ref{eq:FR1bnd1}) and (\ref{eq:FR1bnd2}). Recall
that $c_{j}\sim j^{-3/2}\ell(j)=:g(j)$, which is regularly varying
with index $-3/2$. Since $\ell$ is defined only up to an asymptotic
equivalence, it is without loss of generality to take $\ell$ to be
such that $g$ is monotone decreasing (see Theorem~1.5.3 in BGT).
For the first part of (\ref{eq:FR1bnd1}), using that $\sum_{j=0}^{\infty}c_{j}=0$
we have
\[
\sum_{k=0}^{n-1}\left(\sum_{j=0}^{k}c_{j}\right)^{2}\sim\sum_{k=1}^{n-1}\left(\sum_{j=k+1}^{\infty}c_{j}\right)^{2}=\sum_{k=1}^{n-1}k^{2}g^{2}(k)\left(k^{-1}g(k)^{-1}\sum_{j=k+1}^{\infty}c_{j}\right)^{2}.
\]
By Lemma~\ref{SV}(ii), 
\[
k^{-1}g(k)^{-1}\sum_{j=k+1}^{\infty}c_{j}\rightarrow\int_{1}^{\infty}x^{-3/2}dx=2
\]
as $k\rightarrow\infty$, while by Lemma~\ref{SV}(iii), 
\[
L(n)^{-1}\sum_{k=1}^{n-1}k^{2}g^{2}(k)=L(n)^{-1}\sum_{k=1}^{n-1}\frac{\ell^{2}(k)}{k}\rightarrow1,
\]
whence by the Toeplitz lemma (Hall and Heyde, 1980, p.~31)
\[
L(n)^{-1}\sum_{k=0}^{n-1}\left(\sum_{j=0}^{k}c_{j}\right)^{2}\rightarrow2^{2}=4
\]
as required. For the second part of (\ref{eq:FR1bnd1}), we have
\begin{align*}
\sum_{k=0}^{n-1}\left(\sum_{j=k+n+1}^{\infty}c_{j}\right)^{2} & =\sum_{k=0}^{n-1}\left(\sum_{j=0}^{k+n}c_{j}\right)^{2}\\
 & =\sum_{k=n}^{2n-1}\left(\sum_{j=0}^{k}c_{j}\right)^{2}=\sum_{k=0}^{2n-1}\left(\sum_{j=0}^{k}c_{j}\right)^{2}-\sum_{k=0}^{n-1}\left(\sum_{j=0}^{k}c_{j}\right)^{2}.
\end{align*}
By the first part of (\ref{eq:FR1bnd1}), the r.h.s.\ is asymptotically
equivalent to
\[
L(2n)-L(n)=L(2n)\left[1-\frac{L(2n)}{L(n)}\right]=o[L(2n)]=o[L(n)]
\]
by the slow variation of $L(n)$.

Finally, we turn to (\ref{eq:FR1bnd2}). Since $c_{j}\sim g(j)$,
$c_{j}$ is eventually strictly positive, and there exists a $C<\infty$
such that $0<c_{j}\leq Cg(j)$ for all $j$ sufficiently large. Hence,
for all $n$ sufficiently large
\begin{align*}
\sum_{k=\delta_{1,n}}^{\infty}\left(\sum_{j=k+1}^{k+\delta_{2,n}}c_{j}\right)^{2} & \leq C_{1}\sum_{k=\delta_{1,n}}^{\infty}\left(\sum_{j=k+1}^{k+\delta_{2,n}}g(j)\right)^{2}\leq C_{1}\delta_{2,n}^{2}\sum_{k=\delta_{1,n}}^{\infty}g^{2}(k),
\end{align*}
where the second inequality holds by the monotonicity of $g(j)$.
By Lemma~\ref{SV}(ii),
\[
\sum_{k=\delta_{1,n}}^{\infty}g^{2}(k)=\sum_{k=\delta_{1,n}}^{\infty}k^{-3}\ell^{2}(k)\sim\delta_{1,n}^{-2}\ell^{2}(\delta_{1,n}).
\]
Since $\delta_{1,n}\asymp n\asymp\delta_{2,n}$, it follows that 
\begin{align*}
L(n)^{-1}\sum_{k=\delta_{1,n}}^{\infty}\left(\sum_{j=k+1}^{k+\delta_{2,n}}c_{j}\right)^{2} & \leq C_{2}\frac{\ell^{2}(n)}{L(n)}\frac{\ell^{2}(\delta_{1,n})}{\ell^{2}(n)}=o(1)
\end{align*}
by the slow variation of $\ell^{2}$, and Theorem~1.5.9.a in BGT.

\textbf{FR2.} Consider $x_{n}^{+}$ first. We have
\begin{align}
Var(x_{n}^{+}) & =\sum_{k=0}^{n-1}\left(\sum_{j=0}^{k}\phi_{j}c_{k-j}\right)^{2}=\sum_{k=0}^{n-1}\left(\sum_{j=0}^{k}\phi_{k-j}c_{j}\right)^{2}\nonumber \\
 & =\sum_{k=0}^{n-1}\sum_{i=0}^{k}\sum_{j=0}^{k}c_{i}c_{j}\phi_{k-i}\phi_{k-j}\nonumber \\
 & =\sum_{i=0}^{n-1}\sum_{j=0}^{n-1}c_{i}c_{j}\sum_{k=i\vee j}^{n-1}\phi_{k-i}\phi_{k-j}\label{eq:FR2varplus}
\end{align}
Fix $i,j\in\mathbb{N}$, taking $i\geq j$ without loss of generality.
Then
\[
\sum_{k=i\vee j}^{n-1}\phi_{k-i}\phi_{k-j}=\sum_{k=i}^{n-1}\phi_{k-j}\phi_{k-i}=\sum_{k=0}^{n-i-1}\phi_{k+(i-j)}\phi_{k}=\sum_{k=0}^{n-i-1}\phi_{k}^{2}\frac{\phi_{k+(i-j)}}{\phi_{k}}.
\]
Since $\phi_{k}\sim k^{-1/2}\ell(k)$, which is regularly varying,
$\phi_{k+(i-j)}/\phi_{k}\rightarrow1$ as $k\rightarrow\infty$, while
by Lemma~\ref{SV}(iii),
\[
\gamma_{n}^{-2}\sum_{k=0}^{n-i-1}\phi_{k}^{2}\sim L(n)^{-1}\sum_{k=1}^{n}k^{-1}\ell^{2}(k)\rightarrow1.
\]
Hence by the Toeplitz Lemma (Hall and Heyde, 1980, p.~31),
\begin{equation}
\gamma_{n}^{-2}\sum_{k=i\vee j}^{n-1}\phi_{k-i}\phi_{k-j}=\gamma_{n}^{-2}\sum_{k=0}^{n-i-1}\phi_{k}^{2}\frac{\phi_{k+(i-j)}}{\phi_{k}}\rightarrow1\label{eq:FR2termcvg}
\end{equation}
as $n\rightarrow\infty$, for each $i,j\in\mathbb{N}$ with $i\geq j$.
Moreover, by the Cauchy-Schwarz inequality,
\begin{equation}
\gamma_{n}^{-2}\sum_{k=i\vee j}^{n-1}\lvert\phi_{k-i}\phi_{k-j}\rvert\leq\gamma_{n}^{-2}\sum_{k=0}^{n-1}\phi_{k}^{2}<C<\infty\label{eq:FR2termbnd}
\end{equation}
for some $C<\infty$ not depending on $i$ or $j$. Since $\sum_{k=0}^{\infty}\lvert c_{k}\rvert<\infty$,
it follows from (\ref{eq:FR2varplus})--(\ref{eq:FR2termbnd}) and
the dominated convergence theorem that 
\[
\gamma_{n}^{-2}Var(x_{n}^{+})\rightarrow\sum_{i=0}^{\infty}\sum_{j=0}^{\infty}c_{i}c_{j}=\left(\sum_{i=0}^{\infty}c_{i}\right)^{2}.
\]

With respect to $x_{n}^{-}$, we have:
\begin{align*}
Var(x_{n}^{-}) & =\sum_{k=n}^{\infty}\left(\sum_{j=0}^{n-1}\phi_{j}c_{k-j}\right)^{2}=\sum_{k=0}^{\infty}\left(\sum_{j=0}^{n-1}\phi_{j}c_{n+k-j}\right)^{2}\\
 & \leq2\sum_{k=0}^{\infty}\sum_{j=0}^{n-1}\lvert\phi_{j}c_{n+k-j}\rvert\sum_{i=j}^{n-1}\lvert\phi_{i}c_{n+k-i}\rvert\\
 & \leq_{(1)}C\sum_{k=0}^{\infty}\sum_{j=0}^{n-1}\phi_{j}^{2}\lvert c_{n+k-j}\rvert\sum_{i=j}^{n-1}\lvert c_{n+k-i}\rvert\\
 & \leq_{(2)}C_{1}\sum_{k=0}^{\infty}\sum_{j=0}^{n-1}\phi_{j}^{2}\lvert c_{n+k-j}\rvert,
\end{align*}
where $\leq_{(1)}$ holds since $\phi_{k}\sim k^{-1/2}\ell(k)$, which
can be taken to be monotone without loss of generality (see Theorem~1.5.3
in BGT), while $\leq_{(2)}$ follows from $\sum_{i=j}^{n-1}\lvert c_{n+k-i}\rvert\leq\sum_{i=0}^{\infty}\lvert c_{i}\rvert<\infty$.
Letting $\overline{c}_{m}:=\sum_{k=m}^{\infty}\lvert c_{k}\rvert$,
we thus have that $Var(x_{n}^{-})$ is bounded, up to multiplicative
constant, by
\begin{align*}
\sum_{j=0}^{n-1}\phi_{j}^{2}\sum_{k=0}^{\infty}\lvert c_{n+k-j}\rvert & =\sum_{j=0}^{n-1}\phi_{j}^{2}\overline{c}_{n-j}\\
 & \leq\left(\sum_{j=0}^{\lfloor n/2\rfloor}+\sum_{j=\lfloor n/2\rfloor+1}^{n-1}\right)\phi_{j}^{2}\overline{c}_{n-j}\\
 & \leq\overline{c}_{\lfloor n/2\rfloor}\sum_{j=0}^{\lfloor n/2\rfloor}\phi_{j}^{2}+C\sum_{j=\lfloor n/2\rfloor+1}^{n-1}\phi_{j}^{2}.
\end{align*}
We claim that the r.h.s.\ is of smaller order than $\gamma_{n}^{2}=L(n)$.
For the first term, this follows from $\overline{c}_{\lfloor n/2\rfloor}\rightarrow0$
and $\sum_{j=0}^{\lfloor n/2\rfloor}\phi_{j}^{2}\sim L(\lfloor n/2\rfloor)\sim L(n)$
by Lemma~\ref{SV}(iii) and the slow variation of $L(n)$. By the
same argument,
\begin{multline*}
\sum_{j=\lfloor n/2\rfloor+1}^{n-1}\phi_{j}^{2}=\sum_{j=1}^{n-1}\phi_{j}^{2}-\sum_{j=1}^{\lfloor n/2\rfloor}\phi_{j}^{2}\sim L(n)-L(\lfloor n/2\rfloor)\\
=L(n)\left[1-\frac{L(\lfloor n/2\rfloor)}{L(n)}\right]=o[L(n)].
\end{multline*}

\textbf{MI.} By the same argument as which led to (\ref{eq:FR2varplus}),
and noting that $\phi_{j}(n)=\rho_{n}^{j}$ for $\rho_{n}=1-\kappa_{n}^{-1}$
in this case, we have
\begin{align*}
Var(x_{t_{n}}^{+}(n)) & =\sum_{i=0}^{t_{n}-1}\sum_{j=0}^{t_{n}-1}c_{i}c_{j}\sum_{k=i\vee j}^{t_{n}-1}\rho_{n}^{k-i}\rho_{n}^{k-j}.
\end{align*}
Fix $i,j\in\mathbb{N}$, taking $i\geq j$ without loss of generality.
Then
\[
\sum_{k=i\vee j}^{t_{n}-1}\rho_{n}^{k-i}\rho_{n}^{k-j}=\sum_{k=i}^{t_{n}-1}\rho_{n}^{2(k-i)+(i-j)}=\rho_{n}^{i-j}\sum_{k=0}^{t_{n}-i-1}\rho_{n}^{2k}=\rho_{n}^{i-j}\frac{1-\rho_{n}^{2(t_{n}-i)}}{1-\rho_{n}^{2}}.
\]
Since $1-\rho_{n}^{2}=(1+\rho_{n})(1-\rho_{n})\sim2\kappa_{n}^{-1}$,
it follows that
\[
\kappa_{n}^{-1}\sum_{k=i\vee j}^{t_{n}-1}\rho_{n}^{k-i}\rho_{n}^{k-j}\sim2^{-1}\rho_{n}^{i-j}(1-\rho_{n}^{2(t_{n}-i)})\rightarrow2^{-1}
\]
as $n\rightarrow\infty$ for each $i,j\in\mathbb{N}$ with $i\geq j$,
since $\rho_{n}^{i-j}\rightarrow1$ as $n\rightarrow\infty$ and 
\begin{equation}
\rho_{n}^{t_{n}}=(1-\kappa_{n}^{-1})^{t_{n}}=[(1-\kappa_{n}^{-1})^{\kappa_{n}}]^{t_{n}/\kappa_{n}}\sim e^{-t_{n}/\kappa_{n}}\rightarrow0\label{eq:rhot}
\end{equation}
since $\kappa_{n}=o(t_{n})$. Moreover, for all $i,j\in\{0,\ldots,t_{n}\}$
with $i\geq j$, and all $n\in\mathbb{N}$,
\[
\kappa_{n}^{-1}\sum_{k=i\vee j}^{t_{n}-1}\rho_{n}^{k-i}\rho_{n}^{k-j}\leq C\lvert1-\rho_{n}^{2(t_{n}-i)}\rvert\leq C.
\]
Hence it follows by the absolute summability of $\{c_{s}\}$ and the
dominated convergence theorem that
\[
\kappa_{n}^{-1}Var(x_{t_{n}}^{+}(n))\rightarrow\frac{1}{2}\sum_{i=0}^{\infty}\sum_{j=0}^{\infty}c_{i}c_{j}=\frac{1}{2}\left(\sum_{i=0}^{\infty}c_{i}\right)^{2}.
\]

Next, since the coefficients $\{a_{k,t}^{-}(n)\}$ are uniformly bounded
by Lemma~\ref{COEF}, we have 
\begin{align*}
Var(x_{t_{n}}^{-}(n)) & =\sum_{k=t_{n}}^{\infty}[a_{k,t_{n}}^{-}(n)]^{2}\leq C\sum_{k=t_{n}}^{\infty}\lvert a_{k,t_{n}}^{-}(n)\rvert\leq C\sum_{k=t_{n}}^{\infty}\sum_{j=0}^{t_{n}-1}\lvert\rho_{n}^{j}c_{k-j}\lvert.
\end{align*}
Letting $\overline{c}_{m}:=\sum_{k=m}^{\infty}\lvert c_{k}\rvert$,
we can bound the r.h.s., up to a multiplicative constant, by
\begin{align*}
\sum_{j=0}^{t_{n}-1}\rho_{n}^{j}\sum_{k=t_{n}}^{\infty}\lvert c_{k-j}\rvert & =\sum_{j=0}^{t_{n}-1}\rho_{n}^{j}\overline{c}_{t_{n}-j}\\
 & \leq\left(\sum_{j=0}^{\lfloor t_{n}/2\rfloor-1}+\sum_{j=\lfloor t_{n}/2\rfloor}^{t_{n}-1}\right)\rho_{n}^{j}\overline{c}_{t_{n}-j}\\
 & \leq_{(1)}C\left(\overline{c}_{\lfloor t_{n}/2\rfloor}+\rho_{n}^{\lfloor t_{n}/2\rfloor}\right)\sum_{j=0}^{t_{n}-1}\rho_{n}^{j}\\
 & =_{(2)}o(\kappa_{n})
\end{align*}
where $\leq_{(1)}$ holds for all $n$ sufficiently large that $\rho_{n}\in(0,1)$,
and $=_{(2)}$ from $\overline{c}_{\lfloor t_{n}/2\rfloor}+\rho_{n}^{\lfloor t_{n}/2\rfloor}\rightarrow0$
(as per (\ref{eq:rhot}) above), and
\[
\sum_{j=0}^{t_{n}-1}\rho_{n}^{j}=\frac{1-\rho_{n}^{t_{n}}}{1-\rho_{n}}=\kappa_{n}(1-\rho_{n}^{t_{n}})\sim\kappa_{n}.\qedhere
\]
\end{proof}
\begin{proof}[Proof of Lemma~\ref{CLT}]
 By Lemma~\ref{LVAR}, we have 
\begin{align*}
Var\{\gamma_{n}^{-1}x_{t_{n}}^{+}(n)\} & \rightarrow\sigma_{+}^{2} & Var\{\gamma_{n}^{-1}x_{t_{n}}^{-}(n)\} & \rightarrow\sigma_{-}^{2}
\end{align*}
for appropriate $\sigma_{+}$ and $\sigma_{-}$. Each of $\{\gamma_{n}^{-1}x_{t_{n}}^{+}(n)\}$
and $\{\gamma_{n}^{-1}x_{t_{n}}^{-}(n)\}$ are independent linear
processes arrays with respective coefficients $\{\gamma_{n}^{-1}a_{k}(n)\}_{k=0}^{t_{n}-1}$
and $\{\gamma_{n}^{-1}a_{k,t_{n}}^{-}(n)\}_{k=t_{n}}^{\infty}$, which
by Lemma~\ref{COEF} have the property that
\[
\lim_{n\rightarrow\infty}\gamma_{n}^{-1}\left[\max_{0\leq k\leq t_{n}-1}\lvert a_{k}(n)\rvert+\sup_{k\geq t_{n}}\lvert a_{k,t_{n}}^{-}(n)\rvert\right]=0.
\]
The postulated weak convergence thus follows by Lemma 2.1 in Abadir,
Distaso, Giraitis and Koul (2014).
\end{proof}

\subsection{Proofs of Lemmas B.1 and B.2}

\label{subsec:technicalproofs}
\begin{proof}[Proof of Lemma~\ref{lem:UICF}]
 Without loss of generality, we will prove the result as stated,
but with $\mathbb{N}$ in place of $\mathbb{N}_{0}$, so that $\{\vartheta_{k}\}$
is indexed by $k\in\{1,2,\ldots\}$. Let $\eta:=\sum_{k=1}^{\infty}\vartheta_{k}\xi_{k}$,
and define
\[
\psi_{\eta}(\lambda):=\mathbf{E}\exp(\mathrm{i}\lambda\eta)=\prod_{k=1}^{\infty}\psi_{\xi}(\lambda\vartheta_{k}),
\]
where the second equality holds by independence. Since $\mathbf{E}\xi_{1}^{2}<\infty$,
by a Taylor expansion of the log characteristic function, there exists
a $\gamma\in(0,\infty)$ such that
\begin{equation}
\lvert\psi_{\xi}(\lambda)\rvert\leq e^{-\gamma\lambda^{2}}\vee e^{-\gamma}.\label{eq:cfbnd}
\end{equation}

\textbf{(i). }Without loss of generality, suppose $\{\vartheta_{k}\}$
is ordered as so that $\lvert\vartheta_{1}\rvert\geq\lvert\vartheta_{2}\lvert\geq\cdots$.
Define
\[
\mathcal{K}:=\left\{ k\geq\theta+1\text{ and }k\in\mathbb{N}\mid\vartheta_{k}^{2}\geq\frac{3\sigma_{\vartheta}^{2}}{2\pi}k^{-2}\right\} .
\]
We claim that $\mathcal{K}$ is nonempty. To see this, observe that
by definition of $\mathcal{K}$ and the stated condition on $\max_{k\in\mathbb{N}}\vartheta_{k}^{2}$,
\[
\sum_{k\notin\mathcal{K}}^{\infty}\vartheta_{k}^{2}=\sum_{k=1}^{\theta}\vartheta_{k}^{2}+\sum_{\substack{k\notin\mathcal{K}\\
k\geq\theta+1
}
}\vartheta_{k}^{2}\leq\theta\cdot\frac{\sigma_{\vartheta}^{2}}{2\theta}+\frac{3\sigma_{\vartheta}^{2}}{2\pi}\sum_{k=\theta+1}^{\infty}k^{-2}\leq\frac{3\sigma_{\vartheta}^{2}}{4}
\]
where we have used that $\sum_{k=1}^{\infty}k^{-2}=\pi/6$. Hence
\begin{equation}
\sum_{k\in\mathcal{K}}^{\infty}\vartheta_{k}^{2}\geq\frac{\sigma_{\vartheta}^{2}}{4}>0,\label{eq:varKbnd}
\end{equation}
and so $\mathcal{K}$ is nonempty. Let $k_{0}$ denote its first element,
noting that $k_{0}\geq\theta+1$ by construction.

We want to bound the integral of $\lvert\psi_{\eta}(\lambda)\rvert$
over $[A,\infty)$. To that end, decompose
\[
[A,\infty)=[A,A\vee\vartheta_{k_{0}}^{-1}]\cup[A\vee\vartheta_{k_{0}}^{-1},\infty),
\]
and consider the integral on each of these two pieces separately.
We have
\begin{align*}
\int_{\{\lvert\lambda\vert\in[A,A\vee\vartheta_{k_{0}}^{-1}]\}}\lvert\psi_{\eta}(\lambda)\rvert d\lambda & \leq\int_{\{\lvert\lambda\vert\in[A,A\vee\vartheta_{k_{0}}^{-1}]\}}\prod_{k\in\mathcal{K}}^{\infty}\lvert\psi_{\xi}(\lambda\vartheta_{k})\rvert d\lambda.
\end{align*}
Since $\lvert\lambda\vert\leq\vartheta_{k_{0}}^{-1}$ on the domain
of integration, we have $\lvert\lambda\vartheta_{k}\rvert\leq\lvert\vartheta_{k_{0}}^{-1}\vartheta_{k}\rvert\leq1$
for all $k\in\mathcal{K}$, and so by (\ref{eq:cfbnd}) and then (\ref{eq:varKbnd}),
\begin{align*}
\int_{\{\lvert\lambda\vert\in[A,A\vee\vartheta_{k_{0}}^{-1}]\}}\prod_{k\in\mathcal{K}}^{\infty}\lvert\psi_{\xi}(\lambda\vartheta_{k})\rvert d\lambda & \leq\int_{\{\lvert\lambda\vert\in[A,A\vee\vartheta_{k_{0}}^{-1}]\}}\exp\left(-\gamma\lambda^{2}\sum_{k\in\mathcal{K}}\vartheta_{k}^{2}\right)d\lambda\\
 & \leq\int_{\{\lvert\lambda\vert\geq A\}}\exp\left(-\gamma\lambda^{2}\sigma_{\vartheta}^{2}/4\right)d\lambda.\\
 & =:G_{1}(A;\sigma_{\vartheta}^{2}),
\end{align*}
where $G_{1}$ depends on $\psi_{\xi}$ through $\gamma$. By a change
of variables $G_{1}(A;\sigma_{\vartheta}^{2})\leq C_{\gamma}\sigma_{\vartheta}^{-1}$,
for some $C_{\gamma}$ depending only on $\gamma$. Moreover, $\lim_{A\rightarrow\infty}G_{1}(A;\sigma_{\vartheta}^{2})=0$
and $\sigma^{2}\mapsto G_{1}(A;\sigma^{2})$ is decreasing in $\sigma^{2}$,
as required,

Next, noting that the leading $k_{0}$ elements of $\{\vartheta_{k}\}$
(as ordered) must be nonzero, we have 
\begin{align*}
\int_{\{\lvert\lambda\vert\geq A\vee\vartheta_{k_{0}}^{-1}\}}\lvert\psi_{\eta}(\lambda)\rvert d\lambda & \leq\int_{\{\lvert\lambda\vert\geq A\vee\vartheta_{k_{0}}^{-1}\}}\prod_{k=1}^{k_{0}}\lvert\psi_{\xi}(\lambda\vartheta_{k})\rvert d\lambda\\
 & \leq e^{-\gamma(k_{0}-\theta)}\int_{\{\lvert\lambda\vert\geq A\vee\vartheta_{k_{0}}^{-1}\}}\prod_{k=1}^{\theta}\lvert\psi_{\xi}(\lambda\vartheta_{k})\rvert d\lambda
\end{align*}
where the second equality follows from (\ref{eq:cfbnd}) and the fact
that $\lvert\lambda\vartheta_{k}\rvert\geq\lvert\vartheta_{k_{0}}^{-1}\vartheta_{k}\rvert\geq1$
on the domain of integration, for all $k\geq k_{0}$. By Hölder's
inequality,
\begin{align}
\int_{\{\lvert\lambda\vert\geq A\vee\vartheta_{k_{0}}^{-1}\}}\prod_{k=1}^{\theta}\lvert\psi_{\xi}(\lambda\vartheta_{k})\rvert d\lambda & \leq\int_{\{\lvert\lambda\vert\geq A\}}\prod_{k=1}^{\theta}\lvert\psi_{\xi}(\lambda\vartheta_{k})\rvert d\lambda\nonumber \\
 & \leq\prod_{k=1}^{\theta}\left(\int_{\{\lvert\lambda\vert\geq A\}}\lvert\psi_{\xi}(\lambda\vartheta_{k})\rvert^{\theta}d\lambda\right)^{1/\theta}\nonumber \\
 & \leq\max_{1\leq k\leq\theta}\vartheta_{k}^{-1}\int_{\{\lvert\lambda\vert\geq A\vartheta_{k}\}}\lvert\psi_{\xi}(\lambda)\rvert^{\theta}d\lambda.\label{eq:holder}
\end{align}
Now $\vartheta_{k}\geq\vartheta_{k_{0}}\geq c_{0}\sigma_{\vartheta}k_{0}^{-1}$
for $k\leq k_{0}$, where $c_{0}:=(3/2\pi)^{1/2}$, whence
\[
\max_{1\leq k\leq\theta}\vartheta_{k}^{-1}\int_{\{\lvert\lambda\vert\geq A\vartheta_{k}\}}\lvert\psi_{\xi}(\lambda)\rvert^{\theta}d\lambda\leq c_{0}^{-1}\sigma_{\vartheta}^{-1}k_{0}\int_{\{\lvert\lambda\vert\geq Ac_{0}k_{0}^{-1}\sigma_{\vartheta}\}}\lvert\psi_{\xi}(\lambda)\rvert^{\theta}d\lambda.
\]
Thus,
\begin{align*}
\int_{\{\lvert\lambda\vert\geq A\vee\vartheta_{k_{0}}^{-1}\}}\lvert\psi_{\eta}(\lambda)\rvert d\lambda & \leq c_{0}^{-1}e^{\gamma\theta}\sigma_{\vartheta}^{-1}e^{-\gamma k_{0}}k_{0}\int_{\{\lvert\lambda\vert\geq Ac_{0}k_{0}^{-1}\sigma_{\vartheta}\}}\lvert\psi_{\xi}(\lambda)\rvert^{\theta}d\lambda\\
 & =:G_{2,k_{0}}(A;\sigma_{\vartheta}^{2}).
\end{align*}
Since the final integral is bounded by $\int\lvert\psi_{\xi}(\lambda)\rvert^{\theta}d\lambda<\infty$,
and $e^{-\gamma k}k\rightarrow0$ as $k\rightarrow\infty$, it is
evident that 
\[
G_{2}(A;\sigma_{\vartheta}^{2}):=\sup_{k\geq\theta+1}G_{2,k}(A;\sigma_{\vartheta}^{2})\leq C_{\psi_{\xi}}\sigma_{\vartheta}^{-1},
\]
for some $C_{\psi_{\xi}}<\infty$ depending only on $\psi_{\xi}$. 

Since each $G_{2,k}(A;\sigma^{2})$ is weakly decreasing in $\sigma^{2}$,
so too must be $G_{2}(A;\sigma^{2})$. It remains therefore to show
that $G_{2}(A;\sigma_{\vartheta}^{2})\rightarrow0$ as $A\rightarrow\infty$.
To that end, let $\epsilon>0$ and note that since $e^{-\gamma k}k\rightarrow0$
as $k\rightarrow\infty$, we may choose $k^{\ast}$ such that $G_{2,k}(A;\sigma_{\vartheta}^{2})\leq\epsilon$
for all $k\geq k^{\ast}$. Hence
\[
G_{2}(A;\sigma^{2})\leq\epsilon+\max_{\theta+1\leq k\leq k^{\ast}}G_{2,k}(A;\sigma_{\vartheta}^{2})\rightarrow\epsilon
\]
as $A\rightarrow\infty$, since $G_{2,k}(A;\sigma_{\vartheta}^{2})\rightarrow0$
as $A\rightarrow\infty$ for each $k$ fixed.

\textbf{(ii).} Without loss of generality, we may take $k_{i}=i$
for each $i\in\{1,\ldots,\theta\}$. Then by the same argument as
which led to (\ref{eq:holder}) above,
\begin{align*}
\int_{\mathbb{R}}\lvert\psi_{\eta}(\lambda)\rvert d\lambda & \leq\int_{\mathbb{R}}\prod_{k=1}^{\theta}\lvert\psi_{\xi}(\lambda\vartheta_{k})\vert d\lambda\\
 & \leq\max_{1\leq k\leq\theta}\int_{\mathbb{R}}\lvert\psi_{\xi}(\lambda\vartheta_{k})\vert^{\theta}d\lambda\\
 & =\max_{1\leq k\leq\theta}\vartheta_{k}^{-1}\int_{\mathbb{R}}\lvert\psi_{\xi}(\lambda)\vert^{\theta}d\lambda\\
 & \leq\delta^{-1}\int_{\mathbb{R}}\lvert\psi_{\xi}(\lambda)\vert^{\theta}d\lambda.\qedhere
\end{align*}
\end{proof}
\begin{proof}[Proof of Lemma~\ref{SV}]
 \textbf{(i).} The result follows by similar arguments to those given
in Giraitis, Koul and Surgailis (2012, p.~20).

\textbf{(ii).} Since $g(j)=j^{l}\varsigma(j)$ with $l<-1$, by Theorem~1.5.3
in BGT, we may without loss of generality take $\varsigma$ to be
such that $g$ is monotone decreasing. Since $\sum_{j=n}^{\infty}\varphi_{j}\sim\sum_{j=n}^{\infty}g(j)$
as $n\rightarrow\infty$, we have
\[
\frac{1}{ng(n)}\sum_{j=n}^{\infty}\varphi_{j}\sim\frac{1}{ng(n)}\sum_{j=n}^{\infty}g(j).
\]
By monotonicity of $g$,
\[
\frac{\int_{n}^{\infty}g(x)dx}{ng(n)}\leq\frac{\sum_{j=n}^{\infty}g(j)}{ng(n)}\leq\frac{\int_{n-1}^{\infty}g(x)dx}{ng(n)}.
\]
By Theorem~1.5.11 in BGT,
\[
\frac{\int_{n}^{\infty}g(x)dx}{ng(n)}\rightarrow-\frac{1}{l+1}=\int_{1}^{\infty}x^{l}dx.
\]
while by the preceding and Theorem~1.5.2 in BGT,
\[
\frac{\int_{n-1}^{\infty}g(x)dx}{ng(n)}=\frac{(n-1)g(n-1)}{ng(n)}\cdot\frac{\int_{n-1}^{\infty}g(x)dx}{(n-1)g(n-1)}\rightarrow-\frac{1}{l+1}.
\]

\textbf{(iii).} Set $s(x):=x^{-1}\varsigma(x)$, and note that
\[
\underline{s}(x):=\inf_{1\leq u\leq x}s(u)\leq s(x)\leq\sup_{u\geq x}\overline{s}(u)=:\overline{s}(x).
\]
Let $\epsilon>0$. By Theorem~1.5.3 of BGT, $\underline{s}(x)\sim s(x)\sim\overline{s}(x)$
as $x\rightarrow\infty$ , and so we may choose $x_{0}\in\mathbb{N}$
with $x_{0}\geq2$ such that 
\begin{equation}
(1-\epsilon)s(x)\leq\underline{s}(x)\leq s(x)\leq\overline{s}(x)\leq(1+\epsilon)s(x)\label{eq:fnsandwich}
\end{equation}
for all $x\geq x_{0}$. Since $\underline{s}$ and $\overline{s}$
are monotone decreasing, we also have
\begin{equation}
\int_{x_{0}}^{n}\underline{s}(u)du\leq\sum_{j=x_{0}}^{n}\underline{s}(j)\leq\sum_{j=x_{0}}^{n}s(j)\leq\sum_{j=x_{0}}^{n}\overline{s}(j)\leq\int_{x_{0}-1}^{n}\overline{s}(u)du.\label{eq:integsandwich}
\end{equation}
By (\ref{eq:fnsandwich}) and (\ref{eq:integsandwich}),
\[
\sum_{j=x_{0}}^{n}s(j)\geq\int_{x_{0}}^{n}\underline{s}(u)du\geq(1-\epsilon)\int_{x_{0}}^{n}s(u)du\sim(1-\epsilon)S(n)
\]
as $n\rightarrow\infty$, noting that since $S(n)\rightarrow\infty$
by hypothesis, $\int_{x}^{n}s(u)du\sim S(n)$ for each $x\geq1$.
Similarly,
\[
\sum_{j=x_{0}}^{n}s(j)\leq(1+\epsilon)\int_{x_{0}-1}^{n}s(u)du\sim(1+\epsilon)S(n)
\]
Deduce
\[
\sum_{j=1}^{n}s(j)\sim\sum_{j=x_{0}}^{n}s(j)\sim S(n)
\]
as $n\rightarrow\infty$.
\end{proof}

\section{Proofs of Theorems~4.1 and 4.2}

\label{sec:Proofs-of-Theorems4.1}
\begin{proof}[Proof of Theorem~\ref{LS}]
Let $\kappa_{gn}:=\kappa_{g}(\beta_{n})$. By standard arguments,
\begin{align}
n^{1/2}\begin{bmatrix}\hat{\mu}-\mu\\
\kappa_{gn}(\hat{\gamma}-\gamma)
\end{bmatrix} & =\left\{ \frac{1}{n}\sum_{t=2}^{n}\left[\begin{array}{cc}
1 & \kappa_{gn}^{-1}g(x_{t-1})\\
\kappa_{gn}^{-1}g(x_{t-1}) & \kappa_{gn}^{-2}g^{2}(x_{t-1})
\end{array}\right]\right\} ^{-1}\nonumber \\
 & \qquad\qquad\qquad\qquad\qquad\qquad\cdot\frac{1}{n^{1/2}}\sum_{t=2}^{n}\begin{bmatrix}1\\
\kappa_{gn}^{-1}g(x_{t-1})
\end{bmatrix}u_{t}\nonumber \\
 & =:M_{n}^{-1}S_{n}.\label{eq:olsdecomp}
\end{align}

Consider $M_{n}$ first. We claim that
\begin{equation}
\frac{1}{n}\sum_{t=2}^{n}[\kappa_{gn}^{-j}g^{j}(x_{t-1})-H_{n,t-1}^{j}]=o_{p}(1)\label{eq:homogdecomp}
\end{equation}
for $j\in\{1,2\}$, where $H_{n,t}:=H_{g}(x_{n,t-1})$. When $j=1$,
this follows immediately from the definition of an AHF. When $j=2$,
the l.h.s.\ is bounded (in absolute value) by 
\begin{equation}
\frac{1}{n\kappa_{gn}^{2}}\sum_{t=2}^{n}\lvert H_{n,t-1}R_{g}(x_{t-1},\beta_{n})\rvert+\frac{1}{n\kappa_{gn}^{2}}\sum_{t=2}^{n}R_{g}^{2}(x_{t-1},\beta_{n}).\label{eq:lsdecomp}
\end{equation}
The second term is $o_{p}(1)$ by (\ref{AHF}). Since $H_{g}^{2}$
satisfies the conditions of Theorem~\ref{MainLL}, we have 
\[
\frac{1}{n}\sum_{t=2}^{n}H_{n,t-1}^{2}=O_{p}(1),
\]
 whence the first term in (\ref{eq:lsdecomp}) is also $o_{p}(1)$,
by the Cauchy-Schwarz inequality. Thus by (\ref{eq:homogdecomp})
and Theorem~\ref{MainLL},
\begin{multline}
M_{n}=\frac{1}{n}\sum_{t=2}^{n}\left[\begin{array}{cc}
1 & \kappa_{gn}^{-1}g(x_{t-1})\\
H_{n,t-1}^{j} & \kappa_{gn}^{-2}g^{2}(x_{t-1})
\end{array}\right]\\
\overset{d}{\rightarrow}\int_{\mathbb{R}}\begin{bmatrix}1 & H_{g}(x)\\
H_{g}(x) & H_{g}^{2}(x)
\end{bmatrix}\varrho(x)dx=:M_{H_{g}}.\label{eq:Mncvg}
\end{multline}
which is a.s.\ nonsingular, since $H_{g}$ is not a.e.\ equal from
a constant function by assumption.

We turn next to $S_{n}$. We first note that
\[
\mathbf{E}\left|\frac{1}{n^{1/2}}\sum_{t=2}^{n}[\kappa_{gn}^{-1}g(x_{t-1})-H_{n,t-1}]u_{t}\right|^{2}=\frac{\sigma_{u}^{2}}{n\kappa_{gn}^{2}}\sum_{t=2}^{n}\mathbf{E}R_{g}\left(x_{t-1}(n),\beta_{n}\right)^{2}=o_{p}(1)
\]
by (\ref{AHF}). Thus
\[
S_{n}=\frac{1}{n^{1/2}}\sum_{t=2}^{n}\begin{bmatrix}1\\
H_{n,t-1}
\end{bmatrix}u_{t}+o_{p}(1)
\]
where the leading r.h.s.\ term is a vector martingale with conditional
variance matrix
\[
\frac{\sigma_{u}^{2}}{n}\sum_{t=2}^{n}\left[\begin{array}{cc}
1 & H_{n,t-1}\\
H_{n,t-1} & H_{n,t-1}^{2}
\end{array}\right]\overset{d}{\rightarrow}\sigma_{u}^{2}M_{H_{g}}
\]
by Theorem~\ref{MainLL}. Thus, if we can verify the requirements
of Wang's (2014) martingale CLT (his Theorem~2.1), the convergence
\begin{equation}
S_{n}\overset{d}{\rightarrow}\sigma_{u}M_{H_{g}}^{1/2}\zeta\label{eq:SnOLS}
\end{equation}
will hold jointly with \prettyref{eq:Mncvg}, where $\zeta\sim N[0,I_{2}]$
is independent of $X^{-}$ (and of the possibly random density $\varrho$
that depends on $X^{-}$) whence the result follows.

Regarding Wang's CLT, the only condition that is not trivially satisfied
in our setting is his Assumption~2. Since the mixing variate $X^{-}$
is the weak limit of $\beta_{n}^{-1}x_{n}^{-}(n)$, which is $\mathcal{F}_{0}$-measurable,
it can be seen from the proof of Wang's CLT that his condition (2.3)
is unnecessary in our case. Thus we need only to verify that $n^{-1/2}\max_{2\leq t\leq n}\lvert H_{n,t-1}\rvert=o_{p}(1)$,
which is equivalent (see e.g.\ Hall and Heyde, 1980, p.\ 53) to
the Lindeberg condition that for each $\eta>0$,
\begin{equation}
\frac{1}{n}\sum_{t=1}^{n-1}H_{g}^{2}\left(\beta_{n}^{-1}x_{t}(n)\right)1\left\{ \left\vert H_{g}\left(\beta_{n}^{-1}x_{t}(n)\right)\right\vert >\eta n^{1/2}\right\} =o_{p}(1)\label{proofLS2}
\end{equation}
as $n\rightarrow\infty$. Recall that $H_{g}^{2}$ is assumed to satisfy
the conditions of Theorem~\ref{MainLL}: therefore so too does $F_{A}(x):=H_{g}^{2}(x)\boldsymbol{1}\{\lvert H_{g}(x)\rvert>A\}$
for each $A\in\mathbb{R}$. Now
\begin{align*}
\Gamma_{n}(A):=\frac{1}{n}\sum_{t=1}^{n-1}F_{A}(\beta_{n}^{-1}x_{t}(n)) & \overset{d}{\rightarrow}\int_{\mathbb{R}}F_{A}(x)\varrho(x)dx\\
 & =\int_{\mathbb{R}}H_{g}^{2}(x)\boldsymbol{1}\{\lvert H_{g}(x)\rvert>A\}\varrho(x)dx=:\Gamma(A)
\end{align*}
by Theorem~\ref{MainLL}. By dominated convergence, $\Gamma(A)\overset{a.s.}{\rightarrow}0$
as $A\rightarrow\infty$, since $H_{g}^{2}\varrho$ is integrable.
Therefore for a given $\epsilon>0$, we may choose $A_{\epsilon}$
such that $\mathbf{P}\{\Gamma(A_{\epsilon})\geq\epsilon\}\leq\epsilon$,
whence
\begin{multline*}
\limsup_{n\rightarrow\infty}\mathbf{P}\{\Gamma_{n}(\eta n^{1/2})\geq\epsilon\}\leq_{(1)}\limsup_{n\rightarrow\infty}\mathbf{P}\{\Gamma_{n}(A_{\epsilon})\geq\epsilon\}\\
\leq_{(2)}\mathbf{P}\{\Gamma(A_{\epsilon})\geq\epsilon\}\leq\epsilon
\end{multline*}
where $\leq_{(1)}$ holds since $A_{\epsilon}<\eta n^{1/2}$ for all
$n$ sufficiently large, and $\leq_{(2)}$ follows by the portmanteau
theorem. Deduce $\Gamma_{n}(\eta n^{1/2})\overset{p}{\rightarrow}0$,
i.e.\ (\ref{proofLS2}) holds.
\end{proof}
\begin{proof}[Proof of Theorem~\ref{thm:npboth}]
\textbf{(i).} We first prove (\ref{eq: NWLimit}). By Theorem~\ref{KernelLL},
\begin{equation}
\frac{\beta_{n}}{h_{n}n}\sum_{t=2}^{n}K^{j}\left(\frac{x_{t}(n)-x}{h_{n}}\right)^{2}\overset{d}{\rightarrow}\varrho(0)\int_{\mathbb{R}}K^{2}(u)du=:V_{K^{j}}\label{eq:proofNW0}
\end{equation}
jointly for $j\in\{1,2\}$. Since $m$ has bounded derivative, $\lvert m(x_{t-1})-m(x)\lvert\leq C\lvert x_{t-1}-x\vert$,
and thus
\begin{align*}
\left|\sum_{t=2}^{n}K\left(\frac{x_{t-1}-x}{h_{n}}\right)[m(x_{t-1})-m(x)]\right| & \leq Ch_{n}\sum_{t=2}^{n}\left|K\left(\frac{x_{t-1}-x}{h_{n}}\right)\frac{x_{t-1}-x}{h_{n}}\right|\\
 & =O_{p}(h_{n}^{2}n/\beta_{n})
\end{align*}
also by Theorem~\ref{KernelLL}. Hence
\begin{multline}
\left(\frac{h_{n}n}{\beta_{n}}\right)^{1/2}[\hat{m}(x)-m(x)]\\
=\frac{\left(\frac{\beta_{n}}{h_{n}n}\right)^{1/2}\sum_{t=2}^{n}K\left(\frac{x_{t-1}-x}{h_{n}}\right)u_{t}}{\frac{\beta_{n}}{h_{n}n}\sum_{t=2}^{n}K\left(\frac{x_{t-1}-x}{h_{n}}\right)}+O_{p}\left(\frac{h_{n}^{3}n}{\beta_{n}}\right)^{1/2},\label{eq:mhatdecomp}
\end{multline}
where the second r.h.s.\ term is $o_{p}(1)$ since $nh_{n}^{3}/\beta_{n}\rightarrow0$
by assumption. The first r.h.s.\ term is a martingale with conditional
variance
\[
\sigma_{u}^{2}\frac{\beta_{n}}{h_{n}n}\sum_{t=2}^{n}K^{2}\left(\frac{x_{t-1}-x}{h_{n}}\right)\overset{d}{\rightarrow}\sigma_{u}^{2}V_{K^{2}}
\]
by (\ref{eq:proofNW0}). Thus, similarly to the argument given in
the proof of Theorem~\ref{LS}, if we can show the following Lindeberg
condition holds, that
\begin{equation}
\frac{\beta_{n}}{h_{n}n}\sum_{t=2}^{n}K^{2}\left(\frac{x_{t-1}-x}{h_{n}}\right)\mathbf{1}\left\{ \left\vert K\left(\frac{x_{t}-x}{h_{n}}\right)\right\vert >\eta\left(\frac{h_{n}n}{\beta_{n}}\right)^{1/2}\right\} =o_{p}(1),\label{proofNW4}
\end{equation}
for each $\eta>0$, then by Wang's (2014) martingale CLT, 
\begin{equation}
\left(\frac{\beta_{n}}{h_{n}n}\right)^{1/2}\sum_{t=2}^{n}K\left(\frac{x_{t-1}-x}{h_{n}}\right)u_{t}\overset{d}{\rightarrow}\zeta\sigma_{u}V_{K^{2}}^{1/2}\label{eq:mhatmgclt}
\end{equation}
jointly with (\ref{eq:proofNW0}), for $\zeta\sim N[0,1]$ independent
of $\varrho(0)$. (\ref{eq: NWLimit}) then follows from (\ref{eq:proofNW0}),
(\ref{eq:mhatdecomp}) and (\ref{eq:mhatmgclt}).

Finally, to verify (\ref{proofNW4}), we note that for $n$ sufficiently
large, the l.h.s\ of (\ref{proofNW4}) is bounded by
\begin{multline*}
\frac{\beta_{n}}{h_{n}n}\sum_{t=2}^{n}K^{2}\left(\frac{x_{t-1}-x}{h_{n}}\right)\mathbf{1}\left\{ \left\vert K\left(\frac{x_{t}-x}{h_{n}}\right)\right\vert >A\right\} \\
\overset{d}{\rightarrow}_{(1)}\varrho(0)\int_{\mathbb{R}}K^{2}(u)\boldsymbol{1}\{\lvert K(u)\rvert>A\}du\overset{a.s.}{\rightarrow}_{(2)}0
\end{multline*}
where $\overset{d}{\rightarrow}_{(1)}$ holds by Theorem~\ref{KernelLL}
as $n\rightarrow\infty$, and $\overset{a.s.}{\rightarrow}_{(2)}$
by dominated convergence as $A\rightarrow\infty$. (\ref{proofNW4})
thus follows by similar arguments as were used to prove (\ref{proofLS2})
above.

\textbf{(ii).} We next prove (\ref{eq:loclin}), using arguments similar
to those given in part~(i). Let $X_{t}^{\prime}:=(1,x_{t-1}-x)$,
$K_{th}:=K[h_{n}^{-1}(x_{t-1}-x)]$, $H_{n}:=\sum_{t=2}^{n}K_{th}X_{t}X_{t}^{\prime}$
and $\Lambda_{n}:=(nh_{n}/\beta_{n})^{1/2}\mathrm{diag}\{1,h_{n}\}$.
Then the LL estimator can be written as 
\[
\mathbf{\tilde{m}}(x):=\begin{bmatrix}\tilde{m}(x)\\
\tilde{m}^{(1)}(x)
\end{bmatrix}=H_{n}^{-1}\sum_{t=2}^{n}X_{t}K_{th}y_{t},
\]
and by standard arguments decomposes as
\begin{align*}
\mathbf{\tilde{m}}(x)-\mathbf{m}(x) & =H_{n}^{-1}\left\{ \sum_{t=2}^{n}K_{th}X_{t}m(x_{t-1})-H_{n}\mathbf{m}(x)\right\} +H_{n}^{-1}\sum_{t=2}^{n}K_{th}X_{t}u_{t}\\
 & =:H_{n}^{-1}R_{n}+H_{n}^{-1}S_{n}.
\end{align*}
where $\mathbf{m}(x):=[m(x),m^{(1)}(x)]^{\prime}$.

We consider each of $H_{n}$, $R_{n}$ and $S_{n}$ in turn. For $H_{n}$,
we have
\begin{align}
\Lambda_{n}^{-1}H_{n}\Lambda_{n}^{-1} & =\frac{\beta_{n}}{nh_{n}}\sum_{t=2}^{n}K\left(\frac{x_{t-1}-x}{h_{n}}\right)\begin{bmatrix}1 & \frac{x_{t-1}-x}{h_{n}}\\
\frac{x_{t-1}-x}{h_{n}} & \left(\frac{x_{t-1}-x}{h_{n}}\right)^{2}
\end{bmatrix}\nonumber \\
 & \stackrel{d}{\rightarrow}\varrho(0)\int\left[\begin{array}{cc}
1 & u\\
u & u^{2}
\end{array}\right]K(u)du\label{eq:Hcvg}
\end{align}
by Theorem~\ref{KernelLL}. For $R_{n}$,
\begin{align*}
R_{n} & =\sum_{t=2}^{n}K_{th}X_{t}\{m(x_{t-1})-X_{t}^{\prime}\mathbf{m}(x)\}\\
 & =\sum_{t=2}^{n}K_{th}X_{t}\{m(x_{t-1})-m(x)-m^{(1)}(x)(x_{t-1}-x)\}\\
 & =\sum_{t=2}^{n}K_{th}X_{t}m^{(2)}(\overline{x}_{t-1})(x_{t-1}-x)^{2}
\end{align*}
by Taylor's theorem, for some $\overline{x}_{t-1}$ lying between
$x_{t-1}$ and $x$. Thus,
\begin{equation}
\Lambda_{n}^{-1}R_{n}=h_{n}^{2}\left(\frac{\beta_{n}}{nh_{n}}\right)^{1/2}\sum_{t=2}^{n}K_{th}\begin{bmatrix}[h_{n}^{-1}(x_{t-1}-x)]^{2}\\{}
[h_{n}^{-1}(x_{t-1}-x)]^{3}
\end{bmatrix}m^{(2)}(\overline{x}_{t-1}),\label{eq:Rn1}
\end{equation}
which since $m^{(2)}$ is bounded, is bounded (in norm) by a multiple
of
\begin{multline}
h_{n}^{2}\left(\frac{\beta_{n}}{nh_{n}}\right)^{1/2}\sum_{t=2}^{n}\lvert K_{th}\rvert\left\Vert \begin{bmatrix}[h_{n}^{-1}\rvert x_{t-1}-x\rvert]^{2}\\{}
[h_{n}^{-1}\rvert x_{t-1}-x\lvert]^{3}
\end{bmatrix}\right\Vert \\
=_{(1)}O_{p}\left[h_{n}^{2}\left(\frac{nh_{n}}{\beta_{n}}\right)^{1/2}\right]=O_{p}\left(\frac{nh_{n}^{5}}{\beta_{n}}\right)^{1/2}=_{(2)}o_{p}(1)\label{eq:Rn2}
\end{multline}
where $=_{(1)}$ is by Theorem~\ref{KernelLL}, and $=_{(2)}$ by
the assumption that $nh_{n}^{5}/\beta_{n}\rightarrow0$. Finally,
since $\Lambda_{n}^{-1}S_{n}$ is a martingale with conditional variance
matrix
\begin{multline*}
\sigma_{u}^{2}\frac{\beta_{n}}{nh_{n}}\sum_{t=2}^{n}K_{th}^{2}\begin{bmatrix}1 & \frac{x_{t-1}-x}{h_{n}}\\
\frac{x_{t-1}-x}{h_{n}} & \left(\frac{x_{t-1}-x}{h_{n}}\right)^{2}
\end{bmatrix}\\
\stackrel{d}{\rightarrow}\varrho(0)\int\left[\begin{array}{cc}
1 & u\\
u & u^{2}
\end{array}\right]K^{2}(u)du=:\varrho(0)V
\end{multline*}
by Theorem~\ref{KernelLL}, it follows by Wang's (2014) CLT and
similar arguments as were used in part~(i) that
\begin{equation}
\Lambda_{n}^{-1}S_{n}\stackrel{d}{\rightarrow}\sigma_{u}\varrho(0)^{1/2}V^{1/2}\zeta\label{eq:Svar}
\end{equation}
jointly with (\ref{eq:Hcvg}), where $\zeta\sim N[0,I_{2}]$ is independent
of $\varrho(0)$. Since
\[
\Lambda_{n}[\mathbf{\tilde{m}}(x)-\mathbf{m}(x)]=[\Lambda_{n}^{-1}H_{n}\Lambda_{n}^{-1}]^{-1}\Lambda_{n}^{-1}R_{n}+\Lambda_{n}^{-1}S_{n}
\]
the result now follows from (\ref{eq:Hcvg})--(\ref{eq:Svar}), noting
that the r.h.s.\ of (\ref{eq:Hcvg}) is invertible under condition~(i).
\end{proof}

\section{Proof of Theorem 5.1}

\label{sec:Proof-of-Theorem5.1}

The proof is divided into four parts. Part~I derives the asymptotic
distribution of $\tilde{F}$ under $\mathcal{H}_{0}$. Part~II derives
the asymptotics of $(\hat{\mu},\hat{\gamma})$, and part~III those
of $\tilde{\sigma}_{u}^{2}(x)$, under $\mathcal{H}_{1}$. Finally,
part~IV draws on the results of parts~II and III to derive the asymptotics
of $\tilde{F}$ under $\mathcal{H}_{1}$.

\textbf{I. Asymptotics of $\boldsymbol{\tilde{F}}$ under }$\boldsymbol{\mathcal{H}_{0}}$\textbf{.}
Recall from (\ref{eq:spec_test}) that $\tilde{F}$ is constructed
from an ensemble of $t$ statistics of the form
\[
\tilde{t}(x;\hat{\mu},\hat{\gamma}):=\left[\frac{\sum_{t=2}^{n}K[(x_{t-1}-x)/h_{n}]}{\tilde{\sigma}_{u}^{2}(x)Q_{11}}\right]^{1/2}\left[\tilde{m}_{g}(x)-\hat{\mu}-\hat{\gamma}g(x)\right]
\]
for $x$ taking values in some finite set $\mathcal{X}\subset\mathbb{R}$.
Let $x\in\mathcal{X}$ be fixed. We shall show that:
\begin{equation}
(nh_{n}/\beta_{n})^{1/2}[\tilde{m}_{g}(x)-m(x)]\overset{d}{\rightarrow}\zeta_{x}\sigma_{u}\varrho(0)^{-1/2}Q_{11}^{1/2},\label{eq:mgasy}
\end{equation}
jointly over $x\in\mathcal{X}$, where $\zeta_{x}\sim N[0,1]$ is
independent of $\zeta_{x^{\prime}}\sim N[0,1]$ for each $x,x^{\prime}\in\mathcal{X}$;
and that
\begin{equation}
\tilde{\sigma}_{u}^{2}(x)\overset{p}{\rightarrow}\sigma_{u}^{2}\label{eq:sigmaasy}
\end{equation}
for each $x\in\mathcal{X}$. Since $(\beta_{n}/nh_{n})\sum_{t=2}^{n}K[(x_{t-1}-x)/h_{n}]\overset{d}{\rightarrow}\varrho(0)$
by Theorem~\ref{KernelLL}, and $\hat{\mu}$ and $\hat{\gamma}$
are consistent at rates $n^{-1/2}$ and $n^{-1/2}\kappa_{g}(\beta_{n})=O_{p}(n^{-1/2})$
respectively by Theorem~\ref{LS}, it will then follow that
\begin{align*}
\tilde{t}(x;\hat{\mu},\hat{\gamma}) & =\tilde{t}(x;\mu,\gamma)+O_{p}(h_{n}/\beta_{n})\\
 & =\left[\frac{\sum_{t=2}^{n}K[(x_{t-1}-x)/h_{n}]}{(\sigma_{u}^{2}+o_{p}(1))Q_{11}}\right]^{1/2}\left[\tilde{m}_{g}(x)-m(x)\right]\overset{d}{\rightarrow}\zeta_{x}
\end{align*}
whence 
\[
\tilde{F}=\sum_{x\in\mathcal{X}}\tilde{t}(x;\hat{\mu},\hat{\gamma})^{2}\overset{d}{\rightarrow}\sum_{x\in\mathcal{X}}\zeta_{x}^{2}\sim\chi_{p}^{2}
\]
as required.

It thus remains to prove (\ref{eq:mgasy}) and (\ref{eq:sigmaasy}).
Consider (\ref{eq:mgasy}) first: the proof uses arguments similar
to those used in part~(ii) of the proof of Theorem~\ref{thm:npboth}.
Define $X_{t}^{\prime}:=[1,g(x_{t-1})-g(x)]$, $K_{th}:=K[h_{n}^{-1}(x_{t-1}-x)]$,
$\Lambda_{n}:=(nh_{n}/\beta_{n})^{1/2}\mathrm{diag}\{1,h_{n}\}$ and
$H_{n}:=\sum_{t=2}^{n}K_{th}X_{t}X_{t}^{\prime}$, noting how the
definition of $X_{t}$ has been modified. Recall that $\tilde{m}_{g}$
is obtained by a $K_{th}$-weighted regression of $y_{t}$ on a constant
and $g(x_{t-1})-g(x)$. Therefore under $\mathcal{H}_{0}$, this estimator
suffers from no approximation bias, and standard arguments give
\begin{equation}
\tilde{m}_{g}(x)-m(x)=e_{1}^{\prime}H_{n}^{-1}\sum_{t=2}^{n}K_{th}X_{t}u_{t}=:e_{1}^{\prime}H_{n}^{-1}S_{n}(x),\label{eq:mgdecomp}
\end{equation}
where $e_{1}^{\prime}:=(1,0)$.

We shall consider each of $H_{n}$ and $S_{n}$ in turn. We claim
that
\begin{equation}
\Lambda_{n}^{-1}H_{n}\Lambda_{n}^{-1}\stackrel{d}{\rightarrow}\varrho(0)\int\left[\begin{array}{cc}
1 & g^{(1)}(x)u\\
g^{(1)}(x)u & g^{(1)}(x)^{2}u^{2}
\end{array}\right]K(u)du=:\varrho(0)\tilde{H}(x),\label{eq:H0cvg}
\end{equation}
and recall that $g^{(1)}(x)\neq0$ by condition~(ii). Consider e.g.\ the
$(2,2)$ element of the l.h.s.\ matrix: this is equal to
\begin{multline}
\frac{\beta_{n}}{nh_{n}^{3}}\sum_{t=2}^{n}K_{ht}[g(x_{t-1})-g(x)]^{2}\\
=\frac{\beta_{n}}{nh_{n}^{3}}\sum_{t=2}^{n}K_{ht}[g^{(1)}(x)(x_{t-1}-x)+g^{(2)}(\overline{x}_{t-1})(x_{t-1}-x)^{2}]^{2}\label{eq:spectaylor}
\end{multline}
for some $\overline{x}_{t-1}$ lying between $x_{t-1}$ and $x$,
by a Taylor expansion. By Theorem~\ref{thm:npboth},
\begin{multline}
\frac{\beta_{n}}{nh_{n}^{3}}\sum_{t=2}^{n}K_{ht}(x_{t-1}-x)^{2}\\
=\frac{\beta_{n}}{nh_{n}}\sum_{t=2}^{n}K\left(\frac{x_{t-1}-x}{h_{n}}\right)\left(\frac{x_{t-1}-x}{h_{n}}\right)^{2}\overset{d}{\rightarrow}\varrho(0)\int u^{2}K(u)du,\label{eq:Kprod}
\end{multline}
whereas, since $K(u)u^{4}$ is integrable by condition~(i)
\[
\frac{\beta_{n}}{nh_{n}^{3}}\sum_{t=2}^{n}K_{ht}(x_{t-1}-x)^{4}=O_{p}(h_{n}^{2})=o_{p}(1),
\]
whence by the Cauchy--Schwarz inequality,
\begin{align*}
\frac{\beta_{n}}{nh_{n}^{3}}\sum_{t=2}^{n}K_{ht}[g(x_{t-1})-g(x)]^{2} & =g^{(1)}(x)^{2}\frac{\beta_{n}}{nh_{n}^{3}}\sum_{t=2}^{n}K_{ht}(x_{t-1}-x)^{2}+o_{p}(1)\\
 & \overset{d}{\rightarrow}g^{(1)}(x)^{2}\varrho(0)\int u^{2}K(u)du.
\end{align*}
This gives the claimed convergence for the $(2,2)$ element in (\ref{eq:H0cvg});
the result for the other elements follows by analogous arguments.

We next turn to $S_{n}$. For each $x\in\mathcal{X}$, $\Lambda_{n}^{-1}S_{n}(x)$
is a vector martingale with conditional variance
\begin{multline*}
\sigma_{u}^{2}\frac{\beta_{n}}{nh_{n}}\sum_{t=2}^{n}K_{th}^{2}\begin{bmatrix}1 & \frac{g(x_{t-1})-g(x)}{h_{n}}\\
\frac{g(x_{t-1})-g(x)}{h_{n}} & \left(\frac{g(x_{t-1})-g(x)}{h_{n}}\right)^{2}
\end{bmatrix}\\
\stackrel{d}{\rightarrow}_{(1)}\varrho(0)\int\left[\begin{array}{cc}
1 & g^{(1)}(x)u\\
g^{(1)}(x)u & g^{(1)}(x)^{2}u^{2}
\end{array}\right]K^{2}(u)du=:\varrho(0)V(x)
\end{multline*}
where $\stackrel{d}{\rightarrow}_{(1)}$ follows by arguments identical
to those used to prove (\ref{eq:H0cvg}), with $K^{2}$ in place of
$K$. Moreover, the conditional covariation of $\Lambda_{n}^{-1}S_{n}(x)$
with $\Lambda_{n}^{-1}S_{n}(x^{\prime})$ for $x^{\prime}\neq x$
is a $2\times2$ matrix with $(i,j)$ element
\begin{equation}
\sigma_{u}^{2}\frac{\beta_{n}}{nh_{n}}\sum_{t=2}^{n}K_{th}(x)K_{th}(x^{\prime})\left(\frac{g(x_{t-1})-g(x)}{h_{n}}\right)^{i-1}\left(\frac{g(x_{t-1})-g(x^{\prime})}{h_{n}}\right)^{j-1},\label{eq:covar}
\end{equation}
where $K_{th}(x):=K[h_{n}^{-1}(x_{t-1}-x)]$. We claim this is $o_{p}(1)$,
in which case it will follow by an application of Wang's (2014) CLT
that $\Lambda_{n}^{-1}S_{n}(x)$ and $\Lambda_{n}^{-1}S_{n}(x^{\prime})$
are asymptotically independent, with
\begin{equation}
\Lambda_{n}^{-1}S_{n}(x)\overset{d}{\rightarrow}\varrho^{1/2}(0)V^{1/2}(x)\zeta_{x}\label{eq:S0cvg}
\end{equation}
jointly with (\ref{eq:H0cvg}) over all $x\in\mathcal{X}$, where
$\zeta_{x}\sim N[0,I_{2}]$ is independent of $\zeta_{x^{\prime}}\sim N[0,I_{2}]$
for $x^{\prime}\neq x$.

To verify that (\ref{eq:covar}) is indeed $o_{p}(1)$, note that
by similar arguments to those given in (\ref{eq:spectaylor}) above,
we can bound (\ref{eq:covar}) by linear combinations of functionals
of the form
\[
\sigma_{u}^{2}\frac{\beta_{n}}{nh_{n}}\sum_{t=2}^{n}L_{a}\left(\frac{x_{t-1}-x}{h_{n}}\right)L_{b}\left(\frac{x_{t-1}-x^{\prime}}{h_{n}}\right)
\]
where $L_{a}(u)=\lvert K(u)u^{a}\rvert$ for $a\in\{0,1,2\}$, which
is bounded and integrable by condition~(i). Finally, note that by
the proof of Theorem~\ref{MainLL}, $\{x_{t}(n)\}$ satisfies \textbf{HL3}
(for some $t_{0}\in\mathbb{N}$) and \textbf{HL6}. Hence for $t\geq t_{0}$
and $n\geq n_{0}$,
\begin{align}
 & \frac{1}{h_{n}}\mathbf{E}L_{a}\left(\frac{x_{t}-x}{h_{n}}\right)L_{b}\left(\frac{x_{t}-x^{\prime}}{h_{n}}\right)\nonumber \\
 & \quad=\frac{1}{h_{n}}\int_{\mathbb{R}}L_{a}\left(\frac{\beta_{t}u-x}{h_{n}}\right)L_{b}\left(\frac{\beta_{t}u-x^{\prime}}{h_{n}}\right)\mathcal{D}_{n,t}(u)du\nonumber \\
 & \quad\leq\sup_{n\geq n_{0},t_{0}\leq t\leq n}\sup_{v\in\mathbb{R}}\lvert\mathcal{D}_{n,t}(v)\rvert\beta_{t}^{-1}\int_{\mathbb{R}}L_{a}(u)L_{b}\left(u-h_{n}^{-1}(x^{\prime}-x)\right)du.\label{eq:bddinteg}
\end{align}
Since $L_{4}$ is bounded by condition~(i), it must be the case that
$L_{b}(u)\rightarrow0$ as $\lvert u\rvert\rightarrow\infty$ for
$b\in\{0,1,2\}$, whence the r.h.s.\ integral converges to zero by
the dominated convergence theorem as $h_{n}\rightarrow0$. Thus
\begin{align*}
 & \frac{\beta_{n}}{nh_{n}}\sum_{t=2}^{n}\mathbf{E}L_{a}\left(\frac{x_{t-1}-x}{h_{n}}\right)L_{b}\left(\frac{x_{t-1}-x^{\prime}}{h_{n}}\right)\\
 & \qquad\qquad=\frac{\beta_{n}}{nh_{n}}\sum_{t=t_{0}}^{n-1}\int_{\mathbb{R}}L_{a}\left(\frac{\beta_{t}^{-1}x_{t}-x}{\beta_{t}^{-1}h_{n}}\right)\left(\frac{\beta_{t}^{-1}x_{t}-x}{\beta_{t}^{-1}h_{n}}\right)+o(1)\\
 & \qquad\qquad=o\left(\frac{\beta_{n}}{n}\sum_{t=t_{0}}^{n-1}\beta_{t}^{-1}+1\right)=o(1),
\end{align*}
as was required, where the final equality follows since \textbf{HL6}
holds, as was noted above.

(\ref{eq:mgasy}) now follows from (\ref{eq:mgdecomp}), (\ref{eq:H0cvg}),
(\ref{eq:S0cvg}) and the fact that
\[
Q_{11}=e_{1}^{\prime}\tilde{H}(x)^{-1}V(x)\tilde{H}(x)^{-1}e_{1}
\]
as may be verified by direct calculation.

We turn therefore to (\ref{eq:sigmaasy}). We have
\begin{equation}
\tilde{\sigma}_{u}^{2}(x)=\frac{\sum_{t=2}^{n}\left[(\mu-\hat{\mu})+(\gamma-\hat{\gamma})g(x_{t-1})+u_{t}\right]^{2}K_{th}}{\sum_{t=2}^{n}K_{th}}.\label{eq:sigmatilde}
\end{equation}
Recognising
\begin{align*}
 & \lvert g^{2}(x_{t-1})-g^{2}(x)\rvert\\
 & \qquad=\lvert[g(x_{t-1})-g(x)]^{2}+2[g(x_{t-1})-g(x)]g(x)\rvert\\
 & \qquad\leq2\lvert g^{(1)}(x)\rvert^{2}(x_{t-1}-x)^{2}+2\lvert g^{(2)}(\overline{x}_{t-1})\rvert^{2}(x_{t-1}-x)^{4}\\
 & \qquad\qquad+2\lvert g(x)\rvert\{\lvert g^{(1)}(x)\rvert\lvert x_{t-1}-x\vert+\lvert g^{(2)}(\overline{x}_{t-1})\rvert\lvert x_{t-1}-x\vert^{2}\},
\end{align*}
for some $\overline{x}_{t-1}$ lying between $x_{t-1}$ and $x$,
and that $g^{(2)}$ is bounded, it follows from Theorem~\ref{thm:npboth}
(arguing similarly as in (\ref{eq:Kprod}) above) that
\begin{equation}
\frac{\beta_{n}}{nh_{n}}\sum_{t=2}^{n}K_{th}g^{2}(x_{t-1})=g^{2}(x)\frac{\beta_{n}}{nh_{n}}\sum_{t=2}^{n}K_{th}+O_{p}(h_{n})\overset{d}{\rightarrow}g^{2}(x)\varrho(0).\label{eq:g2lim}
\end{equation}
Since $\left\{ u_{t}^{2}-\sigma_{u}^{2},\mathcal{F}_{t}\right\} $
is a martingale difference sequence and $K_{th}$ is $\mathcal{F}_{t-1}$-measurable,
$(\beta_{n}/nh_{n})\sum_{t=2}^{n}K_{th}(u_{t}^{2}-\sigma_{u}^{2})$
is a martingale with conditional variance

\begin{align*}
\left(\frac{\beta_{n}}{nh_{n}}\right)^{2}\sum_{t=2}^{n}K_{th}^{2}\mathbf{E}[(u_{t}^{2}-\sigma_{u}^{2})^{2}\mid\mathcal{F}_{t-1}] & \leq_{(1)}Z\left(\frac{\beta_{n}}{nh_{n}}\right)^{2}\sum_{t=t_{0}}^{n}K_{th}^{2}+o_{p}(1)\\
 & =_{(2)}O_{p}\left(\frac{\beta_{n}}{nh_{n}}\right)+o_{p}(1)\\
 & \rightarrow0
\end{align*}
where $\leq_{(1)}$ holds by $Z:=\sup_{t}\mathbf{E}[u_{t}^{4}\mid\mathcal{F}_{t-1}]<\infty$
a.s., and $=_{(2)}$ by Theorem~\ref{MainLL}. Thus $(\beta_{n}/nh_{n})\sum_{t=2}^{n}K_{th}(u_{t}^{2}-\sigma_{u}^{2})=o_{p}(1)$
by Corollary~3.1 in Hall and Heyde (1980), whence
\begin{equation}
\frac{\sum_{t=2}^{n}u_{t}^{2}K_{th}}{\sum_{t=2}^{n}K_{th}}=\sigma_{u}^{2}+\frac{\sum_{t=2}^{n}K_{th}(u_{t}^{2}-\sigma_{u}^{2})}{\sum_{t=2}^{n}K_{th}}\stackrel{p}{\rightarrow}\sigma_{u}^{2}.\label{eq:truesig}
\end{equation}
Since $(\beta_{n}/nh_{n})\sum_{t=2}^{n}K_{th}\overset{d}{\rightarrow}\varrho(0)$
by Theorem~\ref{MainLL}, (\ref{eq:g2lim}), (\ref{eq:truesig}),
the consistency of $(\hat{\mu},\hat{\gamma})$ (from Theorem~\ref{LS})
and applications of the Cauchy--Schwarz inequality to the cross-product
terms on the r.h.s.\ of (\ref{eq:sigmatilde}) yield
\[
\tilde{\sigma}_{u}^{2}(x)=\frac{\sum_{t=2}^{n}u_{t}^{2}K_{th}}{\sum_{t=2}^{n}K_{th}}+o_{p}(1)\stackrel{p}{\rightarrow}\sigma_{u}^{2},
\]
i.e.\ (\ref{eq:sigmaasy}) holds.

\textbf{II. Asymptotics of $\boldsymbol{(\hat{\mu},\hat{\gamma})}$
under $\boldsymbol{\mathcal{H}_{1}}$.} Recall that 
\[
m(x)=\mu+\gamma g(x)+r_{n}g_{1}(x)
\]
 under $\mathcal{H}_{1}$. Let $\kappa_{gn}:=\kappa_{g}(\beta_{n})$
and $D_{n}:=n^{1/2}\mathrm{diag}\{1,\kappa_{gn}\}$.

\textbf{(a).} Suppose condition~(iii.a) holds. Let $R_{n}^{\prime}:=r_{n}^{-1}n^{-1/2}\beta_{n}$,
which is $o(1)$ by assumption. By (\ref{eq:olsdecomp}) and subsequent
arguments given in the proof of Theorem~\ref{LS},
\begin{align*}
R_{n}^{\prime}D_{n}\begin{bmatrix}\hat{\mu}-\mu\\
\hat{\gamma}-\gamma
\end{bmatrix} & =M_{n}^{-1}\frac{R_{n}^{\prime}}{n^{1/2}}\sum_{t=2}^{n}\begin{bmatrix}1\\
\kappa_{gn}^{-1}g(x_{t-1})
\end{bmatrix}[u_{t}+r_{n}g_{1}(x_{t-1})]\\
 & =:M_{n}^{-1}[R_{n}^{\prime}S_{n}+R_{n}^{\prime}S_{n}^{\dagger}]
\end{align*}
where $M_{n}\overset{d}{\rightarrow}M_{H_{g}}$ and $S_{n}=O_{p}(1)$
by (\ref{eq:Mncvg}) and (\ref{eq:SnOLS}) respectively, and
\[
R_{n}^{\prime}S_{n}^{\dagger}=\frac{\beta_{n}}{n}\sum_{t=2}^{n}\begin{bmatrix}g_{1}(x_{t-1})\\
\kappa_{gn}^{-1}g(x_{t-1})g_{1}(x_{t-1})
\end{bmatrix}\overset{d}{\rightarrow}\begin{bmatrix}\varrho(0)\int g_{1}\\
0
\end{bmatrix}
\]
by Theorem~\ref{KernelLL}, since $g_{1}$ and $g\cdot g_{1}$ are
integrable by assumption, and $\kappa_{gn}\rightarrow\infty$. Since
$R_{n}^{\prime}\rightarrow0$, it follows that
\[
R_{n}^{\prime}D_{n}\begin{bmatrix}\hat{\mu}-\mu\\
\hat{\gamma}-\gamma
\end{bmatrix}\overset{d}{\rightarrow}\left\{ \intop\left[\begin{array}{cc}
1 & H_{g}\\
H_{g} & H_{g^{2}}
\end{array}\right]\varrho\right\} ^{-1}\begin{bmatrix}\varrho(0)\int g_{1}\\
0
\end{bmatrix}.
\]

\textbf{(b).} Suppose condition~(iii.b) holds. Let $\kappa_{g_{1}n}:=\kappa_{g_{1}}(\beta_{n})$
and $R_{n}^{\prime\prime}:=r_{n}^{-1}n^{-1/2}\kappa_{g_{1}n}^{-1}$;
the latter is $o(1)$ by assumption. By analogous arguments to those
given in part~(a),
\begin{align}
R_{n}^{\prime\prime}D_{n}\begin{bmatrix}\hat{\mu}-\mu\\
\hat{\gamma}-\gamma
\end{bmatrix} & =M_{n}^{-1}\frac{R_{n}^{\prime\prime}}{n^{1/2}}\sum_{t=2}^{n}\begin{bmatrix}1\\
\kappa_{gn}^{-1}g(x_{t-1})
\end{bmatrix}r_{n}g_{1}(x_{t-1})+o_{p}(1)\nonumber \\
 & =M_{n}^{-1}\frac{1}{n}\sum_{t=2}^{n}\begin{bmatrix}\kappa_{g_{1}n}^{-1}g_{1}(x_{t-1})\\
\kappa_{gn}^{-1}g(x_{t-1})\kappa_{g_{1}n}^{-1}g_{1}(x_{t-1})
\end{bmatrix}+o_{p}(1)\nonumber \\
 & =_{(1)}\frac{1}{n}\sum_{t=2}^{n}\begin{bmatrix}H_{g_{1}}(x_{t-1}/\beta_{n})\\
H_{g}(x_{t-1}/\beta_{n})H_{g_{1}}(x_{t-1}/\beta_{n})
\end{bmatrix}+o_{p}(1)\nonumber \\
 & \overset{d}{\rightarrow}_{(2)}\left\{ \intop\left[\begin{array}{cc}
1 & H_{g}\\
H_{g} & H_{g^{2}}
\end{array}\right]\varrho\right\} ^{-1}\begin{bmatrix}\int H_{g_{1}}\varrho\\
\int H_{g}H_{g_{1}}\varrho
\end{bmatrix}=:-\begin{bmatrix}\mu_{*}\\
\gamma_{*}
\end{bmatrix},\label{eq:limdistols}
\end{align}
where $=_{(1)}$ follows by straightforward calculations, since $g$
and $g_{1}$ are AHF, and $\overset{d}{\rightarrow}_{(2)}$ follows
by Theorem~\ref{MainLL}.

\textbf{III. Asymptotics of $\boldsymbol{\tilde{\sigma}_{u}^{2}(x)}$
under $\boldsymbol{\mathcal{H}_{1}}$.} We have
\begin{equation}
\tilde{\sigma}_{u}^{2}(x)=\frac{\sum_{t=2}^{n}\left[(\mu-\hat{\mu})+(\gamma-\hat{\gamma})g(x_{t-1})+r_{n}g_{1}(x_{t-1})+u_{t}\right]^{2}K_{th}}{\sum_{t=2}^{n}K_{th}}.\label{eq:newsigmatilde}
\end{equation}

\textbf{(a).} Suppose condition~(iii.a) holds. The argument here
is similar to the proof of (\ref{eq:sigmaasy}) given in part~I.
Recall (\ref{eq:g2lim}), noting that this also holds with $g_{1}$
in place of $g$, since $g_{1}$ is also assumed to have bounded second
derivative. Thus
\[
\frac{\sum_{t=2}^{n}f^{2}(x_{t-1})K_{th}}{\sum_{t=2}^{n}K_{th}}\overset{p}{\rightarrow}f^{2}(x)
\]
for $f\in\{g,g_{1}\}$. By the analysis of part~II(a) of the proof,
and noting that both diagonal elements of
\begin{equation}
R_{n}^{\prime}D_{n}=(r_{n}^{-1}n^{-1/2}\beta_{n})n^{1/2}\mathrm{diag}\{1,\kappa_{gn}\}=r_{n}^{-1}\beta_{n}\mathrm{diag}\{1,\kappa_{gn}\}\label{eq:RpD}
\end{equation}
are divergent, it follows that $(\hat{\mu},\hat{\gamma})\overset{p}{\rightarrow}(\mu,\gamma)$.
In view of $r_{n}=o(1)$ and (\ref{eq:truesig}), the preceding facts
and applications of the Cauchy--Schwarz inequality to the cross-product
terms on the r.h.s.\ of (\ref{eq:newsigmatilde}) yield
\[
\tilde{\sigma}_{u}^{2}(x)=\frac{\sum_{t=2}^{n}u_{t}^{2}K_{th}}{\sum_{t=2}^{n}K_{th}}+o_{p}(1)\stackrel{p}{\rightarrow}\sigma_{u}^{2}.
\]

\textbf{(b).} Suppose condition~(iii.b) holds. Recall the analysis
given in part~II(b) of the proof, and note that
\begin{align}
R_{n}^{\prime\prime}D_{n} & =(r_{n}^{-1}n^{-1/2}\kappa_{g_{1}n}^{-1})n^{1/2}\mathrm{diag}\{1,\kappa_{gn}\}\nonumber \\
 & =\mathrm{diag}\{r_{n}^{-1}\kappa_{g_{1}n}^{-1},r_{n}^{-1}\kappa_{g_{1}n}^{-1}\kappa_{gn}\}.\label{eq:RppD}
\end{align}
The behaviour of these sequences, and thus that of $\tilde{\sigma}_{u}^{2}(x)$,
will depend on $\kappa_{\ast\ast}:=\lim_{n\rightarrow\infty}r_{n}\kappa_{g_{1}n}$;
we need to separately consider the cases where: $\kappa_{\ast\ast}=0$;
$\kappa_{\ast\ast}\in(0,\infty)$; or $\kappa_{\ast\ast}=\infty$.
By a suitable rescaling of $r_{n}$, it is without loss of generality
to normalise $\kappa_{\ast\ast}=1$ in the second of these cases.

Suppose $\kappa_{\ast\ast}=0$. Then both diagonal element of $R_{n}^{\prime\prime}D_{n}$
are divergent, and $(\hat{\mu},\hat{\gamma})\overset{p}{\rightarrow}(\mu,\gamma)$.
The same arguments given in part~III(a) thus imply that $\tilde{\sigma}_{u}^{2}(x)\stackrel{p}{\rightarrow}\sigma_{u}^{2}$.

Suppose $\kappa_{\ast\ast}=1$. Only the second element diagonal element
of $R_{n}^{\prime\prime}D_{n}$ diverges, so $\hat{\gamma}\overset{p}{\rightarrow}\gamma$;
and since the first diagonal element converges to $\kappa_{\ast\ast}^{-1}=1$,
$\hat{\mu}-\mu\stackrel{d}{\rightarrow}-\mu_{\ast}$ by (\ref{eq:limdistols}).
Thus in this case, we have
\begin{align*}
\tilde{\sigma}_{u}^{2}(x) & =\frac{\sum_{t=2}^{n}\left[(\mu-\hat{\mu})+u_{t}\right]^{2}K_{th}}{\sum_{t=2}^{n}K_{th}}+o_{p}(1)\\
 & =\mu_{\ast}^{2}+\sigma_{u}^{2}+2\frac{(\mu-\hat{\mu})\sum_{t=2}^{n}K_{th}u_{t}}{\sum_{t=2}^{n}K_{th}}+o_{p}(1)\\
 & \stackrel{d}{\rightarrow}_{(1)}\mu_{\ast}^{2}+\sigma_{u}^{2}
\end{align*}
where $\stackrel{d}{\rightarrow}_{(1)}$ follows by $\sum_{t=2}^{n}K_{th}u_{t}=O_{p}(\beta_{n}/nh_{n})^{1/2}$,
as follows e.g.\ from (\ref{eq:mhatmgclt}) in the proof of Theorem~\ref{thm:npboth}. 

Suppose $\kappa_{\ast\ast}=\infty$. In this case, (\ref{eq:RppD})
implies that $\hat{\mu}-\mu$ is divergent at rate $r_{n}\kappa_{g_{1}n}$,
and dominates $\hat{\gamma}-\gamma$. Thus by (\ref{eq:limdistols})
\begin{align*}
(r_{n}\kappa_{g_{1}n})^{-2}\tilde{\sigma}_{u}^{2}(x) & =(r_{n}\kappa_{g_{1}n})^{-2}(\mu-\hat{\mu})^{2}+o_{p}(1)\stackrel{d}{\rightarrow}\mu_{\ast}^{2}.
\end{align*}

In summary, in each case $\varsigma_{n}^{-1}\tilde{\sigma}_{u}^{2}(x)\stackrel{d}{\rightarrow}\sigma_{\ast}^{2}$,
where
\begin{equation}
[\varsigma_{n},\sigma_{*}^{2}]:=\left\{ \begin{alignedat}{4} & [1, & \  & \sigma_{u}^{2} &  & ] & \qquad & \text{if }\kappa_{\ast\ast}=0\\
 & [r_{n}^{2}\kappa_{g_{1}n}^{2}, &  & \mu_{\ast}^{2}+\sigma_{u}^{2} &  & ] &  & \text{if }\kappa_{\ast\ast}=1\\
 & [r_{n}^{2}\kappa_{g_{1}n}^{2}, &  & \mu_{\ast}^{2} &  & ] &  & \text{if }\kappa_{\ast\ast}=\infty.
\end{alignedat}
\right.\label{eq:spec_lim_var-1}
\end{equation}
Note that in the case where $\kappa_{\ast\ast}=1$, we have $\lim_{n\rightarrow\infty}r_{n}\kappa_{g_{1}n}=\kappa_{\ast\ast}=1$,
and so we may equivalently take either $\varsigma_{n}=1$ or $\varsigma_{n}=r_{n}^{2}\kappa_{g_{1}n}^{2}$;
we present (\ref{eq:spec_lim_var-1}) in terms of the latter to facilitate
the next part of the proof.

\textbf{IV. Asymptotics of $\boldsymbol{\tilde{F}}$ under $\boldsymbol{\mathcal{H}_{1}}$.}
Since $g_{1}$ has bounded second derivative (by condition~(v)),
arguments similar to those given in part~I of this proof and in the
proof of Theorem~\ref{thm:npboth} yield that, so long as $\tilde{\sigma}_{u}^{2}(x)\stackrel{p}{\rightarrow}\sigma_{u}^{2}$
(and $nh_{n}^{5}/\beta_{n}\rightarrow0$ as assumed),
\begin{align}
t^{*}(x;\mu,\gamma) & :=\left[\frac{\sum_{t=2}^{n}K_{th}}{\tilde{\sigma}_{u}^{2}(x)Q_{11}}\right]^{1/2}\left[\tilde{m}_{g}(x)-\mu-\gamma g(x)-r_{n}g_{1}(x)\right]\nonumber \\
 & \stackrel{d}{\rightarrow}N[0,1].\label{eq:spec_t-0}
\end{align}
For future reference, we note here that
\begin{multline}
\tilde{t}(x;\hat{\mu},\hat{\gamma})-t^{*}(x;\mu,\gamma)\\
=\left[\frac{\sum_{t=2}^{n}K_{th}}{\tilde{\sigma}_{u}^{2}(x)Q_{11}}\right]^{1/2}[(\mu-\hat{\mu})+(\gamma-\hat{\gamma})g(x)+r_{n}g_{1}(x)]\label{eq:tdecomp}
\end{multline}
and that if $\tilde{\sigma}_{u}^{2}(x)\stackrel{p}{\rightarrow}\sigma_{u}^{2}$,
then by Theorem~\ref{KernelLL}
\begin{equation}
U_{n}:=\frac{\frac{\beta_{n}}{nh_{n}}\sum_{t=2}^{n}K[(x_{t-1}-x)/h_{n}]}{\tilde{\sigma}_{u}^{2}(x)Q_{11}}\stackrel{d}{\rightarrow}\frac{\varrho(0)}{\sigma_{u}^{2}Q_{11}}=:U,\label{eq:Un}
\end{equation}
which is a.s.\ nonzero.

We show below that in each of the cases contemplated by condition~(iii),
$\delta_{n}\tilde{t}(x;\hat{\mu},\hat{\gamma})\overset{d}{\rightarrow}\omega(x)$,
where $\delta_{n}\rightarrow0$, and $\omega(x)$ is nonzero for at
least one $x\in\mathcal{X}$. That $\tilde{F}\overset{p}{\rightarrow}\infty$
then follows immediately from (\ref{eq:spec_test}).

\textbf{(a).} Suppose condition~(iii.a) holds. By the results of
parts~II(a) and III(a), and (\ref{eq:RpD}) in particular,
\begin{align*}
(\mu-\hat{\mu})+(\gamma-\hat{\gamma})g(x)+r_{n}g_{1}(x) & =O_{p}(r_{n}\beta_{n}^{-1})+O_{p}(r_{n}\beta_{n}^{-1}\kappa_{gn}^{-1})+r_{n}g_{1}(x).\\
 & =r_{n}[g_{1}(x)+o_{p}(1)].
\end{align*}
Thus by (\ref{eq:tdecomp}) and recalling from part~III(a) that $\tilde{\sigma}_{u}^{2}(x)\stackrel{p}{\rightarrow}\sigma_{u}^{2}$
\begin{align*}
\tilde{t}(x;\hat{\mu},\hat{\gamma}) & =t^{*}(x;\mu,\gamma)+o_{p}(1)\\
 & \qquad\qquad+\left[\frac{\sum_{t=2}^{n}K_{th}}{\tilde{\sigma}_{u}^{2}(x)Q_{11}}\right]^{1/2}r_{n}[g_{1}(x)+o_{p}(1)]
\end{align*}
whence, noting that $\beta_{n}/r_{n}^{2}nh_{n}\rightarrow0$ by condition~(iv),
\begin{align*}
\left(\frac{\beta_{n}}{r_{n}^{2}nh_{n}}\right)^{1/2}\tilde{t}(x;\hat{\mu},\hat{\gamma}) & =o_{p}(1)+U_{n}^{1/2}g_{1}(x)\stackrel{d}{\rightarrow}U^{1/2}g_{1}(x)
\end{align*}
as required, by (\ref{eq:Un}). By condition~(vi), $g_{1}(x)\neq0$
for at least one $x\in\mathcal{X}$.

\textbf{(b).} Suppose condition (iii.b) holds. Recall $\kappa_{\ast}:=\lim_{n\rightarrow\infty}\kappa_{g_{1}n}$
and $\kappa_{*\ast}:=\lim_{n\rightarrow\infty}\kappa_{g_{1}n}r_{n}$.
We consider each of the cases contemplated by condition~(iv), in
turn.

Suppose $\kappa_{*}\in[0,\infty)$, which implies $\kappa_{\ast\ast}=0$.
By the results of part~II(b) and (\ref{eq:RppD}),
\begin{gather*}
r_{n}^{-1}(\hat{\mu}-\mu)=\kappa_{g_{1}n}(r_{n}\kappa_{g_{1}n})^{-1}(\hat{\mu}-\mu)\overset{d}{\rightarrow}-\kappa_{\ast}\mu_{\ast}\\
r_{n}^{-1}(\hat{\gamma}-\gamma)=O_{p}(\kappa_{g_{1}n}\kappa_{gn}^{-1})=o_{p}(1).
\end{gather*}
By part~III(b), $\tilde{\sigma}_{u}^{2}(x)\stackrel{p}{\rightarrow}\sigma_{u}^{2}$
, and thus, similarly to the argument given in part~IV(a), noting
that $\beta_{n}/r_{n}^{2}nh_{n}\rightarrow0$ also in the present
case,
\begin{align*}
\left(\frac{\beta_{n}}{r_{n}^{2}nh_{n}}\right)^{1/2}\tilde{t}(x;\hat{\mu},\hat{\gamma}) & =o_{p}(1)+U_{n}^{1/2}[r_{n}^{-1}(\hat{\mu}-\mu)+g_{1}(x)]\\
 & \stackrel{d}{\rightarrow}U^{1/2}[\kappa_{*}\mu_{*}+g_{1}(x)],
\end{align*}
and note that by condition~(vi), $\kappa_{*}\mu_{*}+g_{1}(x)\neq0$
for at least one $x\in\mathcal{X}$.

Suppose $\kappa_{*}=\infty$ and $\kappa_{\ast\ast}=0$. In this case,
it remains true that $\tilde{\sigma}_{u}^{2}(x)\stackrel{p}{\rightarrow}\sigma_{u}^{2}$
, by part~III(b) of the proof. Moreover, by the results of part~II(b)
and (\ref{eq:RppD}),
\begin{gather}
r_{n}^{-1}\kappa_{g_{1}n}^{-1}(\hat{\mu}-\mu)\overset{d}{\rightarrow}-\mu_{\ast}\label{eq:resclmu}\\
r_{n}^{-1}\kappa_{g_{1}n}^{-1}(\hat{\gamma}-\gamma)=O_{p}(\kappa_{gn}^{-1})=o_{p}(1).\label{eq:resclgam}
\end{gather}
Thus under the assumption that $\beta_{n}/r_{n}^{2}\kappa_{g_{1}n}^{2}nh_{n}\rightarrow0$,
similar arguments as were used in the previous case yield
\begin{align*}
\left(\frac{\beta_{n}}{r_{n}^{2}\kappa_{g_{1}n}^{2}nh_{n}}\right)^{1/2}\tilde{t}(x;\hat{\mu},\hat{\gamma}) & =o_{p}(1)+U_{n}^{1/2}\{r_{n}^{-1}\kappa_{g_{1}n}^{-1}(\hat{\mu}-\mu)+o_{p}(1)\}\\
 & \stackrel{d}{\rightarrow}U^{1/2}\mu_{*}
\end{align*}
which is nonzero a.s.\ by condition~(vi).

Suppose $\kappa_{**}\in(0,\infty]$, which implies $\kappa_{*}=\infty$;
as in part~III(b), normalise $\kappa_{\ast\ast}=1$ if $\kappa_{\ast\ast}\in(0,\infty)$.
We have
\[
(r_{n}\kappa_{g_{1}n})^{-2}\tilde{\sigma}_{u}^{2}(x)\stackrel{d}{\rightarrow}\sigma_{\ast}^{2},
\]
for $\sigma_{\ast}^{2}$ as in (\ref{eq:spec_lim_var-1}). Thus $\tilde{\sigma}_{u}^{2}(x)$
is bounded away from zero, and possibly divergent, whence $t^{*}(x;\mu,\gamma)=O_{p}(1).$
In view of (\ref{eq:resclmu})--(\ref{eq:resclgam}) and the fact
that $\beta_{n}/nh_{n}\rightarrow0$, it follows that
\begin{align*}
\left(\frac{\beta_{n}}{nh_{n}}\right)^{1/2}\tilde{t}(x;\hat{\mu},\hat{\gamma}) & =o_{p}(1)+\left[\frac{\frac{\beta_{n}}{nh_{n}}\sum_{t=2}^{n}K_{th}}{(r_{n}\kappa_{g_{1}n})^{-2}\tilde{\sigma}_{u}^{2}(x)Q_{11}}\right]^{1/2}\\
 & \qquad\qquad\qquad\qquad\qquad\cdot\{r_{n}^{-1}\kappa_{g_{1}n}^{-1}(\mu-\hat{\mu})+o_{p}(1)\}\\
 & \stackrel{d}{\rightarrow}\left[\frac{\varrho(0)}{\sigma_{*}^{2}Q_{11}}\right]^{1/2}\mu_{*},
\end{align*}
which is nonzero a.s.\ by condition~(vi).\hfill{}\qedsymbol{}

\newpage{}

\section{Additional simulations for Section~6}

\label{sec:Additional-simulations-for}

This section provides the results of some additional simulations not
reported in the main text. Table~\ref{tbl:frac-other-power} is the
counterpart of Table~2 in the main text, reporting size-adjusted
power of the procedure for a type II fractional process, relative
to Wang and Phillips (2012), for the cases where $d\in\{0.25,0.75\}$.

We also repeated the simulation exercise described in Section~6,
but now allowing $x_{t}$ to be mildly (or nearly) integrated, that
is of the form
\begin{align*}
x_{t}(n) & =(1-\kappa_{n}^{-1})x_{t-1}(n)+\xi_{t} & \kappa_{n} & =n^{\alpha_{\kappa}}
\end{align*}
with $x_{0}(n)=0$. All other aspects of the simulation design are
exactly as in Section~6. We generated data for $\alpha_{\kappa}\in\{0.25,0.50,0.75,1.00\}$;
this last corresponds to a nearly integrated process. Size results
are given in Table~\ref{tbl:mi-size}, while size-adjusted power
relative to Wang and Phillips (2012) is reported in Tables~\ref{tbl:power-mi-low-alpha}
and \ref{tbl:power-mi-high-alpha}.

\begin{table}[H]
\caption[caption]{MI processes: $x_{t}=(1-\kappa_{n}^{-1})x_{t-1}+\xi_{t}$; $\kappa_{n}=n^{\alpha_{\kappa}}$
\\ Size: maximum rejection frequency over $\rho\in\{-0.5,0,0.5\}$;
$\alpha=0.1$}
\label{tbl:mi-size}%
\begin{tabular}{cccccccccccccc}
\toprule 
\addlinespace
$\boldsymbol{\alpha_{\kappa}}$ & $\boldsymbol{n}$ &  & \multicolumn{3}{c}{\textbf{WP (2012)}} &  & \multicolumn{3}{c}{$\boldsymbol{p=17}$} &  & \multicolumn{3}{c}{$\boldsymbol{p=25}$}\tabularnewline\addlinespace
\multicolumn{2}{c}{$h=n^{b}$, $b=$} &  & -0.2 & -0.1 & -0.05 &  & -0.2 & -0.1 & -0.05 &  & -0.2 & -0.1 & -0.05\tabularnewline\addlinespace
\midrule
\addlinespace
\textbf{$0.25$} & 100 &  & 0.02 & 0.01 & 0.00 &  & 0.06 & 0.03 & 0.02 &  & 0.07 & 0.03 & 0.02\tabularnewline
 & 200 &  & 0.03 & 0.01 & 0.00 &  & 0.06 & 0.04 & 0.02 &  & 0.08 & 0.04 & 0.03\tabularnewline
 & 500 &  & 0.04 & 0.01 & 0.00 &  & 0.07 & 0.04 & 0.03 &  & 0.08 & 0.05 & 0.03\tabularnewline\addlinespace
\textbf{$0.50$} & 100 &  & 0.04 & 0.02 & 0.01 &  & 0.07 & 0.05 & 0.04 &  & 0.09 & 0.06 & 0.04\tabularnewline
 & 200 &  & 0.06 & 0.03 & 0.02 &  & 0.08 & 0.06 & 0.05 &  & 0.10 & 0.08 & 0.06\tabularnewline
 & 500 &  & 0.08 & 0.05 & 0.03 &  & 0.08 & 0.07 & 0.07 &  & 0.10 & 0.09 & 0.08\tabularnewline\addlinespace
\textbf{$0.75$} & 100 &  & 0.06 & 0.03 & 0.02 &  & 0.08 & 0.06 & 0.05 &  & 0.10 & 0.08 & 0.06\tabularnewline
 & 200 &  & 0.08 & 0.05 & 0.04 &  & 0.07 & 0.07 & 0.06 &  & 0.10 & 0.09 & 0.08\tabularnewline
 & 500 &  & 0.09 & 0.07 & 0.06 &  & 0.08 & 0.08 & 0.08 &  & 0.09 & 0.10 & 0.10\tabularnewline\addlinespace
$1.00$ & 100 &  & 0.06 & 0.04 & 0.03 &  & 0.08 & 0.07 & 0.06 &  & 0.11 & 0.09 & 0.08\tabularnewline
 & 200 &  & 0.08 & 0.05 & 0.04 &  & 0.08 & 0.08 & 0.08 &  & 0.10 & 0.10 & 0.10\tabularnewline
 & 500 &  & 0.09 & 0.07 & 0.07 &  & 0.08 & 0.08 & 0.09 &  & 0.09 & 0.10 & 0.11\tabularnewline\addlinespace
\bottomrule
\end{tabular}
\end{table}

\begin{table}[H]
\caption[caption]{Fractional type II processes \\ Size-adjusted power when $d\in\{0.25,0.75\}$;
$\alpha=0.1$}
\label{tbl:frac-other-power}%
\begin{tabular}{lccccccccccccc}
\toprule 
\addlinespace
 & $\boldsymbol{n}$ &  & \multicolumn{3}{c}{\textbf{WP (2012)}} &  & \multicolumn{3}{c}{$\boldsymbol{p=17}$} &  & \multicolumn{3}{c}{$\boldsymbol{p=25}$}\tabularnewline\addlinespace
\multicolumn{2}{l}{$h=n^{b}$, $b=$} &  & -0.2 & -0.1 & -0.05 &  & -0.2 & -0.1 & -0.05 &  & -0.2 & -0.1 & -0.05\tabularnewline\addlinespace
\midrule
\addlinespace
\multicolumn{2}{l}{$\boldsymbol{d=0.25}$} &  & \multicolumn{3}{c}{\textbf{Size adj.\ power}} &  & \multicolumn{7}{c}{\textbf{Size adjusted, relative to WP}}\tabularnewline\addlinespace
$\varphi_{1}(x)$ & 100 &  & 0.07 & 0.05 & 0.04 &  & 0.11 & 0.10 & 0.09 &  & 0.14 & 0.12 & 0.10\tabularnewline
 & 200 &  & 0.16 & 0.14 & 0.12 &  & 0.15 & 0.17 & 0.16 &  & 0.19 & 0.20 & 0.19\tabularnewline
 & 500 &  & 0.47 & 0.53 & 0.51 &  & 0.20 & 0.20 & 0.20 &  & 0.24 & 0.23 & 0.23\tabularnewline\addlinespace
$\varphi_{1}(2x)$ & 100 &  & 0.08 & 0.05 & 0.03 &  & 0.10 & 0.08 & 0.07 &  & 0.13 & 0.10 & 0.08\tabularnewline
 & 200 &  & 0.15 & 0.12 & 0.08 &  & 0.14 & 0.13 & 0.11 &  & 0.18 & 0.15 & 0.13\tabularnewline
 & 500 &  & 0.47 & 0.45 & 0.37 &  & 0.18 & 0.17 & 0.16 &  & 0.22 & 0.21 & 0.20\tabularnewline\addlinespace
$|x|^{-2}\wedge1$ & 100 &  & 0.06 & 0.03 & 0.02 &  & 0.14 & 0.13 & 0.13 &  & 0.18 & 0.16 & 0.15\tabularnewline
$(\times0.5)$ & 200 &  & 0.11 & 0.07 & 0.04 &  & 0.20 & 0.22 & 0.21 &  & 0.24 & 0.25 & 0.24\tabularnewline
 & 500 &  & 0.35 & 0.27 & 0.20 &  & 0.14 & 0.12 & 0.13 &  & 0.16 & 0.14 & 0.14\tabularnewline\addlinespace
$|x|^{-1}\wedge1$ & 100 &  & 0.05 & 0.02 & 0.01 &  & 0.09 & 0.09 & 0.08 &  & 0.12 & 0.11 & 0.09\tabularnewline
$(\times0.5)$ & 200 &  & 0.06 & 0.04 & 0.03 &  & 0.13 & 0.15 & 0.15 &  & 0.16 & 0.18 & 0.17\tabularnewline
 & 500 &  & 0.14 & 0.13 & 0.11 &  & 0.19 & 0.22 & 0.22 &  & 0.23 & 0.25 & 0.25\tabularnewline\addlinespace
$|x|^{1.5}$ & 100 &  & 0.10 & 0.09 & 0.07 &  & 0.03 & 0.04 & 0.04 &  & 0.05 & 0.05 & 0.04\tabularnewline
$(\times0.02)$ & 200 &  & 0.25 & 0.25 & 0.23 &  & 0.03 & 0.04 & 0.04 &  & 0.06 & 0.06 & 0.05\tabularnewline
 & 500 &  & 0.73 & 0.78 & 0.77 &  & 0.04 & 0.07 & 0.06 &  & 0.06 & 0.09 & 0.08\tabularnewline\addlinespace
$x^{2}$ & 100 &  & 0.06 & 0.04 & 0.03 &  & 0.05 & 0.06 & 0.05 &  & 0.07 & 0.07 & 0.06\tabularnewline
$(\times0.02)$ & 200 &  & 0.12 & 0.11 & 0.10 &  & 0.06 & 0.08 & 0.07 &  & 0.09 & 0.10 & 0.09\tabularnewline
 & 500 &  & 0.39 & 0.44 & 0.44 &  & 0.10 & 0.14 & 0.16 &  & 0.13 & 0.18 & 0.19\tabularnewline\addlinespace
\midrule
\addlinespace
\multicolumn{2}{l}{$\boldsymbol{d=0.75}$} &  & \multicolumn{3}{c}{\textbf{Size adj.\ power}} &  & \multicolumn{7}{c}{\textbf{Size adjusted, relative to WP}}\tabularnewline\addlinespace
$\varphi_{1}(x)$ & 100 &  & 0.07 & 0.06 & 0.05 &  & 0.06 & 0.08 & 0.08 &  & 0.10 & 0.11 & 0.10\tabularnewline
 & 200 &  & 0.12 & 0.11 & 0.10 &  & 0.07 & 0.10 & 0.11 &  & 0.11 & 0.14 & 0.14\tabularnewline
 & 500 &  & 0.20 & 0.23 & 0.23 &  & 0.07 & 0.12 & 0.14 &  & 0.12 & 0.16 & 0.17\tabularnewline\addlinespace
$\varphi_{1}(2x)$ & 100 &  & 0.07 & 0.05 & 0.04 &  & 0.05 & 0.06 & 0.06 &  & 0.08 & 0.08 & 0.08\tabularnewline
 & 200 &  & 0.10 & 0.08 & 0.06 &  & 0.05 & 0.07 & 0.07 &  & 0.09 & 0.10 & 0.09\tabularnewline
 & 500 &  & 0.14 & 0.13 & 0.12 &  & 0.04 & 0.07 & 0.08 &  & 0.08 & 0.11 & 0.11\tabularnewline\addlinespace
$|x|^{-2}\wedge1$ & 100 &  & 0.07 & 0.04 & 0.03 &  & 0.09 & 0.11 & 0.11 &  & 0.12 & 0.14 & 0.14\tabularnewline
$(\times0.5)$ & 200 &  & 0.09 & 0.06 & 0.05 &  & 0.10 & 0.13 & 0.14 &  & 0.15 & 0.18 & 0.18\tabularnewline
 & 500 &  & 0.11 & 0.10 & 0.08 &  & 0.07 & 0.12 & 0.13 &  & 0.12 & 0.15 & 0.16\tabularnewline\addlinespace
$|x|^{-1}\wedge1$ & 100 &  & 0.07 & 0.04 & 0.03 &  & 0.07 & 0.08 & 0.08 &  & 0.09 & 0.11 & 0.11\tabularnewline
$(\times0.5)$ & 200 &  & 0.08 & 0.07 & 0.06 &  & 0.07 & 0.11 & 0.12 &  & 0.11 & 0.14 & 0.15\tabularnewline
 & 500 &  & 0.11 & 0.11 & 0.11 &  & 0.08 & 0.13 & 0.15 &  & 0.13 & 0.16 & 0.18\tabularnewline\addlinespace
$|x|^{1.5}$ & 100 &  & 0.10 & 0.09 & 0.07 &  & 0.03 & 0.05 & 0.05 &  & 0.06 & 0.07 & 0.07\tabularnewline
$(\times0.02)$ & 200 &  & 0.18 & 0.18 & 0.18 &  & 0.04 & 0.06 & 0.08 &  & 0.07 & 0.09 & 0.10\tabularnewline
 & 500 &  & 0.37 & 0.42 & 0.44 &  & 0.01 & 0.08 & 0.11 &  & 0.06 & 0.12 & 0.14\tabularnewline\addlinespace
$x^{2}$ & 100 &  & 0.07 & 0.06 & 0.05 &  & 0.08 & 0.11 & 0.13 &  & 0.12 & 0.14 & 0.16\tabularnewline
$(\times0.02)$ & 200 &  & 0.13 & 0.13 & 0.12 &  & 0.04 & 0.09 & 0.10 &  & 0.08 & 0.11 & 0.12\tabularnewline
 & 500 &  & 0.26 & 0.31 & 0.32 &  & 0.00 & 0.00 & 0.00 &  & 0.01 & 0.00 & 0.00\tabularnewline\addlinespace
\bottomrule
\end{tabular}
\end{table}

\begin{table}[H]
\caption[caption]{MI processes: $x_{t}=(1-\kappa_{n}^{-1})x_{t-1}+\xi_{t}$; $\kappa_{n}=n^{\alpha_{\kappa}}$
\\ Size-adjusted power when $\alpha_{\kappa}\in\{0.25,0.50\}$; $\alpha=0.1$}
\label{tbl:power-mi-low-alpha}%
\begin{tabular}{lccccccccccccc}
\toprule 
\addlinespace
 & $\boldsymbol{n}$ &  & \multicolumn{3}{c}{\textbf{WP (2012)}} &  & \multicolumn{3}{c}{$\boldsymbol{p=17}$} &  & \multicolumn{3}{c}{$\boldsymbol{p=25}$}\tabularnewline\addlinespace
\multicolumn{2}{l}{$h=n^{b}$, $b=$} &  & -0.2 & -0.1 & -0.05 &  & -0.2 & -0.1 & -0.05 &  & -0.2 & -0.1 & -0.05\tabularnewline\addlinespace
\midrule
\addlinespace
\multicolumn{2}{l}{$\boldsymbol{\alpha_{\kappa}=0.25}$} &  & \multicolumn{3}{c}{\textbf{Size adj.\ power}} &  & \multicolumn{7}{c}{\textbf{Size adjusted, relative to WP}}\tabularnewline\addlinespace
$\varphi_{1}(x)$ & 100 &  & 0.06 & 0.03 & 0.02 &  & 0.11 & 0.09 & 0.07 &  & 0.13 & 0.11 & 0.08\tabularnewline
 & 200 &  & 0.15 & 0.13 & 0.09 &  & 0.17 & 0.17 & 0.15 &  & 0.20 & 0.20 & 0.17\tabularnewline
 & 500 &  & 0.50 & 0.54 & 0.50 &  & 0.21 & 0.22 & 0.23 &  & 0.24 & 0.24 & 0.25\tabularnewline\addlinespace
$\varphi_{1}(2x)$ & 100 &  & 0.09 & 0.04 & 0.02 &  & 0.13 & 0.10 & 0.07 &  & 0.16 & 0.12 & 0.08\tabularnewline
 & 200 &  & 0.22 & 0.16 & 0.10 &  & 0.19 & 0.17 & 0.13 &  & 0.23 & 0.20 & 0.16\tabularnewline
 & 500 &  & 0.61 & 0.58 & 0.48 &  & 0.19 & 0.18 & 0.19 &  & 0.22 & 0.21 & 0.22\tabularnewline\addlinespace
$|x|^{-2}\wedge1$ & 100 &  & 0.07 & 0.03 & 0.01 &  & 0.12 & 0.12 & 0.10 &  & 0.15 & 0.13 & 0.11\tabularnewline
$(\times0.5)$ & 200 &  & 0.17 & 0.10 & 0.05 &  & 0.20 & 0.22 & 0.21 &  & 0.23 & 0.24 & 0.24\tabularnewline
 & 500 &  & 0.50 & 0.39 & 0.26 &  & 0.15 & 0.15 & 0.16 &  & 0.17 & 0.16 & 0.17\tabularnewline\addlinespace
$|x|^{-1}\wedge1$ & 100 &  & 0.03 & 0.01 & 0.00 &  & 0.08 & 0.06 & 0.05 &  & 0.09 & 0.08 & 0.06\tabularnewline
$(\times0.5)$ & 200 &  & 0.06 & 0.03 & 0.02 &  & 0.12 & 0.12 & 0.11 &  & 0.15 & 0.15 & 0.12\tabularnewline
 & 500 &  & 0.15 & 0.12 & 0.09 &  & 0.20 & 0.24 & 0.24 &  & 0.23 & 0.26 & 0.27\tabularnewline\addlinespace
$|x|^{1.5}$ & 100 &  & 0.07 & 0.05 & 0.03 &  & 0.04 & 0.02 & 0.01 &  & 0.05 & 0.03 & 0.02\tabularnewline
$(\times0.02)$ & 200 &  & 0.22 & 0.20 & 0.16 &  & 0.04 & 0.03 & 0.02 &  & 0.05 & 0.04 & 0.03\tabularnewline
 & 500 &  & 0.72 & 0.76 & 0.74 &  & 0.04 & 0.04 & 0.03 &  & 0.06 & 0.05 & 0.04\tabularnewline\addlinespace
$x^{2}$ & 100 &  & 0.03 & 0.02 & 0.01 &  & 0.04 & 0.03 & 0.02 &  & 0.05 & 0.03 & 0.02\tabularnewline
$(\times0.02)$ & 200 &  & 0.09 & 0.07 & 0.05 &  & 0.05 & 0.04 & 0.03 &  & 0.06 & 0.05 & 0.04\tabularnewline
 & 500 &  & 0.33 & 0.35 & 0.33 &  & 0.06 & 0.07 & 0.07 &  & 0.09 & 0.09 & 0.08\tabularnewline\addlinespace
\midrule
\addlinespace
\multicolumn{2}{l}{$\boldsymbol{\alpha_{\kappa}=0.50}$} &  & \multicolumn{3}{c}{\textbf{Size adj.\ power}} &  & \multicolumn{7}{c}{\textbf{Size adjusted, relative to WP}}\tabularnewline\addlinespace
$\varphi_{1}(x)$ & 100 &  & 0.08 & 0.06 & 0.04 &  & 0.11 & 0.12 & 0.11 &  & 0.15 & 0.14 & 0.13\tabularnewline
 & 200 &  & 0.18 & 0.18 & 0.16 &  & 0.16 & 0.17 & 0.18 &  & 0.20 & 0.21 & 0.21\tabularnewline
 & 500 &  & 0.47 & 0.54 & 0.54 &  & 0.17 & 0.19 & 0.19 &  & 0.22 & 0.22 & 0.22\tabularnewline\addlinespace
$\varphi_{1}(2x)$ & 100 &  & 0.08 & 0.05 & 0.03 &  & 0.11 & 0.09 & 0.07 &  & 0.14 & 0.12 & 0.09\tabularnewline
 & 200 &  & 0.15 & 0.12 & 0.09 &  & 0.14 & 0.13 & 0.12 &  & 0.18 & 0.17 & 0.15\tabularnewline
 & 500 &  & 0.34 & 0.33 & 0.29 &  & 0.15 & 0.17 & 0.16 &  & 0.20 & 0.21 & 0.20\tabularnewline\addlinespace
$|x|^{-2}\wedge1$ & 100 &  & 0.07 & 0.04 & 0.02 &  & 0.15 & 0.15 & 0.15 &  & 0.19 & 0.18 & 0.17\tabularnewline
$(\times0.5)$ & 200 &  & 0.12 & 0.08 & 0.05 &  & 0.19 & 0.22 & 0.22 &  & 0.24 & 0.25 & 0.25\tabularnewline
 & 500 &  & 0.24 & 0.20 & 0.15 &  & 0.12 & 0.11 & 0.10 &  & 0.14 & 0.12 & 0.11\tabularnewline\addlinespace
$|x|^{-1}\wedge1$ & 100 &  & 0.05 & 0.03 & 0.01 &  & 0.10 & 0.11 & 0.10 &  & 0.13 & 0.13 & 0.11\tabularnewline
$(\times0.5)$ & 200 &  & 0.08 & 0.06 & 0.05 &  & 0.15 & 0.17 & 0.17 &  & 0.19 & 0.21 & 0.21\tabularnewline
 & 500 &  & 0.17 & 0.18 & 0.17 &  & 0.17 & 0.18 & 0.19 &  & 0.21 & 0.21 & 0.21\tabularnewline\addlinespace
$|x|^{1.5}$ & 100 &  & 0.12 & 0.10 & 0.09 &  & 0.03 & 0.03 & 0.03 &  & 0.05 & 0.04 & 0.04\tabularnewline
$(\times0.02)$ & 200 &  & 0.31 & 0.33 & 0.31 &  & 0.04 & 0.06 & 0.06 &  & 0.07 & 0.08 & 0.08\tabularnewline
 & 500 &  & 0.76 & 0.83 & 0.85 &  & 0.06 & 0.11 & 0.14 &  & 0.10 & 0.15 & 0.17\tabularnewline\addlinespace
$x^{2}$ & 100 &  & 0.07 & 0.05 & 0.04 &  & 0.05 & 0.06 & 0.05 &  & 0.07 & 0.07 & 0.07\tabularnewline
$(\times0.02)$ & 200 &  & 0.16 & 0.17 & 0.16 &  & 0.09 & 0.12 & 0.13 &  & 0.13 & 0.15 & 0.17\tabularnewline
 & 500 &  & 0.50 & 0.60 & 0.61 &  & 0.04 & 0.07 & 0.07 &  & 0.08 & 0.09 & 0.09\tabularnewline\addlinespace
\bottomrule
\end{tabular}
\end{table}

\begin{table}[H]
\caption[caption]{MI processes: $x_{t}=(1-\kappa_{n}^{-1})x_{t-1}+\xi_{t}$; $\kappa_{n}=n^{\alpha_{\kappa}}$
\\ Size-adjusted power when $\alpha_{\kappa}\in\{0.75,1.00\}$; $\alpha=0.1$}
\label{tbl:power-mi-high-alpha}%
\begin{tabular}{lccccccccccccc}
\toprule 
\addlinespace
 & $\boldsymbol{n}$ &  & \multicolumn{3}{c}{\textbf{WP (2012)}} &  & \multicolumn{3}{c}{$\boldsymbol{p=17}$} &  & \multicolumn{3}{c}{$\boldsymbol{p=25}$}\tabularnewline\addlinespace
\multicolumn{2}{l}{$h=n^{b}$, $b=$} &  & -0.2 & -0.1 & -0.05 &  & -0.2 & -0.1 & -0.05 &  & -0.2 & -0.1 & -0.05\tabularnewline\addlinespace
\midrule
\addlinespace
\multicolumn{2}{l}{$\boldsymbol{\alpha_{\kappa}=0.75}$} &  & \multicolumn{3}{c}{\textbf{Size adj.\ power}} &  & \multicolumn{7}{c}{\textbf{Size adjusted, relative to WP}}\tabularnewline\addlinespace
$\varphi_{1}(x)$ & 100 &  & 0.08 & 0.06 & 0.05 &  & 0.08 & 0.09 & 0.09 &  & 0.11 & 0.12 & 0.11\tabularnewline
 & 200 &  & 0.13 & 0.13 & 0.12 &  & 0.09 & 0.12 & 0.13 &  & 0.14 & 0.16 & 0.16\tabularnewline
 & 500 &  & 0.25 & 0.29 & 0.29 &  & 0.09 & 0.15 & 0.17 &  & 0.14 & 0.19 & 0.21\tabularnewline\addlinespace
$\varphi_{1}(2x)$ & 100 &  & 0.07 & 0.05 & 0.03 &  & 0.06 & 0.07 & 0.06 &  & 0.10 & 0.09 & 0.08\tabularnewline
 & 200 &  & 0.11 & 0.09 & 0.07 &  & 0.06 & 0.08 & 0.08 &  & 0.10 & 0.11 & 0.11\tabularnewline
 & 500 &  & 0.17 & 0.17 & 0.15 &  & 0.06 & 0.09 & 0.09 &  & 0.10 & 0.13 & 0.14\tabularnewline\addlinespace
$|x|^{-2}\wedge1$ & 100 &  & 0.07 & 0.04 & 0.03 &  & 0.11 & 0.13 & 0.13 &  & 0.15 & 0.16 & 0.15\tabularnewline
$(\times0.5)$ & 200 &  & 0.09 & 0.07 & 0.05 &  & 0.13 & 0.17 & 0.18 &  & 0.18 & 0.20 & 0.21\tabularnewline
 & 500 &  & 0.13 & 0.11 & 0.09 &  & 0.09 & 0.14 & 0.15 &  & 0.15 & 0.18 & 0.18\tabularnewline\addlinespace
$|x|^{-1}\wedge1$ & 100 &  & 0.06 & 0.03 & 0.02 &  & 0.08 & 0.09 & 0.10 &  & 0.11 & 0.12 & 0.11\tabularnewline
$(\times0.5)$ & 200 &  & 0.09 & 0.06 & 0.05 &  & 0.10 & 0.13 & 0.14 &  & 0.14 & 0.16 & 0.17\tabularnewline
 & 500 &  & 0.13 & 0.14 & 0.14 &  & 0.10 & 0.16 & 0.17 &  & 0.15 & 0.20 & 0.21\tabularnewline\addlinespace
$|x|^{1.5}$ & 100 &  & 0.11 & 0.09 & 0.08 &  & 0.03 & 0.04 & 0.05 &  & 0.06 & 0.06 & 0.06\tabularnewline
$(\times0.02)$ & 200 &  & 0.21 & 0.23 & 0.22 &  & 0.04 & 0.07 & 0.08 &  & 0.06 & 0.10 & 0.11\tabularnewline
 & 500 &  & 0.45 & 0.52 & 0.54 &  & 0.00 & 0.07 & 0.09 &  & 0.06 & 0.11 & 0.12\tabularnewline\addlinespace
$x^{2}$ & 100 &  & 0.08 & 0.06 & 0.05 &  & 0.08 & 0.11 & 0.11 &  & 0.11 & 0.13 & 0.13\tabularnewline
$(\times0.02)$ & 200 &  & 0.14 & 0.15 & 0.14 &  & 0.06 & 0.11 & 0.12 &  & 0.10 & 0.14 & 0.14\tabularnewline
 & 500 &  & 0.33 & 0.39 & 0.42 &  & 0.00 & 0.00 & 0.00 &  & 0.00 & 0.00 & 0.00\tabularnewline\addlinespace
\midrule
\addlinespace
\multicolumn{2}{l}{$\boldsymbol{\alpha_{\kappa}=1.00}$} &  & \multicolumn{3}{c}{\textbf{Size adj.\ power}} &  & \multicolumn{7}{c}{\textbf{Size adjusted, relative to WP}}\tabularnewline\addlinespace
$\varphi_{1}(x)$ & 100 &  & 0.08 & 0.06 & 0.05 &  & 0.06 & 0.08 & 0.07 &  & 0.09 & 0.10 & 0.09\tabularnewline
 & 200 &  & 0.11 & 0.10 & 0.09 &  & 0.06 & 0.10 & 0.11 &  & 0.10 & 0.13 & 0.14\tabularnewline
 & 500 &  & 0.17 & 0.19 & 0.19 &  & 0.04 & 0.09 & 0.10 &  & 0.08 & 0.13 & 0.14\tabularnewline\addlinespace
$\varphi_{1}(2x)$ & 100 &  & 0.07 & 0.05 & 0.04 &  & 0.05 & 0.05 & 0.06 &  & 0.08 & 0.08 & 0.07\tabularnewline
 & 200 &  & 0.10 & 0.08 & 0.06 &  & 0.04 & 0.06 & 0.06 &  & 0.07 & 0.09 & 0.09\tabularnewline
 & 500 &  & 0.13 & 0.12 & 0.11 &  & 0.03 & 0.05 & 0.06 &  & 0.05 & 0.08 & 0.10\tabularnewline\addlinespace
$|x|^{-2}\wedge1$ & 100 &  & 0.07 & 0.04 & 0.03 &  & 0.09 & 0.11 & 0.11 &  & 0.13 & 0.14 & 0.14\tabularnewline
$(\times0.5)$ & 200 &  & 0.09 & 0.06 & 0.05 &  & 0.10 & 0.13 & 0.14 &  & 0.14 & 0.17 & 0.17\tabularnewline
 & 500 &  & 0.11 & 0.09 & 0.08 &  & 0.05 & 0.10 & 0.13 &  & 0.10 & 0.14 & 0.16\tabularnewline\addlinespace
$|x|^{-1}\wedge1$ & 100 &  & 0.06 & 0.04 & 0.03 &  & 0.06 & 0.08 & 0.08 &  & 0.10 & 0.10 & 0.10\tabularnewline
$(\times0.5)$ & 200 &  & 0.09 & 0.07 & 0.06 &  & 0.07 & 0.11 & 0.12 &  & 0.11 & 0.14 & 0.15\tabularnewline
 & 500 &  & 0.12 & 0.12 & 0.11 &  & 0.05 & 0.10 & 0.13 &  & 0.09 & 0.14 & 0.17\tabularnewline\addlinespace
$|x|^{1.5}$ & 100 &  & 0.09 & 0.08 & 0.07 &  & 0.03 & 0.04 & 0.05 &  & 0.05 & 0.06 & 0.06\tabularnewline
$(\times0.02)$ & 200 &  & 0.16 & 0.17 & 0.16 &  & 0.04 & 0.07 & 0.08 &  & 0.07 & 0.10 & 0.11\tabularnewline
 & 500 &  & 0.29 & 0.33 & 0.34 &  & -0.04 & 0.03 & 0.05 &  & 0.01 & 0.06 & 0.07\tabularnewline\addlinespace
$x^{2}$ & 100 &  & 0.07 & 0.06 & 0.05 &  & 0.08 & 0.11 & 0.13 &  & 0.12 & 0.14 & 0.15\tabularnewline
$(\times0.02)$ & 200 &  & 0.11 & 0.12 & 0.11 &  & 0.04 & 0.06 & 0.07 &  & 0.07 & 0.08 & 0.08\tabularnewline
 & 500 &  & 0.22 & 0.26 & 0.27 &  & 0.00 & 0.00 & 0.00 &  & 0.00 & 0.00 & 0.00\tabularnewline\addlinespace
\bottomrule
\end{tabular}
\end{table}

\clearpage{}

\end{document}